\documentclass[10pt,a4paper]{article}

\usepackage[utf8]{inputenc}

\usepackage{rsfso}
\usepackage{mathtools,amsfonts,amssymb,mathrsfs}
\usepackage{enumerate}
\usepackage{lmodern}
\usepackage[T1]{fontenc}
\usepackage{verbatim}
\usepackage{textcomp}
\usepackage[a4paper,vmargin={3.5cm,3.5cm},hmargin={2.5cm,2.5cm}]{geometry}
\usepackage[font=sf, labelfont={sf,bf}, margin=1cm]{caption}

\usepackage[pdftex]{color,graphicx}

\usepackage{stmaryrd}

\usepackage{xspace}

\usepackage{dsfont}

\usepackage{enumitem}

\usepackage{amsthm}

\usepackage{color}

\usepackage[english]{babel}
\usepackage[pdftex]{hyperref}

\renewcommand{\l}{\left}
\renewcommand{\r}{\right}
\newcommand{\ind}[1]{\mathds{1}_{\l\{#1\r\}}}

\newcommand{\eqdef}{\overset{\text{def}}{=}}

\newcommand{\floor}[1]{{\left\lfloor #1 \right\rfloor}}
\newcommand{\ceil}[1]{{\left\lceil #1 \right\rceil}}
\newcommand{\usim}[2]{\underset{#1 \to #2}{\sim}}

\newcommand{\ulim}[2]{\underset{#1 \to #2}{\longrightarrow}}

\newcommand{\ulimd}[2]{\overset{\textrm{(d)}}{\ulim{#1}{#2}}}

\newcommand{\ulimas}[2]{\overset{\textrm{\as}}{\ulim{#1}{#2}}}
\newcommand{\pfrac}[2]{{\left(\frac{#1}{#2}\right)}}

\newcommand\la{\longrightarrow}		
\newcommand{\figref}[1]{figure \ref{#1}}

\newcommand{\as}{\text{a.s.}\xspace}
\newcommand{\whp}{\text{w.h.p.}\xspace}
\newcommand{\wrt}{\text{w.r.t.}\xspace}
\newcommand{\iid}{\text{i.i.d.}\xspace}

\newcommand{\fpp}{\text{fpp}}

\newcommand{\dgr}{\mathrm{d}_{\mathrm{gr}}}

\newcommand{\dfpp}{\mathrm{d}_{\mathrm{fpp}}}

\newcommand{\Htr}{\HH^\text{tr}}
\newcommand{\Hull}{B^\bullet}

\newcommand{\UIPQ}{\textsc{uipq}\xspace}

\newcommand{\UIPM}{\textsc{uipm}\xspace}

\newcommand{\LHPQ}{\textsc{lhpq}\xspace}

\DeclareMathOperator{\Epaiss}{Thickness}

\newcommand{\BGW}{Bienaymé-Galton-Watson\xspace}

\newcommand{\resp}{resp.\xspace}
\newcommand{\st}{{s.t. }}

\newcommand\wt{\widetilde}
\newcommand\wh{\widehat}

\newcommand\bc{\textbf{c}}

\newcommand\be{\textbf{e}}

\newcommand\ur{\textrm{r}}

\newcommand\E{\mathbb{E}}
\newcommand\F{\mathbb{F}}

\newcommand\N{\mathbb{N}}
\renewcommand\P{\mathbb{P}}		
\newcommand\Q{\mathbb{Q}}
\newcommand\R{\mathbb{R}}

\newcommand\Z{\mathbb{Z}}

\newcommand\CC{\mathcal{C}}

\newcommand\EE{\mathcal{E}}
\newcommand\FF{\mathcal{F}}
\newcommand\GG{\mathcal{G}}
\newcommand\HH{\mathcal{H}}

\newcommand\LL{\mathcal{L}}

\newcommand\PP{\mathcal{P}}

\newcommand\RR{\mathcal{R}}
\renewcommand\SS{\mathcal{S}}		
\newcommand\TT{\mathcal{T}}

\newcommand\XX{\mathcal{X}}

\newcommand{\va}{\alpha}
\newcommand{\vb}{\beta}
\newcommand{\vg}{\gamma}
\newcommand{\vd}{\delta}
\newcommand{\ve}{\varepsilon}

\newcommand{\vt}{\theta}

\newcommand{\vk}{\kappa}

\newcommand{\vs}{\sigma}
\newcommand{\vp}{\varphi}
\newcommand{\vo}{\omega}
\newcommand{\vD}{\Delta}
\newcommand{\vG}{\Gamma}
\newcommand{\vT}{\Theta}

\newtheorem{theorem}{Theorem}

\newtheorem{lemma}[theorem]{Lemma}
\newtheorem{prop}[theorem]{Proposition}
\newtheorem{proposition}[theorem]{Proposition}

\newtheorem{corollary}[theorem]{Corollary}

\newcommand{\propref}{Proposition \ref}
\newcommand{\thref}{Theorem \ref}
\newcommand{\corref}{Corollary \ref}
\newcommand{\lemref}{Lemma \ref}

\newcommand{\secref}{Section \ref}
\renewcommand{\figref}{Figure \ref}

\newenvironment{romenum}{
\begin{enumerate}[label=(\roman*)]
}{
\end{enumerate}
}

\newcommand{\van}{\va_{(n)}}
\newcommand{\bmap}{{\text{m}_\infty}}

\newcommand{\Lhm}{\Ti{\LL}}   
\newcommand{\lmg}{left-most geodesic\xspace}
\newcommand{\lmgs}{left-most geodesics\xspace}

\newcommand{\Lmgs}{Left-most geodesics\xspace}

\renewcommand{\ur}{u}
\newcommand{\DP}{\operatorname{DP}}

\newcommand{\Ti}[1]{{\TT\l( #1 \r)}} 
\newcommand{\Qtr}{\Q^\text{tr}}

\newcommand{\CVS}{CVS correspondence\xspace}

\newcommand{\Tuttebij}{Tutte's bijection\xspace}

\author{\textsc{Thomas Lehéricy}\footnote{Université Paris-Saclay}}
\title{First-passage percolation in random planar maps\\
and Tutte's bijection}

\begin{document}

\maketitle

\begin{abstract}
We consider large random planar maps and study the first-passage percolation distance 
obtained by assigning independent identically distributed lengths to the edges. We consider the cases
of quadrangulations and of general planar maps. In both cases, 
the first-passage percolation distance is shown to behave in large scales
like a constant times the usual graph distance. We apply our method to 
the metric properties of the classical Tutte bijection between 
quadrangulations with $n$ faces and general planar maps with 
$n$ edges. We prove that the respective graph distances on the quadrangulation and
on the associated general planar map are in
large scales equivalent when $n\to \infty$. 
\end{abstract}

\tableofcontents

\section{Introduction}

A planar map is a finite planar graph embedded in the sphere and considered up to orientation-preserving homeomorphisms. 
In this work, we only consider rooted planar maps, meaning that we distinguish an oriented edge called the root edge, whose
origin is the root vertex.
There exist many different models of random maps depending on the conditions one imposes on the degrees of faces, the existence or non-existence of multiple edges and loops, etc. In the following, we always allow loops and multiple edges.
A particular case that will play a central role in this article is the case of quadrangulations, where all faces have degree $4$. 
 For any map $M$, we denote its vertex set by $V(M)$ and the graph distance on the map $M$ by $\dgr^M$. The root vertex is
 usually denoted by $\rho$.

Several recent developments (see in particular \cite{ abraham2016rescaled,berry2013scaling,bettinelli2014scaling,leGall2013uniqueness,miermont2013brownian}) have established that, for a wide range of models of random maps, the vertex set viewed as a metric space for the (suitably rescaled) graph distance, converges in distribution when the size of the map tends to infinity towards a random compact metric space called the Brownian map. This convergence holds in the sense of the Gromov-Hausdorff convergence for compact metric spaces. The convergence to the Brownian map gives a unified approach to asymptotic properties of different models of random planar maps.

A natural question is to ask what can be said when the graph distance is replaced by other choices of distances on the vertex set $V(M)$. A simple way to get other distances
is to assign positive weights (or lengths) to the edges: The distance between two vertices will then be the minimal total weight of a path connecting these two vertices.
When the weights of the different edges are chosen to be independent and identically distributed given the planar map $M$ in consideration, this leads to
the so-called first-passage percolation distance, which we denote here by $\dfpp^M$. Of course, when weights are all equal to $1$ we
recover the graph distance. The recent paper \cite{fpp} has investigated the asymptotic properties of the first-passage percolation distance 
in large triangulations. Roughly speaking, the main results of \cite{fpp} show that, in large scales, the first-passage percolation distance
behaves like a constant times the graph distance. The relevant constant is found via a subadditivity argument and cannot be computed in general. Interestingly, this
behavior for random planar maps is quite different from the one observed in deterministic lattices such as $\Z^d$, where the 
first-passage percolation distance is not believed to be asymptotically proportional to the graph distance (nor to the
Euclidean distance). 

One of the main goals of the present work is to show that results similar to those of \cite{fpp} hold both for quadrangulations 
and for general planar maps. Recall that the diameter of a typical quadrangulation with $n$ faces, or of a typical planar map
with $n$ edges is known to be of order $n^{1/4}$. 

\begin{theorem}
\label{Th1}
For every integer $n\geq 1$, let $Q_n$
be uniformly distributed over the set of all rooted quadrangulations with $n$ faces, and let $M_n$
be uniformly distributed over the set of all rooted planar maps with $n$ edges. Define the 
first-passage percolation distances $\dfpp^{Q_n}$ and $\dfpp^{M_n}$ by assigning independent and identically distributed
weights to the edges of $Q_n$ and $M_n$. Assume that the common distribution of the weights is supported 
on a compact subset of $(0,\infty)$. 
Then there exist two positive  constants $\bc$ and $\bc'$ such that 
\begin{equation*}
n^{-1/4} \sup_{x,y \in V(Q_n)} \l| \dfpp^{Q_n}(x,y) - \bc\,\dgr^{Q_n}(x,y) \r| \ulim{n}{\infty} 0
\end{equation*}
and 
\begin{equation*}
n^{-1/4} \sup_{x,y \in V(M_n)} \l| \dfpp^{M_n}(x,y) - \bc' \dgr^{M_n}(x,y) \r| \ulim{n}{\infty} 0
\end{equation*}
where both convergences hold in probability. 
\end{theorem}

As an immediate consequence of the theorem, we get that the convergence to the Brownian map 
still holds for both models in consideration if the graph distance is replaced by the
first-passage percolation distance. More precisely, under the assumptions of Theorem \ref{Th1}
and with the same constants $\bc$ and $\bc'$, we have 
\begin{equation}
\label{Eq_cv_quad}
\l( V(Q_n), \pfrac{9}{8n}^{1/4} \dfpp^{Q_n} \r) \ulimd n \infty (\bmap, \bc\,D^*) ,
\end{equation}
and
\begin{equation}
\label{Eq_cv_map}
\l( V(M_n), \pfrac{9}{8n}^{1/4} \dfpp^{M_n} \r) \ulimd n \infty (\bmap, \bc' D^*) ,
\end{equation}
where $(\bmap, D^*)$ is the Brownian map, and both convergences hold in distribution in the Gromov-Hausdorff space. 
Indeed, this follows from Theorem \ref{Th1} and the known convergences for the graph distance 
which have been established in \cite{leGall2013uniqueness,miermont2013brownian} for $Q_n$
and in \cite{bettinelli2014scaling} for $M_n$. It is remarkable that the same constant $(9/8)^{1/4}$
appears in both \eqref{Eq_cv_quad} and \eqref{Eq_cv_map}. This will be better understood in the next theorem.

Another major goal of the present article is to have a better understanding of the metric properties of 
Tutte's bijection. Recall that Tutte's bijection, also called the {\em trivial bijection}, gives for every $n\geq 1$ a one-to-one correspondence 
between the set of all rooted quadrangulations with $n$ faces and the set of all rooted planar maps with $n$ edges.
The definition of this correspondence should be clear from Figure \ref{Fig_Tutte_definition}. The following theorem can be interpreted by
saying that Tutte's transformation acting on a large quadrangulation is nearly isometric with respect to the 
graph distances. 

\begin{theorem}
\label{Th2}
Let $Q_n$
be uniformly distributed over the set of all rooted quadrangulations with $n$ faces and let $M_n$ 
be the image of $Q_n$ under \Tuttebij, so that
$M_n$ is uniformly distributed over the set of all rooted planar maps with $n$ edges and $V(M_n)$ is identified to a subset of $V(Q_n)$. 
Then,
\begin{equation*}
n^{-1/4} \sup_{x,y \in V(M_n)} \l| \dgr^{Q_n}(x,y) - \dgr^{M_n}(x,y) \r| \ulim{n}{\infty} 0
\end{equation*}
where the convergence holds in probability. 
\end{theorem}

We can combine Theorems \ref{Th1} and \ref{Th2} to get that, if $Q_n$ and $M_n$ are linked by Tutte's bijection, 
the convergences in distribution \eqref{Eq_cv_quad} and \eqref{Eq_cv_map} hold jointly, with the {\em same}
space $(\bmap,D^*)$ in the limit. This is reminiscent of Theorem 1.1 in \cite{bettinelli2014scaling}, which gives
a similar joint convergence, but in the case where $Q_n$ and $M_n$ are linked by a different bijection
(the Ambj\o rn-Budd bijection) which is more faithful to graph distances.

\begin{figure}
\begin{center}
\includegraphics[width=\textwidth]{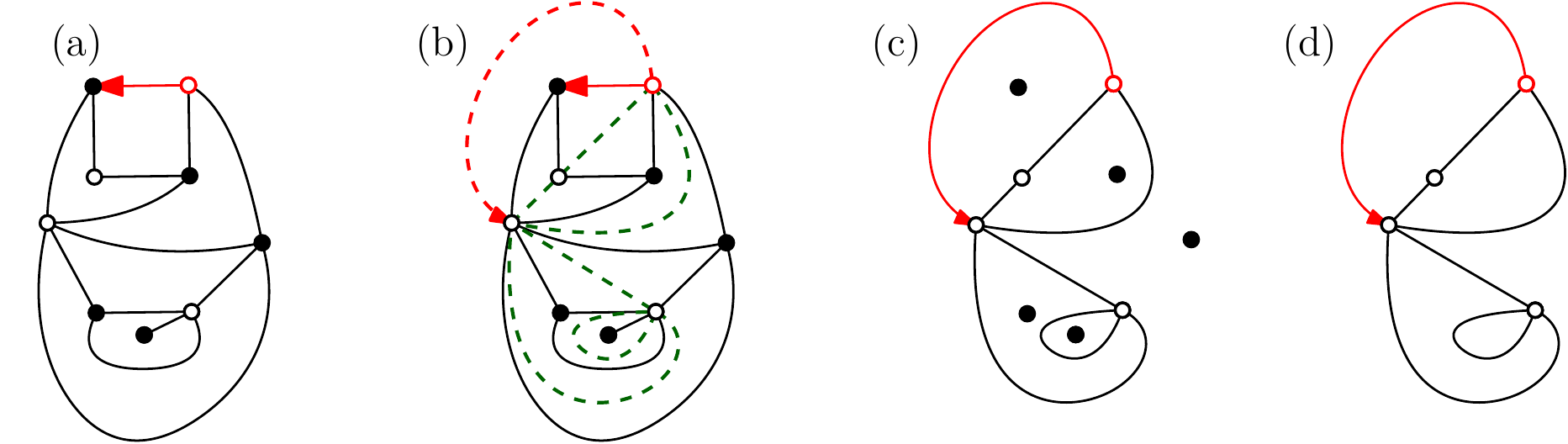}
\end{center}
\caption{Illustration of \Tuttebij. (a) On the left, a quadrangulation with $7$ faces. Color the tail of its root vertex in white, and every other vertex in black and white so that adjacent vertices have a different color. 
(b) In every face of the quadrangulation, add a diagonal between its white corners. 
(c) Erase the edges of the quadrangulation. 
(d) The black vertices now have degree zero and are also erased. We obtain a map with $7$ edges, which is rooted at the edge corresponding to
the diagonal drawn in the face to the right of the root edge of the quadrangulation, oriented so that the root vertex remains the same.}
\label{Fig_Tutte_definition}
\end{figure}

We can also give versions of the preceding results for the infinite random lattices that arise 
as local limits (in the Benjamini-Schramm sense) of large quadrangulations or general
planar maps. We write $Q_\infty$ for the \UIPQ or uniform infinite planar quadrangulation, and
$M_\infty$ for the \UIPM or uniform infinite planar map. As was observed by M\'enard and Nolin \cite{menard2014percolation},
the \UIPM can be obtained by applying (a generalized version of) Tutte's correspondence 
to the \UIPQ.

\begin{theorem}
\label{Th3}
Let $\dfpp^{Q_\infty}$ and $\dfpp^{M_\infty}$ be the first-passage percolation distances defined
on the vertex sets $V(Q_\infty)$ and $V(M_\infty)$, respectively, by assigning
edge weights satisfying the same assumptions as in Theorem \ref{Th1}. Write $\rho_{Q_\infty}$ and 
$\rho_{M_\infty}$ for the respective root vertices of $Q_\infty$ and $M_\infty$. Then, for every $\ve>0$,
$$\lim_{r \to \infty} \P\l( \sup_{x,y \in V(Q_\infty), \ \dgr^{Q_\infty}(\rho_{Q_\infty},x)\vee \dgr^{Q_\infty}(\rho_{Q_\infty},y)\leq r} |\dfpp^{Q_\infty}(x,y) - \bc\, \dgr^{Q_\infty}(x,y) | > \ve r \r) = 0 ,$$
and similarly,
$$\lim_{r \to \infty} \P\l( \sup_{x,y \in V(M_\infty), \ \dgr^{M_\infty}(\rho_{M_\infty},x)\vee \dgr^{M_\infty}(\rho_{M_\infty},y)\leq r} |\dfpp^{M_\infty}(x,y) - \bc' \dgr^{M_\infty}(x,y) | > \ve r \r) = 0 ,$$
with the same constants $\bc$ and $\bc'$ as in Theorem \ref{Th1}. Moreover, if 
the \UIPQ $Q_\infty$ and the \UIPM $M_\infty$ are linked by Tutte's correspondence, we have also
$$\lim_{r \to \infty} \P\l( \sup_{x,y \in V(M_\infty), \ \dgr^{M_\infty}(\rho_{M_\infty},x)\vee \dgr^{M_\infty}(\rho_{M_\infty},y)\leq r} |\dgr^{M_\infty}(x,y) - \dgr^{Q_\infty}(x,y) | > \ve r \r) = 0.$$
\end{theorem}

A consequence of the first assertions of the theorem is the fact that balls (centered at the root vertex) for the first-passage percolation distance are asymptotically the same as 
for the graph distance, both in $Q_\infty$ and in $M_\infty$.  More precisely, in $Q_\infty$ for definiteness, the (metric) ball of radius $r$ for the first-passage percolation distance will be contained in the graph distance ball of radius $(1+\ve)r/\bc$ and will contain the graph distance ball of radius $(1-\ve)r/\bc$, with high probability when $r$ is large. This is in sharp
contrast with the behavior expected for deterministic lattices.

In the same way as in \cite{fpp}, our proofs rely on a  ``skeleton decomposition'' which in the case of
quadrangulations appeared first in the work of Krikun \cite{Krikun2008local}, and has been used 
extensively in \cite{gall2017separating}. Recall that, in the \UIPQ $Q_\infty$, the hull of radius $r$ is defined as the complement of the infinite
connected component of the complement of the ball of radius $r$ centered at the root vertex (informally, the hull
is obtained by filling in the finite holes in the ball of radius $r$, see Section \ref{Sec_Preliminaries}
for a more precise definition). The skeleton decomposition provides a detailed description of the joint
distribution of layers of the \UIPQ, where, roughly speaking, a layer corresponds to the part of the map between the
boundary of the hull of radius $r$ and the boundary of the hull of radius $r+1$. This description allows us
to compare the \UIPQ near the boundary of a hull with another infinite model which we call the
\LHPQ for lower half-plane quadrangulation (Section \ref{Sec_LHPQ}). The point is then that a subadditive ergodic theorem can be applied to evaluate 
 first-passage percolation distances in the \LHPQ. Quite remarkably, this method carries over to 
 the study of the  first-passage percolation distance in the general planar maps
 that are obtained from quadrangulations by Tutte's corrrespondence, with the minor difference that we must restrict
 to hulls of even radius in the quadrangulation. 

Even though the idea of using the skeleton decomposition already appeared in \cite{fpp} in the setting
of triangulations, there are important differences between the present work and \cite{fpp}, and
our proofs are by no means straightforward extensions of those of \cite{fpp}. In particular, a very
important ingredient of our method involves bounds on graph distances along the boundary in the \LHPQ.
To derive these bounds we use a completely different approach from that developed in \cite{fpp}
for the model called the lower half-plane triangulation. Our approach, which relies
on certain ideas of \cite{curien2017geometric}, is simpler and avoids the heavy combinatorial analysis of \cite{fpp}. Similarly,
the application of the subadditive ergodic theorem gives information about the first-passage percolation distance
between points of the boundary of a hull of the \UIPQ and the root vertex (Proposition \ref{Prop_P20 distances root - boundary of the hull, UIPQ} below)
but a key step is then to derive information about the distance between a typical point and the root vertex in the {\rm finite} quadrangulation $Q_n$
(Proposition \ref{Prop_P21 distances root - uniform point, Qn}): For this purpose, the lack of certain explicit combinatorial expressions did not allow
us to use the same approach as in \cite{fpp}, and we had to develop a different method based on a coupling between $Q_n$
and the \UIPQ. Finally, the treatment of a general planar map $M_n$ given as the image of the quadrangulation $Q_n$ under Tutte's bijection also
required a number of new tools, in particular because the graph distance on $M_n$ is not easily controlled 
in terms of the graph distance on $Q_n$. 

The paper is organized as follows. Section \ref{Sec_Preliminaries} gives a number of preliminaries about planar maps and the skeleton 
decomposition. We introduce the notion of a quadrangulation of the cylinder, which already played 
an important role in \cite{gall2017separating}, and we define the so-called truncated hulls, which are variants of
the (standard) hulls considered in earlier work. Section \ref{Sec_LHPQ}
discusses the lower half-plane quadrangulation. In particular, we derive the important bounds
controlling distances along the boundary (Proposition \ref{Prop_P15}). Section \ref{Sec-tech}
gives two technical propositions. The first one (Proposition \ref{Prop_P5}) provides bounds for the
distribution of the skeleton of a (truncated) hull of the \UIPQ in terms of the skeleton associated 
with the \LHPQ. This is the key to transfer results obtained in the \LHPQ (by subadditive arguments)
to the \UIPQ. Section \ref{Sec-main-quad} derives our main results about first-passage percolation distances
in quadrangulations. We start by proving Proposition \ref{Prop_P19 thin annuli}, which estimates the
$\dfpp^{Q_\infty}$-distance between a vertex of the boundary of the (truncated) hull of radius $r$
and the boundary of the hull of radius $r-\floor{\eta r}$, for $\eta>0$ small. This key proposition  is then used to
get Proposition \ref{Prop_P20 distances root - boundary of the hull, UIPQ} concerning the distance between a 
point of the boundary of a hull and the root vertex. Then the hard work is to prove Proposition \ref{Prop_P21 distances root - uniform point, Qn}
controlling the distance between a uniformly distributed vertex of $Q_n$ and the root vertex. From this proposition, it
is not too hard to derive Theorem \ref{Th_T1 controle distances fpp et gr dans quad finies}, which gives 
the part of Theorem \ref{Th1} dealing with quadrangulations. Section \ref{technical-general} contains certain
technical results concerning graph distances in the general maps associated with quadrangulations via
Tutte's correspondence. We introduce the so-called downward paths, which are closely related to the
skeleton decomposition of the associated quadrangulation, and we derive important bounds on the 
length of these paths (Lemma \ref{Lemma_control_length_DP}). Finally, Section \ref{Sec_Main results for generic maps}
is devoted to the proof of the results concerning general maps. In particular, Theorem \ref{dist-fpp-gr}
shows that, if $M_n$ and $Q_n$ are linked by Tutte's bijection, the first-passage percolation distance in $M_n$
is asymptotically proportional to the graph distance in $Q_n$. In the case where weights are equal to $1$, the proportionality constant
has to be equal to $1$ (because of the known results about scaling limits of $Q_n$ and $M_n$), which gives Theorem \ref{Th2} and then
the part of Theorem \ref{Th1} concerning general maps. Several proofs in this section are very similar to the proofs of Section \ref{Sec-main-quad}. For this reason,
we have only sketched certain arguments, but we emphasize the places where new ingredients are required.

\section{Preliminaries}
\label{Sec_Preliminaries}

A (finite) \emph{planar map} is a planar graph embedded in the sphere and seen up to orientation-preserving homeomorphisms. We allow multiple edges and loops. 
 Since we consider only planar maps we often say \emph{map} instead of planar map. 
 If $M$ is a map, we denote the set of vertices, edges, and faces of $M$ by $V(M),E(M), F(M)$  respectively. 
We write $\dgr^M$ for the \emph{graph distance} on $V(M)$. Given $\iid$ random weights $(\vo_e)_{e \in E(M)}$ assigned to the edges of $M$, we also define the 
associated \emph{first-passage percolation distance} $\dfpp^M$ as follows. 
We define the weight of a path $\vg$ as the sum of the weights of its edges, and the first-passage percolation distance $\dfpp^M(x,y)$ between two vertices $x$ and $y$ of $M$ is the infimum of the weights of paths starting at $x$ and ending at $y$. 
 Note that, if $\vo_e = 1$ for every edge $e$, we recover the graph distance on $V(M)$.

A \emph{rooted map} is a map with a distinguished oriented edge called the \emph{root edge}. The tail of the root edge is called the \emph{root vertex} and is usually denoted by $\rho$. The
face lying to the right of the root edge is called the root face.
We say that a rooted map is \emph{pointed} if in addition it has a distinguished vertex $\partial$.


\paragraph{Models of quadrangulations.}

A \emph{quadrangulation} is a rooted map whose faces all have degree $4$. Quadrangulations are bipartite maps, so we may and will color the vertices of a quadrangulation in black and white so that two adjacent vertices have different colors and the root vertex is white.

A \emph{truncated quadrangulation} is a rooted map such that
\begin{itemize}
\item the root face has a simple boundary and an arbitrary even degree,
\item every edge incident to the root face is also incident to another face which has degree 3, and these triangular faces are distinct,
\item all the other faces 
have degree 4.
\end{itemize}
The root face is also called the external face, and faces other than the external face are called inner faces. The length of the external boundary (the boundary of the external face)
is called the perimeter of the truncated quadrangulation.

We will also consider infinite (rooted but not pointed) quadrangulations for which we assume --- except in the case of the \LHPQ discussed in
Section \ref{Sec_LHPQ} --- that they are embedded in the plane in such a way  that 
 all faces are bounded subsets of the plane, and furthermore every compact subset of the plane intersects only finitely many faces (and again infinite quadrangulations are viewed up to orientation preserving homeomorphisms). These properties hold a.s. for the \UIPQ.

\paragraph{Hulls and truncated hulls.}

Let $Q$ be a (finite or infinite) quadrangulation with root vertex $\rho$. For every integer $r\geq 1$, we denote the \emph{ball of radius $r$ in $Q$} by $B_Q(r)$. This ball is the map obtained by taking the union of all faces that are incident to a vertex at graph distance at most $r-1$ from $\rho$. Suppose in addition that $Q$ is finite and pointed, and let $R=\dgr^Q(\rho,\partial)$ be the graph distance between the root vertex and 
the distinguished vertex. Then for every integer $1 \leq r \leq R-2$, the \emph{standard hull of radius $r$ of $Q$}, denoted by $\Hull_{Q}(r)$, is the union of $B_{Q}(r)$ and of the connected components of its complement that do not contain $\partial$. If $Q$ is an infinite quadrangulation, then, for every $r\geq 1$,  the standard hull of radius $r$ of $Q$ is defined as the union of $B_{Q}(r)$ and the finite connected components of its complement, and is also denoted by $\Hull_{Q}(r)$. In both the finite and the infinite case, the 
standard hull $\Hull_{Q}(r)$ is a quadrangulation with a simple  boundary (meaning that all faces are quadrangles, except for one distinguished face, which
has a simple boundary). If $r>1$ is not an integer, we will agree that
$B_Q(r)=B_Q(\floor{r})$ and $\Hull_{Q}(r)=\Hull_{Q}(\floor{r})$.

We also need to define {\em truncated hulls}. To this end, consider first the case where $Q$ is finite and pointed. We
label the vertices of $Q$ with their graph distance to $\rho$, and we consider an integer $r$ such that $0<r<\dgr^Q(\rho,\partial)$.
Inside every face such that the labels of its incident corners are $r, r-1, r, r+1$, we draw a ``diagonal'' between the corners of label $r$. The added edges form a collection of cycles, from which we extract a ``maximal'' simple cycle $\partial_r Q$. 
 This cycle is maximal in the sense that the connected component of the complement of $\partial_r Q$ that contains the distinguished vertex contains no vertex with label less than or equal to $r$. See \cite[Section 2.2]{gall2017separating} for more details. 
 The exterior of $\partial_r Q$ is the connected component of the complement of $\partial_r Q$ that contains the marked vertex of $Q$. 
 If $Q$ is an infinite quadrangulation, the cycles $\partial_rQ$ can be defined in exactly the same way, now for every integer $r>0$ (the exterior of $\partial_rQ$
 is now the the unbounded connected component of the complement of $\partial_r Q$).

 In both the finite and the infinite case, the \emph{truncated hull} of radius $r$ of $Q$ is the map $\Htr_Q(r)$ made of $\partial_r Q$ and of the edges of $Q$ inside $\partial_r Q$, and is rooted
 at the ``same'' edge as $Q$. 
Then we may view $\Htr_Q(r)$ as a truncated quadrangulation (for which the external face corresponds to the exterior of $\partial_r Q$) provided we re-root 
$\Htr_Q(r)$ at an edge of its external boundary.

\paragraph{Quadrangulations of the cylinder.}

A \emph{quadrangulation of the cylinder of height $R>0$} is a rooted map $Q$ with two distinguished faces, called the \emph{top} and \emph{bottom} faces (the other faces are called \emph{inner faces}), such that
\begin{romenum}
\item the boundary of the top (\resp bottom) face, called the \emph{top boundary} (\resp the \emph{bottom boundary}) is a simple cycle, 
\item $Q$ is rooted at an oriented edge of its bottom boundary so that the bottom face lies on the right of the root edge (so the bottom face is the root face),
\item every edge of the top (\resp bottom) boundary is incident both to the top (\resp bottom) face and to a triangular face, these triangular faces are distinct, and all other inner faces have degree 4,
\item any vertex of the top boundary is at graph distance $R$ from the bottom boundary, and the inner triangular face incident to any edge of the top boundary is also incident to a vertex at graph distance $R-1$ from the bottom boundary.
\end{romenum}
Let $Q$ be a quadrangulation of the cylinder of height $R$. Label every vertex of $Q$ by its distance from the bottom boundary. For $0<r<R$ we can define $\partial_r Q$ and the truncated hull of radius $r$ in a way very similar to what we did for pointed quadrangulations ($\partial_r Q$ is the ``maximal'' cycle made of diagonals between corners labeled $r$ in faces 
whose corners are labeled $r-1,r,r+1,r$, and the exterior of $\partial_r Q$ now contains the top cycle). See \cite[Section 2.3]{gall2017separating} for details. 
Note that the truncated hull of radius $r$ of $Q$ is itself a quadrangulation of the cylinder of height $r$. 
By convention, we agree that $\partial_0 Q$ denotes the bottom boundary, and $\partial_R Q$ stands for the top boundary of $Q$. 
We will assume that quadrangulations of the cylinder are drawn in the plane in such a way that the top face is the unbounded face (see \figref{Fig_skeleton_decomp}). Then we will orient the cycles $\partial_r Q$ clockwise by convention.

\begin{figure}
\begin{center}
\includegraphics[width=0.5\textwidth]{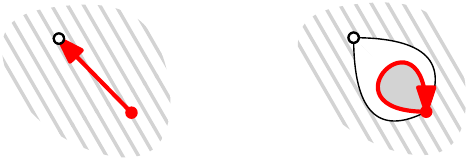}
\end{center}
\caption{We can split the root edge of a truncated hull, add a loop inside the newly created face, and see the map as a quadrangulation of the cylinder of bottom cycle the added loop.}
\label{Fig_cylinder_and_root}
\end{figure}

We may view the truncated hull of radius $r$ of a pointed quadrangulation $Q$ as a quadrangulation of the cylinder of height $r$, by splitting the root edge of $Q$ into a double edge and adding a loop inside the newly created face as in \figref{Fig_cylinder_and_root}. In this way, we get a quadrangulation of the cylinder of height $r$ whose bottom cycle is a loop.

\paragraph{\Lmgs.}

Let $Q$ be a quadrangulation of the cylinder of height $R$ and let $0<r\leq R$. 
We now explain a ``canonical'' choice of a geodesic between a vertex of $\partial_r Q$ and the bottom cycle $\partial_0 Q$. 
So let $v$ be a vertex on $\partial_r Q$, and let
$e$ be the edge of $\partial_r Q$ with tail $v$ (recall our convention for the orientation of $\partial_r Q$). Then list all edges incident to $v$ in clockwise order around $v$, starting from $e$. 
 The first step of the \emph{\lmg} from $v$ to $\partial_0 Q$ is the last edge in this enumeration that connects $v$ to $\partial_{r-1} Q$
 (Property (iv) above ensures that there is at least one such edge). We define the next steps of the geodesic by the obvious induction.

\paragraph{Skeleton decomposition.}

Let $Q$ be a quadrangulation of the cylinder of height $R$. Following ideas in \cite{Krikun2008local}, the article \cite{gall2017separating} gives a representation of $Q$ by a forest of planar trees of height at most $R$ and a collection of truncated quadrangulations indexed by the vertices of this forest. We refer to \cite{gall2017separating} for more details, and give a brief presentation.

We first add the edges of $\partial_r Q$ to $Q$ for every $0<r\leq R$ and recall that these edges are oriented clockwise in each cycle $\partial_r Q$. Let $0<r\leq R$, and
let $e$ be an edge of $\partial_r Q$. Then $e$ is incident to exactly one triangular face in $Q$ whose third vertex is at distance $r-1$ from the bottom boundary. We call this face the \emph{downward triangle with top edge $e$}. 
Furthermore, if $v$ is the aim of $e$, the downward triangle with top edge $e$ is also incident to the first edge of the \lmg from $v$ to the bottom boundary.

\begin{figure}
\begin{center}
\includegraphics[width=0.48\textwidth]{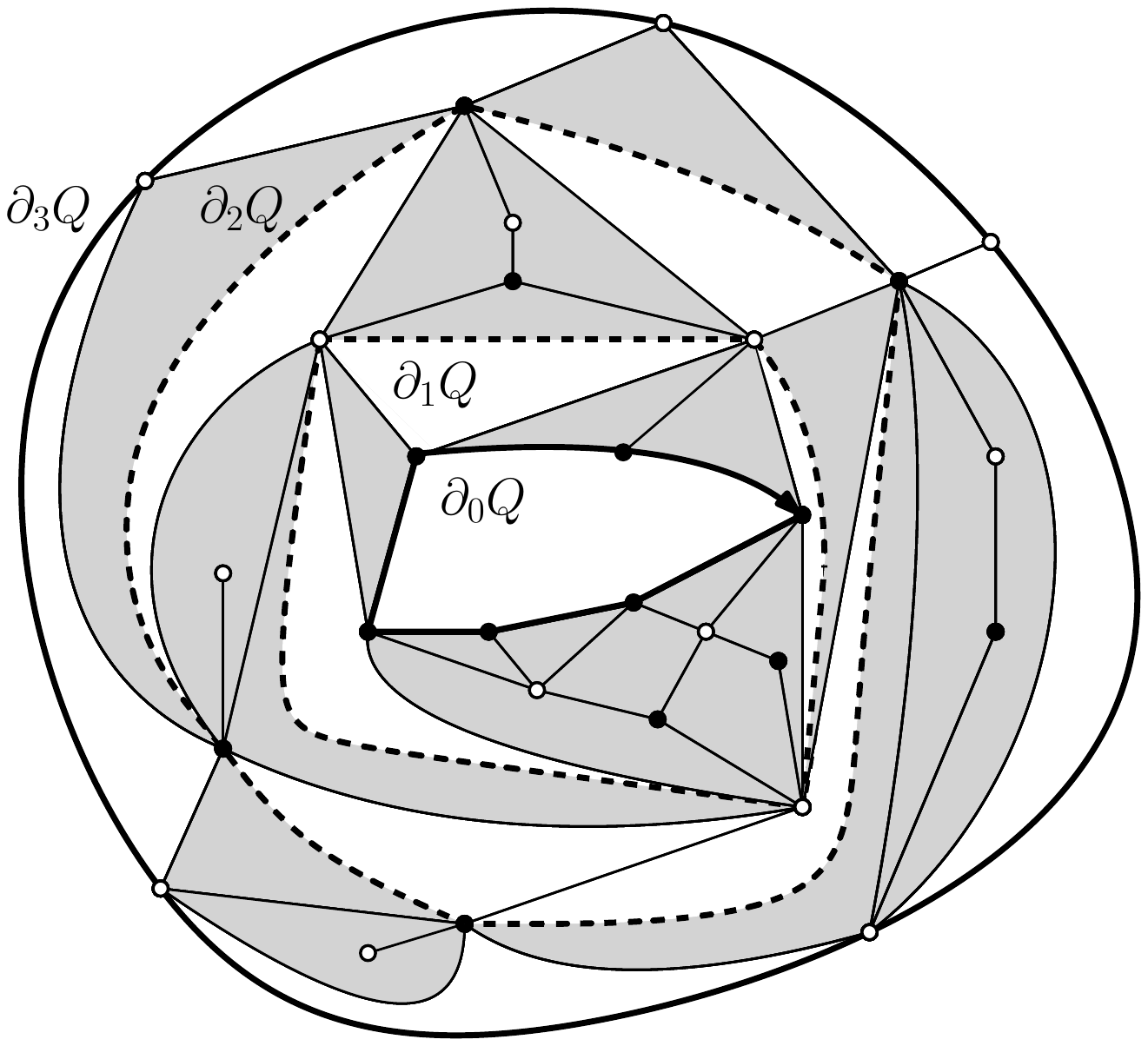}\hfill\includegraphics[width=0.48\textwidth]{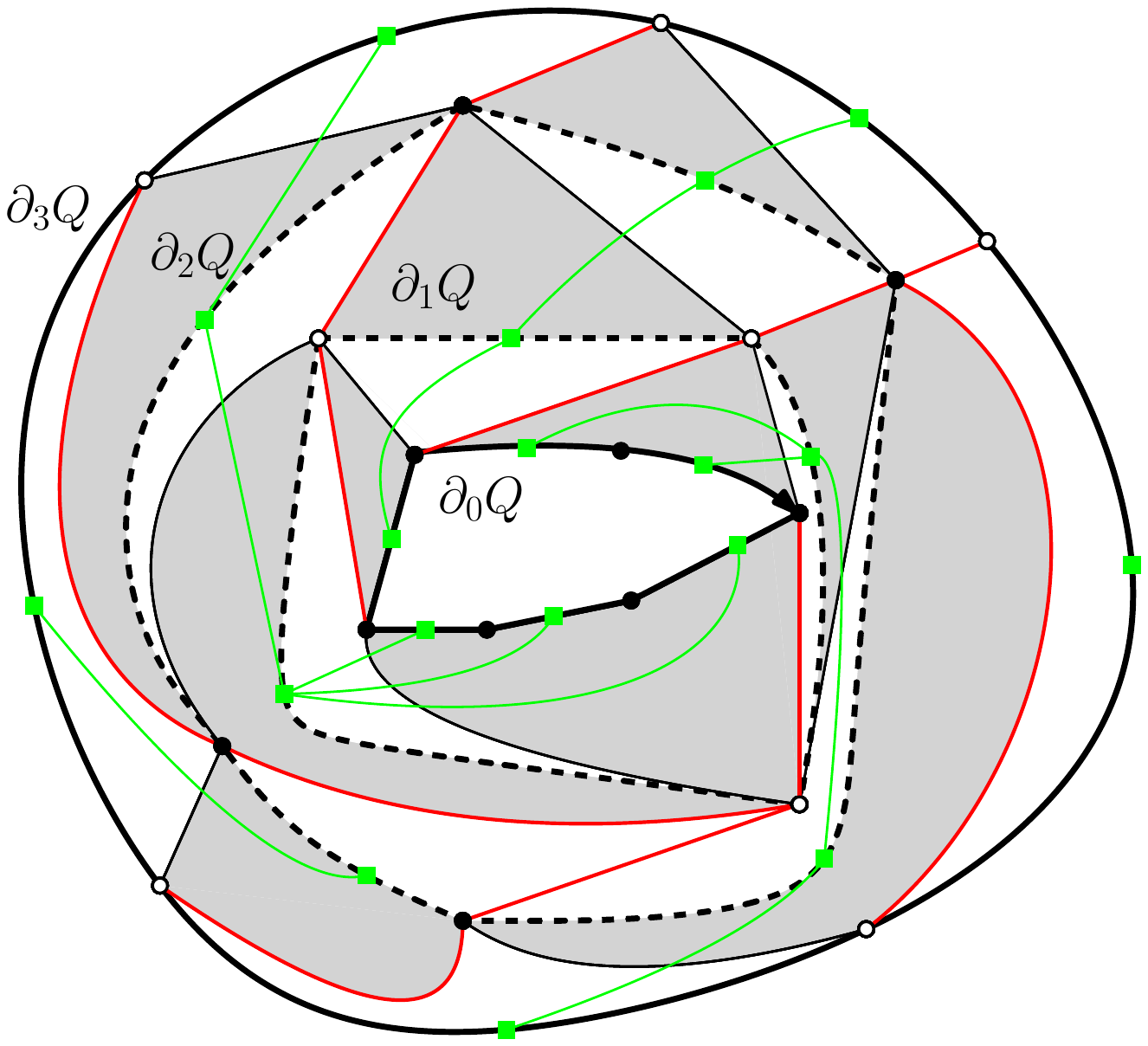}
\end{center}
\caption{Left, a quadrangulation of the cylinder of height $3$, with its cycles $\partial_1 Q$, $\partial_2 Q$ in dashed lines. We always draw the top face as the infinite face. Downward triangles are in white and the slots in grey. Right, we erased the content of the slots; in green, the genealogical relation on edges of $\partial_r Q$ for $0 \leq r \leq 3$, in red, the \lmgs to the bottom cycle follow the ``right'' side of downward triangles (assuming their top edge is ``up''), or equivalently the ``left'' side of slots.}
\label{Fig_skeleton_decomp}
\end{figure}
 
The downward triangles disconnect $Q$ into a collection of \emph{slots}, which are filled in by finite maps with a simple boundary. See \figref{Fig_skeleton_decomp}. Any slot is contained in the region between $\partial_r Q$ and $\partial_{r-1} Q$ for some $1\leq r \leq R$, and there is a unique vertex $v$ of $\partial_r Q$ that is incident to the slot. We then say that the slot is associated with the edge of $\partial_r Q$ whose tail is $v$. We equip the set of edges of $\cup_{r=0}^R \partial_r Q$ with the following genealogical relation: for every $0<r\leq R$, an edge $e \in \partial_r Q$ is the parent of all the edges of $\partial_{r-1} Q$ that are incident to the slot associated with $e$, provided that this slot exists. See the right side of \figref{Fig_skeleton_decomp}.

We now explain how we can define the truncated quadrangulation associated with a slot, or rather with an edge $e \in \partial_r Q$, $1 \leq r \leq R$. Suppose first that there is a slot associated with $e$. The part of $Q$ inside this slot, including its boundary, defines a planar map with a simple boundary and a distinguished vertex on this boundary. Adding one edge in the way explained in \figref{Fig_Construction_Tuile0} turns this map into a truncated quadrangulation $M$, which is rooted at the added edge as shown on \figref{Fig_Construction_Tuile0}. If there is no slot associated with $e$ we let $M$ be the unique truncated quadrangulation with perimeter one and one inner face (rooted at its boundary edge so that the external face lies on the right of the root edge).

\begin{figure}[h!]
\begin{center}
\includegraphics[width=\textwidth]{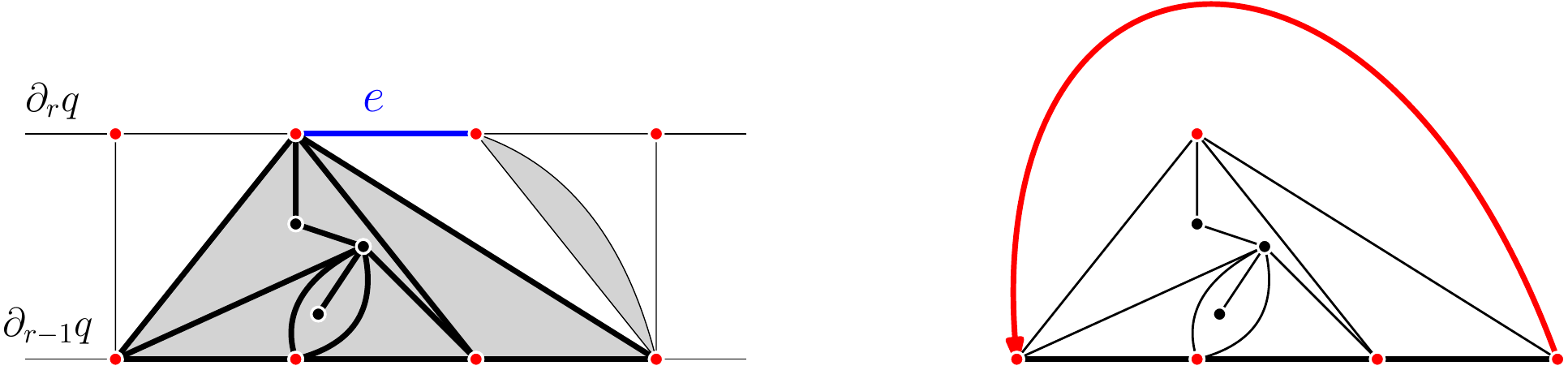}
\end{center}
\caption{Left (in grey and thick black lines), the map with simple boundary filling the slot of an edge $e$ (in blue) of $\partial_r Q$ with $3$ offspring. Right, the truncated quadrangulation we obtain by adding a root edge ``over the top vertex of the slot''.}
\label{Fig_Construction_Tuile0}
\end{figure}

Recalling the definition in \cite{gall2017separating}, we say that a forest $\FF$ with one marked vertex is \emph{$(R,p,q)$-admissible} if
\begin{romenum}
\item the forest consists of an ordered sequence of $q$ rooted plane trees,
\item these trees have height at most $R$,
\item exactly $p$ vertices of the forest are at height $R$,
\item the marked vertex is at height $R$ and belongs to the first tree.
\end{romenum}
Write $\F^0_{R,p,q}$ for the set of all $(R,p,q)$-admissible forests. We will also need the set of all forests with no marked vertex that satisfy properties (i) to (iii) above, and we denote this set by $\F_{R,p,q}$.

For $\FF \in \F^0_{R,p,q}$ or $\FF \in \F_{R,p,q}$ we let $\FF^*$ denote the set of all vertices of $\FF$ at height at most $R-1$. For every $e\in \FF^*$, $c_e$ is the 
offspring number of $e$ in $\FF$.

The construction described above and illustrated in \figref{Fig_skeleton_decomp} provides a bijection that, with every quadrangulation $Q$ of the cylinder of height $R$, with top perimeter $q$ and bottom perimeter $p$, associates an $(R,p,q)$-admissible forest $\FF$ and a collection $(S_e)_{e \in \FF^*}$ of truncated quadrangulations, such that $S_e$ has perimeter $c_e+1$ for every $e \in \FF^*$. The forest encodes the genealogical relation of edges of $\cup_{r=0}^R \partial_r Q$: each tree in $\FF$ corresponds to the descendants of an edge of the top boundary, the first tree is the one that contains the root edge, and the other trees are then listed by following the clockwise order on the top boundary. 
We call the forest $\FF$ the \emph{skeleton} of $Q$.

On the other hand, the union of all \lmgs forms a forest of trees made of edges of $Q$. This forest can be viewed as dual to the skeleton of $Q$, and the two forests do not cross, as suggested in \figref{Fig_skeleton_decomp}. This has the following important consequence. Let $v, w$ be two distinct vertices of the top boundary of $Q$, let $\FF'$ be the forest consisting of the trees of the skeleton that are rooted on the part of the top boundary between $v$ and $w$ (in clockwise order), and let $\FF''$ consist of the other trees in the skeleton. Then for every $0 \leq r < R$, the \lmgs from $v$ and $w$ coalesce before reaching $\partial_r Q$ or when hitting $\partial_r Q$ iff either $\FF'$ or $\FF''$ has height strictly smaller than $R-r$.

\paragraph{Law of the skeleton decomposition of the UIPQ.}

The following formulas are derived by singularity analysis from the generating series of truncated quadrangulations, computed in \cite{Krikun2008local}. See \cite[Section 2.5]{gall2017separating}.

For every $1 \leq p \leq n$, let $\Qtr_{n,p}$ be the set of all truncated quadrangulations with a boundary of length $p$ and $n$ inner faces. There exists a sequence $(\vk_p)_{p \geq 1}$ of positive reals such that for every $p \geq 1$,
\begin{equation*}
\#\Qtr_{n,p} \usim n \infty \vk_p n^{-5/2} 12^{-n} .
\end{equation*}
Furthermore,
\begin{equation}
\label{Eq_asymptotics_vkp}
\vk_p \usim p \infty \frac{64 \sqrt 3}{\pi \sqrt 2} \sqrt p 2^{-p} .
\end{equation}
For every $p \geq 1$, we define
\begin{align*}
h(p) &\eqdef \frac{1}{p} 2^p \vk_p , \\
Z(p) &\eqdef \sum_{n=p}^\infty \#\Qtr_{n,p} 12^{-n} , 
\end{align*}
and we define the Boltzmann probability measure $\vG_p$ on $\cup_{n=p}^\infty \Qtr_{n,p}$ by setting 
\begin{align*}
\vG_p(Q) \eqdef \frac{12^{-n}}{Z(p)} .
\end{align*}
for every $Q\in \Qtr_{n,p}$, $n\geq p$.
We also set, for every $p\geq 0$,
\begin{align*}
\vt(p) &\eqdef 6 \cdot 2^{p} Z(p+1) .
\end{align*}
Then \cite[Lemma 6]{gall2017separating}, $(\vt(p))_{p\geq 0}$ is a critical offspring distribution with generating function
\begin{equation*}
g_\vt(y) = 1-\frac{8}{\l(\sqrt{\frac{9-y}{1-y}}+2\r)^2 -1} .
\end{equation*}
Let $g^{(p)}_\theta=g_\theta\circ\cdots\circ g_\theta$ denote the $p$-th iterate of $g_\theta$. If $(Y_p)_{p\geq 0}$ is a \BGW process with offspring distribution $\vt$ started at $Y_0 = 1$, then for every $p\geq 0$,
\begin{equation*}
\E\l[ y^{Y_p} \r] = g_\vt^{(p)}(y) = 1-\frac{8}{\l(\sqrt{\frac{9-y}{1-y}}+2p\r)^2 -1} , \quad 0 \leq y < 1 .
\end{equation*}

We now consider the uniform infinite planar quadrangulation $Q_\infty$ to which we can apply the preceding definition to get the cycles $\partial_r Q_\infty$ for every $r\geq 0$ (when $r=0$ we split the root edge as explained in \figref{Fig_cylinder_and_root}). For $r\geq 1$, we distinguish the edge of $\partial_r Q_\infty$ that corresponds to the first tree of the skeleton of the truncated hull of radius $r$, viewed as a quadrangulation of the cylinder of height $r$. Then, for every $0\leq r <s$ we define the annulus $\CC(r,s)$ as the quadrangulation of the cylinder of height $s-r$ that corresponds to the part of $Q_\infty$ between $\partial_r Q_\infty$ and $\partial_s Q_\infty$, rooted at the distinguished edge of $\partial_r Q_\infty$.

Let $( \FF^0_{r,s}, (S_e)_{e \in {\FF^{0*}_{r,s}}} )$ be the skeleton decomposition of $\CC(r,s)$. 
Conditionally on the skeleton $\FF_{r,s}^0$, the truncated quadrangulations $S_e$, $e\in \FF_{r,s}^{0,*}$, are independent and the conditional distribution of $S_e$ is $\vG_{c_e+1}$, where we recall that $c_e$ is the number of offspring of $e$ in $\FF_{r,s}^0$. See \cite[Corollary 8]{gall2017separating}.

Let $H_r$ be the length of $\partial_r Q_\infty$.
\begin{prop}[\cite{gall2017separating}, Proposition 11]
\label{Prop_law_perimeter_UIPQ}
For every $r\geq 1$ and $p \geq 1$,
\begin{equation}
\P(H_r = p) = K_r \vk_p (2\pi_r)^p ,
\end{equation}
where
\begin{equation}
\label{Eq_formula_for_pi}
\pi_r = g_\vt^{(r)}(0) = 1-\frac{8}{(3+2r)^2-1} 
\end{equation}
is the probability that a \BGW process of offspring distribution $\vt$ started at $1$ becomes extinct before generation $r$, and
\begin{equation*}
K_r = \frac{32}{3\vk_1} \frac{3+2r}{((3+2r)^2-1)^2} .
\end{equation*}
Consequently, we can find positive constants $M_1, M_2$ and $\rho$ such that for every $a>0$ and $r\geq 1$,
\begin{align*}
\P(H_r \geq ar^2) &\leq M_1 e^{-\rho a} , \\
\P(H_r \leq ar^2) &\leq M_2 a^{3/2} .
\end{align*}
\end{prop}

From the skeleton $\FF_{r,s}^0$, we define a new forest $\FF_{r,s}$ by ``forgetting'' the marked vertex and applying a uniform random circular permutation to the trees of $\FF_{r,s}^0$. Then $\FF_{r,s}$ is a random element of $\cup_{p \geq 1, q \geq 1} \F_{s-r, p, q}$. 

\begin{prop}[\cite{gall2017separating}, Corollary 10]
\label{Prop_law_annulus_UIPQ_cond_bottom}
Let $p\geq 1$. The conditional distribution of $\FF_{r,s}$ knowing that $H_r = p$ is as follows: for every $q>0$, for every $\FF \in \F_{s-r,p,q}$, 
\begin{equation}
\P\l( \FF_{r,s} = \FF \ | \ H_r = p \r) = \frac{h(q)}{h(p)} \prod_{e \in \FF^*} \vt(c_e) .
\end{equation}
\end{prop}

\begin{prop}
\label{Prop_law_annulus_UIPQ_cond_top}
Let $q \geq 1$. The conditional distribution of $\FF_{r,s}$ knowing that $H_s = q$ is as follows: for every $p>0$, for every $\FF \in \F_{s-r,p,q}$, 
\begin{equation}
\P\l( \FF_{r,s} = \FF \ | \ H_s = q \r) = \frac{\vp_r(p)}{\vp_s(q)} \prod_{e \in \FF^*} \vt(c_e) ,
\end{equation}
where
\begin{equation}
\label{Eq_phirp}
\vp_r(p) = \frac{64}{3}p \frac{3+2r}{((3+2r)^2-1)^2} \pi_r^{p-1} .
\end{equation}
\end{prop}

We refer to formulas (18) and (19) in \cite{gall2017separating} for the last proposition.

\section{The lower half-plane quadrangulation}
\label{Sec_LHPQ}

\subsection{Definition of the model}

We construct an infinite quadrangulation with an infinite truncated boundary, which will be denoted by $\LL$ and called the \emph{lower half plane quadrangulation} or \index{LHPQ@\LHPQ}\LHPQ. Roughly speaking, $\LL$ is what we see near a uniformly chosen random edge of the boundary of a very large truncated hull of $Q_\infty$ (see \propref{Prop_LHPQ_is_local_limit_truncated_hulls} below for a more precise statement). Our construction relies on the skeleton decomposition.

It will be convenient to use a particular embedding of $\LL$ in the plane (see Figure~\ref{Fig_Skeleton_geodesics_LHPQ}).  
In the rest of this paper, we denote the set of non-negative integers by $\N$, and the set of non-positive integers by $-\N = \{0, -1, -2, ...\}$. Every point of $\Z \times -\N$ will be a vertex of $\LL$; the edges of the form $((i,0),(i+1,0))$ for $i\in\Z$ will be the edges of the boundary of $\LL$; and the upper half-plane will correspond to an ``external face'' of $\LL$. Furthermore, $\LL$ will be rooted at the edge $((0,0),(1,0))$.

\begin{figure}[h!]
\begin{center}
\includegraphics[width=\textwidth]{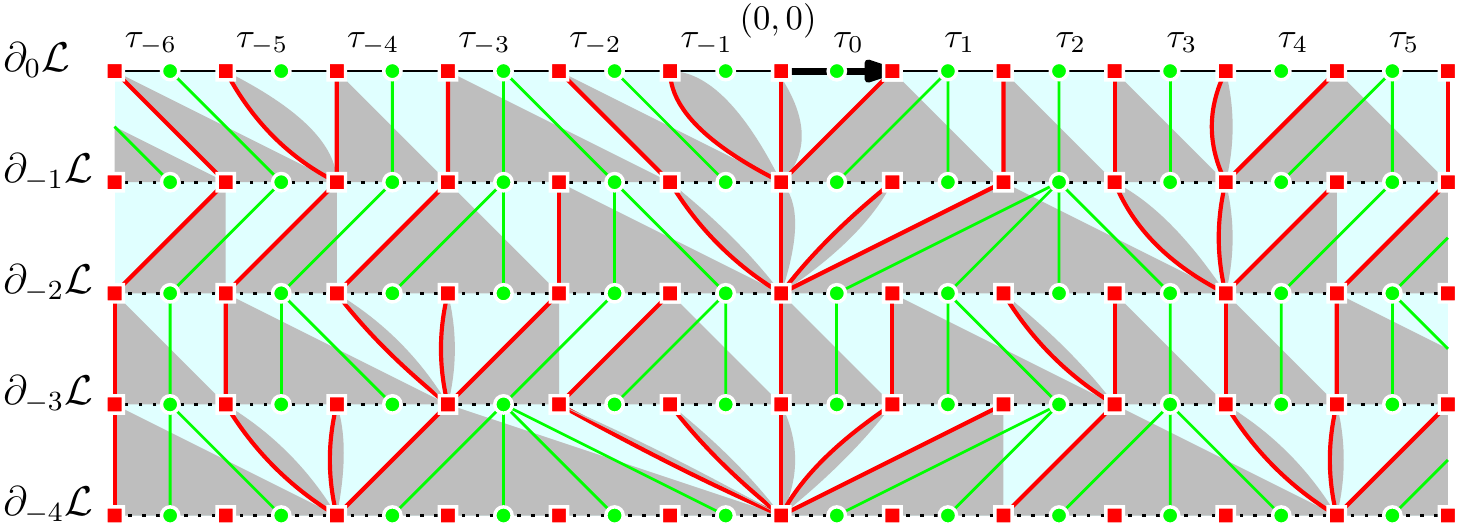}
\end{center}
\caption{Embedding of the \LHPQ and its skeleton. The vertices of the lattice $\Z\times -\N$ are the red dots, the vertices of the forest are the green dots, the trees are drawn in green. We drew the downward triangles in cyan, and the slots in grey. Notice that the slot associated with an edge having no offspring may be empty though it is represented here
as the ``inside'' of a double edge (this simply means that this double edge may be glued to form a single edge).
\\
In red, \lmgs in the \LHPQ, named so because they follow the left-most edge going downward. They follow the left side of slots, or the right side of downward triangles. The \lmgs form the dual tree of the skeleton.}
\label{Fig_Skeleton_geodesics_LHPQ}
\end{figure}

In order to construct $\LL$, we start from a forest $(\tau_i)_{i\in \Z}$ of \iid \BGW trees with offspring distribution $\theta$ (as usual these trees are random plane trees). This forest will be the \emph{skeleton}\index{skeleton} of the \LHPQ. A.s. each generation has an infinite number of individuals: we can thus embed vertices of the skeleton at generation $r \geq 0$ bijectively on $\{ (j+\frac{1}{2},-r), j\in \Z \}$, in a way that is consistent with the order on vertices at generation $r$ of the forest, so that vertices in trees with non-negative indices $(\tau_i)_{i\geq 0}$ fill the lower right quadrant $(1/2,0)+(\N \times -\N)$, and vertices in trees with negative indices $(\tau_i)_{i<0}$ fill the lower left quadrant $(-1/2,0)+ (-\N \times -\N)$.
In particular the root of $\tau_i$ will be $(i+1/2,0)$. See the green trees on Figure~\ref{Fig_Skeleton_geodesics_LHPQ}.
 In what follows, we identify vertices of the skeleton and the points where they are embedded. 

We denote by $\partial_{-r} \LL$ the infinite line $\Z \times \{-r\}$ viewed as a linear graph. The skeleton of $\LL$ induces a genealogical relation on the edges of $\cup_{r \geq 0} \partial_{-r} \LL$, if we identify each edge with its middle point.

To each edge $e$ of $\partial_{-r} \LL$, we associate a downward triangle with ``top boundary'' $e$ and ``bottom vertex'' the vertex $v$ of $\partial_{-r-1} \LL$, chosen as follows. If $e$ has at least one offspring, then $v$ is the right-most vertex of the edge which is the right-most offspring of $e$. If not, let $e'$ be the first edge of $\partial_r \LL$ on the left of $e$ having at least one offspring. Then the downward triangles of $e$ and $e'$ have the same bottom vertex.

Downward triangles delimit a collection of \emph{slots} (see Figure~\ref{Fig_Skeleton_geodesics_LHPQ}), and each slot is associated with an edge of $\cup_{r \geq 0} \partial_{-r} \LL$ as in the \UIPQ. By construction,  the edges of the lower boundary of a slot are exactly the offspring of the associated edge. The last step to get the \LHPQ is to fill in the slots, and we do so exactly as in the \UIPQ, see \figref{Fig_Construction_Tuile2}. We note that, for $r > 0$, edges of $\partial_{-r} \LL$ do not belong to $\LL$: these edges are removed in
order to get quadrangles by the gluing of two triangles.

\begin{figure}[h!]
\begin{center}
\includegraphics[width=0.9\textwidth]{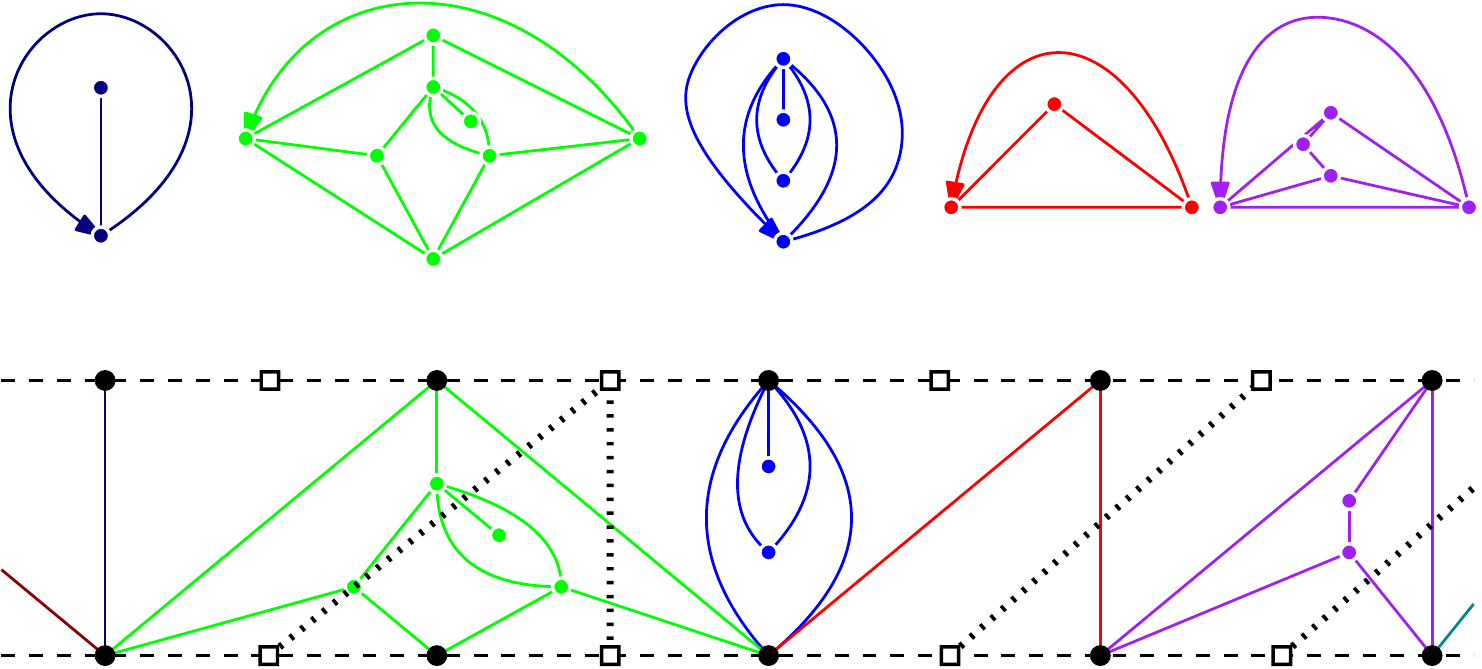}
\end{center}
\caption{Constructing a part of the \LHPQ from truncated quadrangulations. In the lower part, the dotted lines are trees of the skeleton and the dashed lines are $\partial_{-r-1} \LL$ and $\partial_{-r} \LL$ for some $r \geq 0$. 
The upper part of the figure shows 5 truncated quadrangulations that fill in successive slots as shown in the lower part. Note that any edge of a truncated quadrangulation that is glued to some edge of $\partial_{-r} \LL$ is removed in the \LHPQ.
}
\label{Fig_Construction_Tuile2}
\end{figure}

\Lmgs in $\LL$ are defined as in the case of quadrangulations of the cylinder, but are now infinite paths on $\LL$ that start at a vertex of
$\partial_{-r} \LL$ and then visit each line $\partial_{-r'} \LL$, $r'\geq r$, exactly once. \Lmgs form a forest whose vertex set is the lattice $\Z \times -\N$. This forest and the skeleton never intersect and can be seen as ``dual'' to each other. Furthermore, the two \lmgs started at $(i,0)$ and $(i',0)$ with $i<i'$ coalesce before height $j<0$ or exactly at height $j$ if and only if every $\tau_k$ with $i <k+\frac{1}{2}< i'$ becomes extinct before generation $-j$. See \figref{Fig_Skeleton_geodesics_LHPQ} for an illustration.

Note that the \lmg started from any point $(i,0)$ of the top boundary breaks the \LHPQ into two halves. Any path whose endpoints are on different sides of this geodesic must cross it through one of its vertices. Note finally that \lmgs never cross any tree of the skeleton. As an example, the \lmg started from $(0,0)$ is the vertical line $((0,-n))_{n\geq 0}$.

As another useful observation, we note that the graph distance between any point of $\partial_r \LL$ and $\partial_{r'} \LL$ (with $r'<r$) is exactly $r-r'$.

\subsection{The lower-half plane quadrangulation is the local limit of large hulls}

The following proposition explains why we consider the model of the \LHPQ. 

\begin{proposition}
\label{Prop_LHPQ_is_local_limit_truncated_hulls}
For every $r>0$, let $\HH_r$ be the truncated hull of radius $r$ of the \UIPQ, re-rooted at a uniformly chosen edge of its boundary. Then
\begin{equation*}
\HH_r \ulimd r \infty \LL
\end{equation*}
for the local topology on rooted planar maps.
\end{proposition}

We omit the proof as we will not need this result in the remaining part of the paper. See Proposition 7 in \cite{fpp}
for the analogous statement in the case of triangulations.

\subsection{Control of distances along the boundary}
\label{Sec_Control of distances along the boundary of the LHPQ}

\subsubsection{The main estimate}
\label{control-main-estimate}

The following proposition shows that $\dgr^\LL((0,0),(j,0))$ grows at least like $\sqrt j$. 

\begin{proposition}
\label{Prop_P15}
For every $\ve>0$, there exists an integer $K\geq 1$ such that for every $r\geq 1$, for every integer $A>0$, 
\begin{equation*}
\P \left( \min_{|i|\geq A+Kr^2} \min_{-A \leq i' \leq A} \dgr^\LL((i',0),(i,0)) \geq r \right) \geq 1-\epsilon.
\end{equation*}
\end{proposition}

In order to prove this proposition, we adapt the proof of \cite[theorem 5]{curien2017geometric}. This result does not apply directly to our settings, but to another model also constructed from a Bienaym\'e Galton-Watson forest.

Let us first define slices, half-slices, and blocks of the \LHPQ.

\paragraph{Slice.}\index{slice}

Let $-\infty<j'<j \leq 0$. Consider all vertices and edges of $\LL$ contained in $\R \times [j',j]$ and add all edges of the form $((i,j),(i+1,j))$ and $((i,j'),(i+1,j'))$ for $i\in\Z$. The resulting map is called the slice $\LL_{j'}^j$. By convention it is rooted at $((0,j),(1,j))$. 
 The skeleton of $\LL_{j'}^j$ is the planar forest $(\tau^{(j,j')}_n)_{n\in \Z}$ corresponding to the part of the skeleton of $\LL$ between generation $-j$ and $-j'$ (these trees are numbered as
 previously, so that $\tau^{(j,j')}_n$ is rooted at the vertex $(n+\frac{1}{2},j)$).

\paragraph{Half-slice.}\index{half-slice}

The half-slice $\HH\LL_{j'}^j$ is the part of the slice $\LL_{j'}^j$ that is contained in $\R_+ \times [j',j]$. Its skeleton consists of the trees of the skeleton of $\LL_ {j'}^j$ with nonnegative indices.

\paragraph{Blocks.}\index{block}

Cut the half-slice $\HH\LL_{j'}^j$ along the \lmgs that follow the right boundary of trees of maximal height in its skeleton. 
 We obtain a sequence of finite maps, which we will call blocks as in \cite{curien2017geometric}.

Let us give a precise definition of these blocks (see also \figref{Fig_Half slice}). 
 Let $(\xi_n)_{n>0}$ be the sequence of all indices (in increasing order) of trees that reach height $j-j'$ in the skeleton of $\HH\LL_{j'}^j$, and add the convention that $\xi_0 = -1$. Let $n>0$ be an integer. The $n$-th block $\HH\LL_{j'}^j(n)$ is the part of $\HH\LL_{j'}^j$ contained between the \lmgs started at $(\xi_{n-1}+1,j)$ and at $(\xi_{n}+1,j)$ respectively. The left boundary of a block is the \lmg on its left, its right boundary is the \lmg on its right.

\begin{figure}[h!]
\begin{center}
\includegraphics[width=0.9\textwidth]{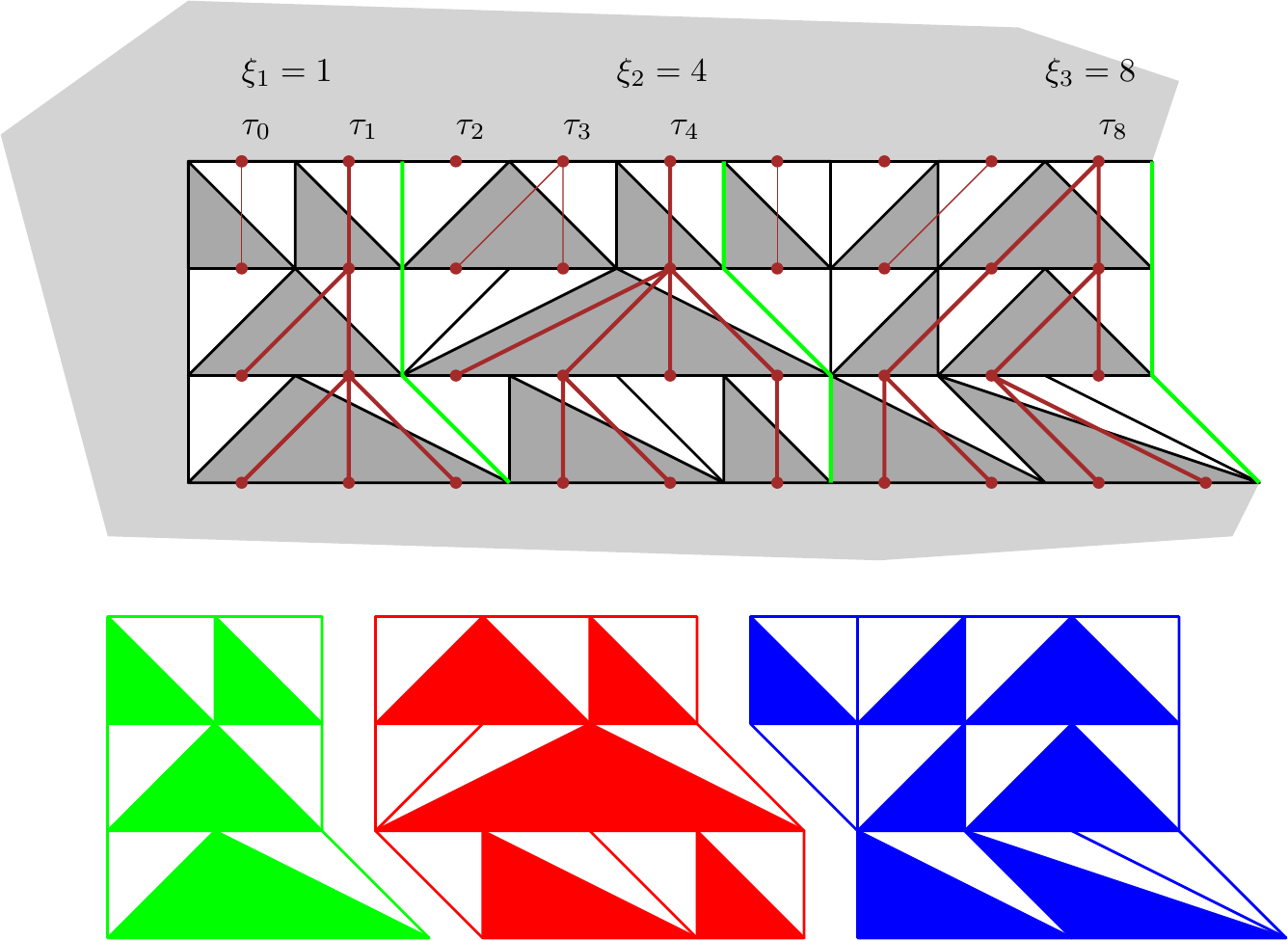}
\end{center}
\caption{Block decomposition of the half-slice $\HH\LL_{-3}^0$. The first block is pictured in green, the second one in red, the third one in blue. In the upper figure, we represented the skeleton in brown (with trees of height $3$ in thick lines), slots in dark gray, and the \lmgs following the right boundary of trees of height $3$ in green. In this example, $\xi_1=1$, $\xi_2=4$, $\xi_3=8$.}
\label{Fig_Half slice}
\end{figure}

The skeleton of $\LL$ is made of \iid trees, thus all the blocks $\HH\LL_{j'}^j(n)$, viewed as planar maps with a boundary, are independent and share the same law, which only depends on $j-j'$.

The \emph{thickness}\index{thickness} of the $n$-th block $\Epaiss(\HH\LL_{j'}^j(n))$ (called diameter in \cite{curien2017geometric}) is the minimal graph distance in this block between a point of its left boundary and a point of its right boundary. Note that the thickness of a block cannot be $0$, since the presence of a tree of maximal height implies that the \lmgs corresponding to the left and right boundary of the block do not coalesce (by a previous remark).

Furthermore, the thickness of $\HH\LL_{j'}^j(n)$ is always smaller than $j-j'+1$. Indeed, the \lmg started at $(\xi_{n-1}+1,j)$ (i.e. the left boundary) and the \lmg started at $(\xi_{n},j)$ coalesce before height $j'$ (i.e. after at most $j-j'$ steps), since no tree rooted between $\xi_{n-1}+1$ and $\xi_n$ reaches height $j-j'$. 
In this way, we get a path of length at most $j-j'$ that connects the left boundary of the block to the vertex $(\xi_{n},j)$, and we just have to add  the edge $((\xi_{n},j),(\xi_{n}+1,j))$ to get the desired bound.

It will be useful to note the following simple fact: any path that stays in the half-slice $\HH\LL_{j'}^j$ with one endpoint on the left side of the left boundary of some $\HH\LL_{j'}^j(n)$, and its other endpoint on the right side of its right boundary, has a length which is at least the thickness of the block.

Let us outline the key idea of the proof of \propref{Prop_P15}. A path of length $r$ between two vertices of the boundary of $\LL$ cannot exit the slice $\LL_{-h}^0$ for any $h \geq r$.
We apply this observation with $h=\lceil C r\rceil$ with some constant $C\geq1$. If we fix $K>0$ large enough, then with high probability we will find a block of $\LL_{-h}^0$ with top boundary included in $[A, A+Kr^2]\times \{0\}$, and any path in $\LL_{-h}^0$ that goes from the left side  to the right side of $[A, A+Kr^2]\times\{0\}$ must cross this block. All we need to conclude is the fact that we can choose $h$ such this block has thickness at least $r$ with high probability.

The latter fact is derived from the following result, which is adapted from \cite[theorem 5]{curien2017geometric}. For every 
integer $h>0$, let $\GG(h)$ be a random variable with the law of a block of height $h$. Fix $\ve \in (0,1)$. 
The $\ve$-quantile $f_\ve(h)$ of the thickness of $\GG(h)$ is the largest integer $n$ such that
\begin{equation}
\label{Eq_minoration_quantile_thickness}
\P(\Epaiss(\GG(h)) \geq n) \geq 1-\ve .
\end{equation}
Note that $1\leq f_\ve(h) \leq h+1$ by previous observations.

\begin{proposition}
\label{Prop_f is linear}
For every $\ve \in (0,1)$, there exists $C_\ve \in (0,1)$ such that for every $h\geq 1$, $f_\ve(h) \geq C_\ve h$.
\end{proposition}

We postpone the proof of Proposition \ref{Prop_f is linear} to the next section and complete the proof of \propref{Prop_P15}.

\begin{proof}[Proof of \propref{Prop_P15}]
Let $r\geq 1$, and set $h=\ceil{r/C_{\ve/4}}$, where $C_{\ve/4}$ is given by \propref{Prop_f is linear}, and consider the first block of the half-slice $\LL_{-h}^0$, that is,
 $\HH\LL_{-h}^0(1)$ with the previous notation.  
By \propref{Prop_f is linear}, we have 
\begin{align*}
\P(\Epaiss(\HH\LL_{-h}^0(1)) \geq r) &\geq \P(\Epaiss(\HH\LL_{-h}^0(1)) \geq C_{\ve/4} h) \\
& \geq \P(\Epaiss(\HH\LL_{-h}^0(1)) \geq f_{\ve/4}(h)) \\
& \geq 1-\ve/4 .
\end{align*}
Let $\EE_1$ denote the event $\{ \Epaiss(\HH\LL_{-h}^0(1)) \geq r \}$.

The right-most point of the top boundary of $\HH\LL_{-h}^0(1)$ is $(\xi_{1}+1, 0)$, where $\xi_1$ follows a geometric law with parameter $1-\pi_{h} \geq c/r^2$ for some constant $c>0$ independent of $r$ (cf. \propref{Prop_law_perimeter_UIPQ}). We can take $K>0$ large enough so that $\xi_{1}+1 \leq Kr^2$ holds with probability larger than $1-\ve/4$. 
Let us call $\EE_2$ the event where $\xi_1+1 \leq Kr^2$.

On the event $\EE_1 \cap \EE_2$ of probability at least $1-\ve/2$ the block $\HH\LL_{-h}^0(1)$ has thickness at least $r$ and its top boundary is contained in $[0, Kr^2]\times \{0\}$. Then any two points $(i,0)$ and $(i',0)$ with $i\leq 0$ and $i' \geq Kr^2$ are at $\dgr^\LL$-distance at least $r$, since any path of length smaller than $r$ linking them necessarily crosses the block $\HH\LL_{-h}^0(1)$.

From an obvious argument of translation invariance, we obtain that for any integer $A>0$, with probability larger than $1-\ve/2$, any point in $(\-\infty, A] \times \{0\}$ and any point in $[A+Kr^2, +\infty)\times \{0\}$ are at $\dgr^\LL$-distance at least $r$. Similarly, with probability larger than $1-\ve/2$ the two half-lines $(\-\infty, -A-Kr^2] \times \{0\}$ and $[-A, +\infty) \times \{0\}$ are also at $\dgr^\LL$-distance larger than $r$. The statement of the proposition follows.
\end{proof}

\subsubsection{Proof of \propref{Prop_f is linear}}

For $h \geq 6$ and $m \in \{ 1,2,\ldots, \floor{\frac{h}{6}} \}$, we set
\begin{equation*}
J_m = \l\{ 0,-m,-2m, ..., -\l(\floor{\frac{h}{m}}-3\r) m, -h+3m \r\} .
\end{equation*}
We write $\GG(h)$ for a block of height $h$, and $\GG_{j'}^j(h)$ for the slice of $\GG(h)$ contained between heights $j'$ and $j$, for $-h \leq j' < j \leq 0$. We recall that $\ve \in (0,1)$ is fixed.

\begin{lemma}
\label{Lemma_claim on thickness of sublayers}
There exists $C\in(0,1/6)$ \st for all $h$ large enough and $0<m \leq C h$, the following property holds with probability at least $1-\ve$: for every $j\in J_m$, the length of any path connecting the left boundary of $\GG(h)$ to its right boundary and staying in $\GG_{j-3m}^j(h)$ is at least $C \pfrac{h}{m}^2 f_\ve(m)$. 
\end{lemma}

As a technical ingredient of the proof of Lemma \ref{Lemma_claim on thickness of sublayers}, we need a uniform lower bound on the size of the block at every generation. For $0\leq k\leq h$, let $X_k(h)$ denote the number of vertices of the skeleton of $\GG(h)$ at generation $k$.

\begin{lemma}
\label{Lemma_minoration de la taille des generations du bloc}
There exists a constant $C_1>0$ which does not depend of $h$, such that 
\begin{equation}
\label{Eq_bound proba EE1}
\P\l(\inf_{0 \leq k \leq h} X_k(h) > C_1 h^2 \r) \geq 1-\ve/2 .
\end{equation}
\end{lemma}
See \cite[Lemma 2]{curien2017geometric} for a proof of \lemref{Lemma_minoration de la taille des generations du bloc}.

\begin{proof}[Proof of \lemref{Lemma_claim on thickness of sublayers}]

The idea is to choose $C'$ small enough so that with high probability, one can find at least $C'(h/m)^2$ blocks of thickness at least $f_\ve(m)$ inside the slice $\GG_{j-3m}^{j}(h)$, for every $j \in J_m$. Any path connecting the left and right boundaries of this slice will then have length at least $C'(h/m)^2f_\ve(m)$.

We argue in the half-slice $\HH\LL_{-h}^0$. Note that the first block of this half-slice has the same law as $\GG(h)$. 

Let $C_1$ be chosen as in \lemref{Lemma_minoration de la taille des generations du bloc} so that \eqref{Eq_bound proba EE1} holds. Consider the half-slice $\HH\LL_{j-3m}^{j}$ for $j \in J_m$, and let $k$ be a positive integer. The number of blocks of this slice whose top boundary lies in $[0,\ceil{C_1 h^2}]\times \{j\}$ is distributed as the number of trees with height at least $3m$ in a forest of $\ceil{C_1 h^2}$ independent \BGW trees with offspring distribution $\vt$. Each block has a probability greater than $1-\ve$ of having thickness at least $f_\ve(3m)$. The number $N_j$ of blocks with thickness at least $f_\ve(3m)$ and top boundary in $[0,\ceil{C_1 h^2}] \times \{ j \}$ is then bounded below in distribution by a binomial variable with parameters $(\ceil{C_1 h^2}, (1-\ve)(1-\pi_{3m}))$, where $\pi_s$ is defined in \propref{Prop_law_perimeter_UIPQ}. By standard large deviation estimates for the binomial distribution and the bound $1-\pi_s \geq c/s^2$, we can find $C,C', C''>0$ such that for all large enough $h$, for every $m\leq Ch$, for every $j \in J_m$
\begin{equation*}
\P\l( N_j  < C' \pfrac h m ^2 \r) < \exp\l(-C'' \pfrac h m ^2 \r) .
\end{equation*}
Summing over $j \in J_m$ and taking $C$ even smaller if necessary, we get
\begin{equation}
\label{Eq_number_blocks_in_halfslice}
\P\l( \forall j\in J_m \ : \ N_j  \geq C \pfrac h m ^2 \r) \geq 1-\ve/2 .
\end{equation}

On the event of probability at least $1-\ve/2$ considered in \lemref{Lemma_minoration de la taille des generations du bloc}, the first $\ceil{C_1 r^2}$ vertices of the skeleton of $\HH\LL_{-r}^0$ at generation $j$ (i.e. the first $\ceil{C_1 r^2}$ vertices of the skeleton of $\HH\LL_{j-3m}^j$) belong to the first block $\HH\LL_{-r}^0(1)$, thus the $N_j$ blocks of thickness at least $f_\ve(3m)$ considered above are contained in $\HH\LL_{-r}^0(1)$. 
The property of \lemref{Lemma_claim on thickness of sublayers} then holds on the intersection of the event in
\lemref{Lemma_minoration de la taille des generations du bloc}
with the event in \eqref{Eq_number_blocks_in_halfslice}.
\end{proof}

\propref{Prop_f is linear} will be proved via the following functional inequality on $f_\ve$:
\begin{proposition}
\label{Prop_functional inequality on f}
Let $C \in (0,1/6)$ be as in \lemref{Lemma_claim on thickness of sublayers}. Then for all $h$ large enough, 
\begin{equation}
f_\ve(h) \geq C \max_{1\leq m \leq Ch} \min \left( m, \pfrac{h}{m}^2 f_\ve(3m) \right) .
\end{equation}
\end{proposition}

\begin{proof}

The idea is to cut the block into slices of height $3m$, and to consider separately the cases where the shortest path crossing the block from left to right stays inside such a slice or not. See \figref{Fig_Disjonction prop thickness} for an illustration.

\begin{figure}[h!]
\begin{center}
\includegraphics[width=1.0\textwidth]{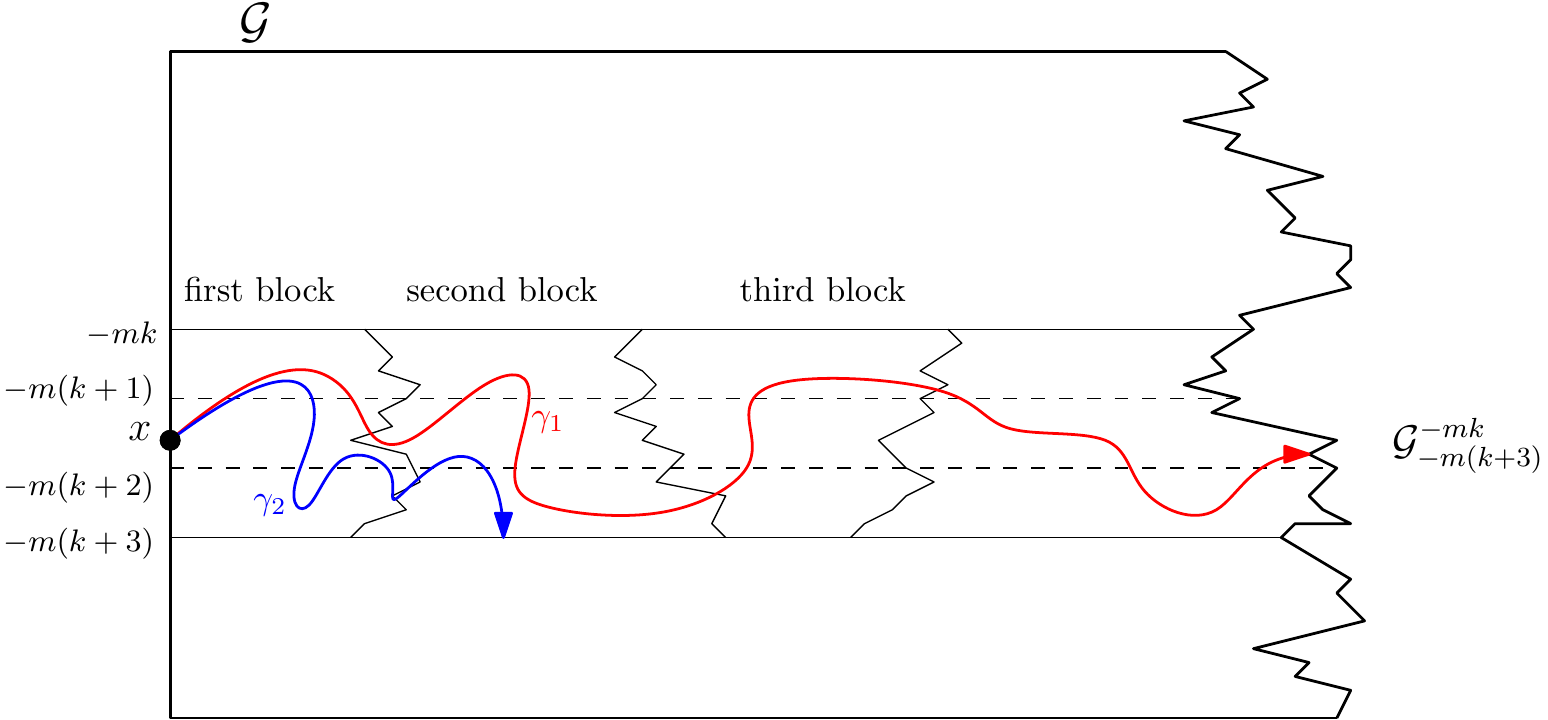}
\end{center}
\caption{Paths started from a point $x$ on the left boundary can either stay in a slice of height $3m$ around $x$ (and have length at least $C \pfrac{h}{m}^2 f_\ve(m)$ by \lemref{Lemma_claim on thickness of sublayers} \whp), or leave it and have length at least $m$ \as.}
\label{Fig_Disjonction prop thickness}
\end{figure}

 Let $m$ be an integer with $1\leq m\leq Ch$. Consider a path $\vg$ in $\GG(h)$ that achieves the thickness of this block, and let $x$ be its starting point on the left boundary. 
We assume for simplicity that $x$ is at distance at least $m$ from the top and bottom boundaries (the case where $x$ is at distance smaller than $m$ from the top or bottom boundaries is treated similarly). By our choice of $J_m$, there is always an index $j \in J_m$ such that $x$ is in the slice $\GG_{j-2m}^{j-m}(h)$. Then either $\vg$ leaves the slice $\GG_{j-3m}^{j}(h)$, which takes at least $m$ steps; or $\vg$ stays in $\GG_{j-3m}^{j}(h)$, but then by  \lemref{Lemma_claim on thickness of sublayers} its length is at least $C \pfrac{h}{m}^2 f_\ve(m)$ with probability at least $1-\ve$. 

We conclude that the thickness of $\GG(h)$ is at least $C \min \left( m, \pfrac{h}{m}^2 f_\ve(m) \right)$ with probability at least $1-\ve$. 
Hence $f_\ve(m)\geq C \min \left( m, \pfrac{h}{m}^2 f_\ve(m) \right)$. 
 Since this holds for all $m$ with $1\leq m\leq Ch$, this gives the result of \propref{Prop_functional inequality on f}.
\end{proof}

\begin{proof}[Proof of \propref{Prop_f is linear}]
First note that by taking $C$ smaller if necessary we may assume that the bound of Proposition \ref{Prop_functional inequality on f} holds for every $h\geq \floor{6/C^2}$. We then prove by induction that $f_\ve(h) \geq \frac{C^2}{6} h$ for every $h\geq 1$. If $h \leq \floor{6/C^2}$ this bound is trivial. So let $h_0 \geq \floor{6/C^2}$ and assume that $f_\ve(h) \geq \frac{C^2}{6} h$ for every $1\leq h \leq h_0$. 
Take $h=h_0+1$ and $m = \floor{C h/3}$. One verifies that $3m \leq Ch <h$, hence $3m\leq h_0$, so by our assumption $f_\ve(3m) \geq C^2 m/2$.

We note as well that $\frac{Ch}{3} > \frac{C}{3}\frac{6}{C^2} = \frac{2}{C} > 12$, so that $\floor{\frac{Ch}{3}} \geq \frac{Ch}{6}$, hence $m\geq Ch/6$. 

By \propref{Prop_functional inequality on f},
$$
f_\ve(h) \geq C \min\left( m , \pfrac{h}{m}^2 f_\ve(3m) \right) 
\geq C \min\left( \frac{Ch}{6} , \pfrac{3}{C}^2 \frac	{C^2}{2} \frac{Ch}{6} \right) 
$$
This completes the proof. \end{proof}

\subsection{Subadditivity}
\label{Sec_Subadditivity in the LHPQ}

It will be convenient to consider the map $\wt \LL$ which is derived from the \LHPQ $\LL$ be removing all ``horizontal edges'' $((j,0),(j+1,0))$ for $j\in\Z$. 
For every integer $j<0$, we also denote by $\wt\LL_{-\infty}^j$ the submap of $\wt \LL$ (or of $\LL$) contained in the half-plane below ordinate $j$.
If $-\infty < j < j' \leq 0$, $\wt \LL_j^{j'}$ is the submap of $\wt \LL$ contained in $\R\times [j,j']$. We equip the vertex sets of these graphs with the
first-passage percolation distance induced by \iid weights to the edges (the common distribution of these weights is supported on $[1,\vk]$). 
Recall our notation $\rho=(0,0)$ for the root vertex of $\LL$ (or of $\wt\LL$), and $\partial_j\LL$ for the line at vertical coordinate $j\leq 0$ (viewed here
as a collection of vertices).

\begin{proposition}
\label{Prop_P18 - ergodic}
There exists a constant $\bc_p \in [1, \vk]$ such that
\begin{align*}
r^{-1} \dfpp^{\wt \LL}(\rho, \partial_{-r}\LL) \overset{\as}{\ulim{r}{\infty}} \bc_p .
\end{align*}
\end{proposition}

\begin{proof}
We derive this proposition from the subadditive ergodic theorem.
Let $-\infty<j' < j < 0$, and let $x_j$ be the left-most vertex of $\partial_j \LL$ such that $\dfpp^{\wt\LL}(\rho, \partial_j \LL) = \dfpp^{\wt\LL}(\rho, x_j)$. Then,
\begin{equation*}
\dfpp^{\wt\LL_{j'}^0}(\rho, \partial_{j'} \LL) \leq \dfpp^{\wt\LL_{j}^0}(\rho, \partial_j \LL) + \dfpp^{\wt\LL_{j'}^j}(x_j, \partial_{j'} \LL) .
\end{equation*}
Note that $x_j$ is a function of $\wt\LL_j^0$ and of the weights on edges of $\wt\LL_j^0$. Thanks to the independence of layers of the map, $\dfpp^{\wt\LL_{j'}^j}(x_j, \partial_{j'} \LL)$ is independent of $\dfpp^{\wt\LL_j^0}(\rho, \partial_j \LL)$ and has the same distribution as $\dfpp^{\wt\LL_{j'-j}^0}(\rho, \partial_{j'-j} \LL)$. We then apply Liggett's version of Kingman's subadditive ergodic theorem \cite[theorem 1.10]{liggett1985improved} to conclude that
\begin{equation*}
r^{-1} \dfpp^{\wt\LL_{-r}^0}(\rho, \partial_{-r}\LL) \overset{\as}{\ulim{r}{\infty}} \bc_p
\end{equation*}
for some constant $\bc_p$. The fact that $\bc_p \in [1,\vk]$ is immediate since weights belong to $[1,\vk]$ and the graph distance from $\rho$ to $\partial_{-r} \LL$
(in $\wt\LL_{-r}^0$) is equal to $r$. The lemma follows by noting that $\dfpp^{\wt \LL}(\rho, \partial_j\LL) = \dfpp^{\wt\LL_{j}^0}(\rho, \partial_j\LL)$.

\end{proof}

\section{Technical tools}
\label{Sec-tech}

\subsection{Density between the LHPQ and truncated hulls of the UIPQ}
\label{Sec_Density between the LHPQ and truncated hulls of the UIPQ}

\propref{Prop_LHPQ_is_local_limit_truncated_hulls} suggests that the neighborhood of a vertex chosen uniformly on the boundary of a large hull in the \UIPQ looks like the \LHPQ. We will need a quantitative version of this property; this is provided by \propref{Prop_P5}, whose proof does not depend on \propref{Prop_LHPQ_is_local_limit_truncated_hulls}.

Let $a \in (0,1)$, let $(\tau_i)_{i \in \Z}$ be an \iid \BGW forest with offspring distribution $\vt$, and for every integer $r \geq 1$, let $N_r^{(a)}$ be a random variable distributed uniformly over $\{ \floor{ar^2}+1, ..., \floor{r^2/a} \}$, and independent of $(\tau_i)_{i \in \Z}$. We denote the tree $\tau_i$ truncated at height $r$ by $[\tau_i]_r$ (we only keep vertices at generation at most $r$). 
For every $0\leq r < s$, let $\FF_{r,s}$ be the forest defined from the skeleton of the annulus $\CC(r,s)$ in the \UIPQ as explained at the end of \secref{Sec_Preliminaries}.

\begin{proposition}
\label{Prop_P5}
For every $a\in(0,1)$, we can find $C_a>0$ such that for every large enough integer $r$, for every choice of the  integers $s>r$ and $\floor{ar^2}+1 \leq p,q \leq \floor{r^2/a}$, for every forest $\FF \in \F_{s-r,p,q}$,
\begin{equation}
\label{Eq_Prop6_temp}
\P\l( \FF_{r,s} = \FF \r) \leq C_a \P\l( ([\tau_1]_{s-r}, ... , [\tau_{N^{(a)}_r}]_{s-r}) = \FF \r) .
\end{equation}
\end{proposition}

\begin{proof}
By \propref{Prop_law_annulus_UIPQ_cond_bottom}, for $\FF \in \F_{s-r,p,q}$, 
\begin{equation}
\P(\FF_{s,r} = \FF \ | \ H_r = p) = \frac{p}{q} \frac{2^q \vk_q}{2^p \vk_p} \prod_{v \in \FF^*} \vt(c_v) .
\label{Eq_P5 squelette}
\end{equation}
where we recall that $\FF^*$ is the set of vertices at generation at most $s-r-1$ in the forest $\FF$.

Let us consider the right-hand side of \eqref{Eq_Prop6_temp}. Using the asymptotics \eqref{Eq_asymptotics_vkp}, we find $C>0$ that only depends on $a$ such that for every large enough $r$, for every $\floor{ar^2} < p,q < \floor{r^2/a}$, we have $\frac{p}{q} \frac{2^q \vk_q}{2^p \vk_p} \leq C$. On the other hand, \propref{Prop_law_perimeter_UIPQ} and \eqref{Eq_asymptotics_vkp} allow us to find $C'>0$ such that $\P(H_r = p) \leq C'/r^2 $ for every $r\geq 1$ and $p\geq 1$. From \eqref{Eq_P5 squelette}, we now get
\begin{equation}
\P\l( \FF_{r,s} = \FF \r) \leq \frac{C C'}{r^2} \prod_{v \in \FF^*} \vt(c_v) .
\label{Eq_P5 perimetre}
\end{equation}
On the other hand,
\begin{equation}
\P\l( ([\tau_1]_{s-r}, ... , [\tau_{N^{(a)}_r}]_{s-r}) = \FF \r) = \P(N^{(a)}_r = q) \P(([\tau_1]_{s-r}, ... , [\tau_q]_{s-r}) = \FF) 
= \frac{1}{\floor{r^2/a}-\floor{ar^2}} \prod_{v \in \FF^*} \vt(c_v) .
\label{Eq_P5 arbres iid}
\end{equation}
The desired result follows by comparing \eqref{Eq_P5 perimetre} and \eqref{Eq_P5 arbres iid}.
\end{proof}

\subsection{Coalescence of \lmgs in the UIPQ}
\label{Sec_Coalescence of geodesics in the UIPQ}

\Lmgs in the \UIPQ coalesce quickly, in the following sense. Consider the set of all \lmgs started from the boundary of the hull of radius $r\geq 1$, and let $\vg \in (0,1)$. Then the number of vertices at distance $\floor{\vg r}$ from the root that belong to one of these geodesics is bounded in distribution when $r$ is large. The next proposition (which is inspired from
\cite[Proposition 17]{fpp})
 gives a precise 
version of this property, which
will be particularly useful in the proof of \propref{Prop_P19 thin annuli} below.

Recall that $\Htr_{Q_\infty}(r)$ is the truncated hull of radius $r$ of the \UIPQ ${Q_\infty}$ and that 
$\partial_rQ_\infty$ is the external boundary of this hull, which has length $H_r$. Pick a vertex $u_0^{(r)}$ on $\partial_rQ_\infty$ uniformly at random, and write $u^{(r)}_0, u^{(r)}_1, ... u^{(r)}_{H_r-1}$ for all vertices of the boundary listed in clockwise order starting from $u_0^{(r)}$. We extend the definition of $u^{(r)}_j$ to all $j\in \Z$ by periodicity, so that
$u^{(r)}_j=u^{(r)}_{j+H_r}$ for every $j$.

\begin{proposition}
\label{Prop_P17 coalescence primal}
Let $\vg\in(0,1/2)$ and $\vd>0$. For every integer $A>0$, let $\XX_{r,A}$ be the event where any \lmg to the root starting from a vertex of $\partial_r Q_\infty$ coalesces before time $\floor{\vg r}$ with one of the \lmgs started from $u^{(r)}_{\floor{kr^2/A}}$, $0\leq k \leq \floor{\frac{A H_r}{r^2}}$. Then we can choose $A$ large enough such that, for every sufficiently large $r$,
\begin{equation*}
\P(\XX_{r,A}) \geq 1-\vd.
\end{equation*}
\end{proposition}

\begin{proof}

The vertices $u^{(r)}_{\floor{kr^2/A}}$, $0\leq k \leq \floor{\frac{A H_r}{r^2}}$ divide $\partial_r Q_\infty$ into a collection of ``intervals'' made of consecutive edges of the boundary. We call an interval bad if at least two trees of the skeleton of $\Htr_{Q_\infty}(r)$ rooted in this interval have height at least $\floor{\vg r}$  and good otherwise.

Now recall the observations made in Section \ref{Sec_Preliminaries} before discussing the law of the skeleton of the \UIPQ. It follows that, if an interval $S$ is good, then the \lmg started from any vertex of $S$ coalesces with one of the two \lmgs started from the endpoints of $S$. The proposition then reduces to proving that we can choose $A>0$ such that, for all $r$ large enough, the probability of having no bad interval is greater than $1-\vd$.

\begin{figure}[h!]
\begin{center}
\includegraphics[width=1.0\textwidth]{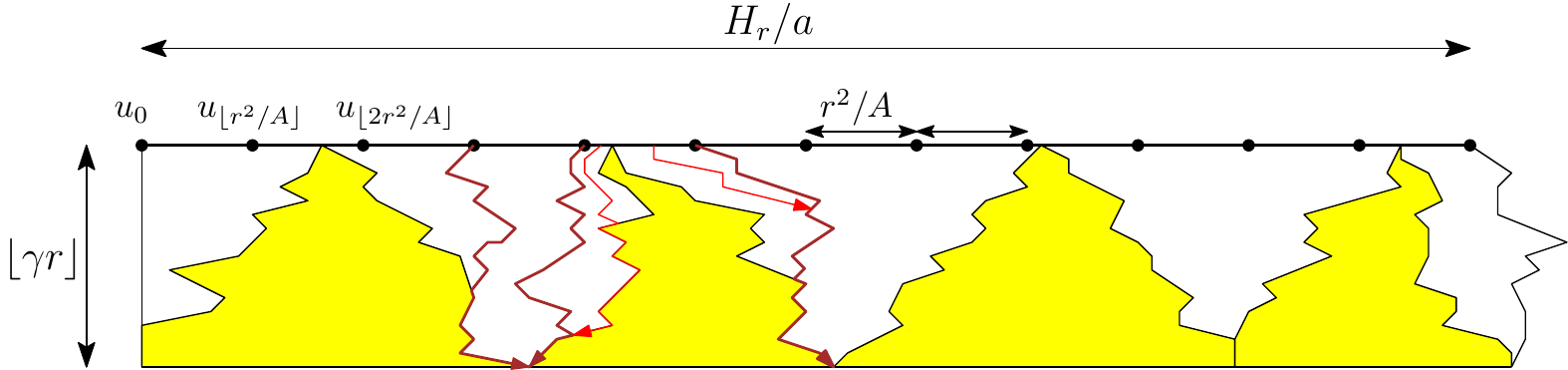}
\end{center}
\caption{As soon as there are no bad intervals, every \lmg started at a vertex $v$ of the top boundary coalesces before time $\vg r$ with the \lmg started at one of the endpoints of the interval that contains $v$. By choosing $A$ large enough, we ensure that with high probability there is no bad interval.}
\label{Fig_Prop_17, geodesics and poisson point process}
\end{figure}

From the explicit law of the perimeter of truncated hulls in \propref{Prop_law_perimeter_UIPQ}, we get that there exists $a\in(0,1)$ such that for all large enough $r$,
\begin{equation}
\label{Eq_coalescence_1}
\P\l( H_r \notin \l[ \floor{ar^2}+1 , \floor{r^2/a} \r] \hbox{ or }H_{\floor{\gamma r}} \notin \l[ \floor{ar^2}+1 , \floor{r^2/a} \r]  \r) < \vd/2.
\end{equation}

Consider first a forest made of $\floor{r^2/A}+1$ independent \BGW trees with offspring distribution $\vt$. Simple estimates show that the probability that at least two trees of the forest have height greater than or equal to $\floor{\vg r}$ is bounded by $C (A\vg)^{-2}$ (use \propref{Prop_law_perimeter_UIPQ}) independently of $r$. If we now consider $\floor{A/a}+1$ independent such forests, the probability that at least one of these forests satisfies the preceding property is bounded above by $C\l(\floor{A/a}+1\r)(A\vg)^{-2}$,
with a constant $C$ that does not depend on $r$ nor on $A$. 
By choosing $A$ large,
the latter quantity can be made smaller than $\vd/(2C_a)$, where $C_a$ is the constant in \propref{Prop_P5}. The proof is completed by using \propref{Prop_P5} and \eqref{Eq_coalescence_1}.
\end{proof}

\section{Main results for the first-passage percolation distance on quadrangulations}
\label{Sec-main-quad}

\subsection{Distance through a thin annulus}
\label{Sec_Distance through a thin annulus}

Recall the constant $\bc_p$ introduced in Proposition \ref{Prop_P18 - ergodic}. 

\begin{proposition}
\label{Prop_P19 thin annuli}
Let $\ve \in (0,1)$ and $\vd >0$. For every $\eta>0$ small enough, for all sufficiently large $n$, the property 
\begin{equation}
\label{Eq_control of distances through a thin annulus}
(1-\ve) \bc_p \eta n \leq  \dfpp^{Q_\infty}(v , \partial_{n-\floor{\eta n}} Q_\infty)  \leq (1+\ve) \bc_p \eta n 
\end{equation}
holds for every $v \in \partial_n Q_\infty$, with probability at least $1-\vd$. 
\end{proposition}

The proof of this result is technical but very similar to the proof of \cite[Proposition 19]{fpp}, to which we refer for additional details.
Let us start by an outline of the main ideas of the proof.
Recalling the absolute continuity relations stated in \propref{Prop_P5}, we observe that a sufficiently thin slice of the \UIPQ (of the form $\mathcal{C}(n-\floor{\eta n},n)$), seen from a uniformly chosen vertex of its outer boundary, looks like a slice of the \LHPQ. This in turn allows us to use \propref{Prop_P18 - ergodic}.

In order to implement the latter observation, we need to make sure that with high probability, distances from a point $v$ of the top boundary of the annulus $\CC(n-\floor{\eta n},n)$ to the bottom boundary are determined by a ``small'' neighborhood of $v$ in the annulus. This essentially follows from the control of distances along the boundary
discussed in Section \ref{Sec_Control of distances along the boundary of the LHPQ}.

Finally, we need \eqref{Eq_control of distances through a thin annulus} to hold simultaneously for all $v$ on the top boundary. \propref{Prop_P17 coalescence primal} ensures that with high probability, the \lmg started at a vertex $v$ of the top boundary coalesces quickly with one of the \lmgs started from a bounded number of points on the top boundary. Thanks to this
observation, it is enough to verify that \eqref{Eq_control of distances through a thin annulus} holds for a {\em bounded} number of vertices $v \in \partial_n Q_\infty$.

\begin{proof}[Proof of \propref{Prop_P19 thin annuli}]

In a way similar to Section \ref{Sec_Subadditivity in the LHPQ}, we let $\wt{\mathcal{H}}_{Q_\infty}(n)$ denote the map obtained from $\Htr_{Q_\infty}(n)$
by removing the edges of the external boundary. 
It is then convenient to write $\dgr^{(n)}$ for the graph distance on $\wt{\mathcal{H}}_{Q_\infty}(n)$, and similarly $\dfpp^{(n)}$
for the first-passage percolation distance on $\wt{\mathcal{H}}_{Q_\infty}(n)$ (in both cases we allow only paths made of edges of $\wt{\mathcal{H}}_{Q_\infty}(n)$).
Similarly as in the proof of \propref{Prop_P17 coalescence primal}, we pick a vertex $u^{(n)}_0$ uniformly at random on $\partial_n Q_\infty$, and denote the vertices of $\partial_n Q_\infty$ in clockwise order starting from $u^{(n)}_0$ by $(u^{(n)}_j)_{0 \leq j < H_n}$. We extend the definition of $u^{(n)}_j$ to $j \in \Z$ by periodicity. 
Let $\vd \in (0,1)$.

First, we use \propref{Prop_law_perimeter_UIPQ} to fix $a\in(0,1)$ small enough such that, for every $\eta\in(0,1/2)$, the top and bottom perimeters of the annulus $\CC(n-\floor{\eta n},n)$ are both within the range $[an^2, a^{-1}n^2]$ with probability at least $1-\vd/4$. In the remaining part of the proof, we implicitly argue on the event $\EE_\eta^{(n)}$
where the latter properties hold. We also set $N=\lceil 9a^{-2}\rceil$.

\propref{Prop_P5} allows us to bound the probability of any event concerning the forest encoding the skeleton of $\CC(n-\floor{\eta n},n)$ by a constant times the probability of the same event concerning an \iid forest of Bienaym\'e-Galton-Watson trees with offspring distribution $\vt$.
In particular, by taking $\eta$ small enough, one can ensure that the \lmgs started at $u^{(n)}_{-\floor{an^2/4}}$ and $u^{(n)}_{\floor{an^2/4}}$ do not coalesce before reaching $\partial_{n-\floor{\eta n}} Q_\infty$ except on a set of probability at most $\vd/(8N)$.  
On this event, the complement in the annulus $\CC(n-\floor{\eta n},n)$ of the union of the left-most geodesics started at $u^{(n)}_{-\floor{an^2/4}}$ and at $u^{(n)}_{\floor{an^2/4}}$ has two components, and we call $\GG^{(n)}_0$ the one containing the part of $\partial_n Q_\infty$ between $u^{(n)}_{-\floor{an^2/4}}$ and $u^{(n)}_{\floor{an^2/4}}$ in clockwise order. The lateral boundary $\partial^l \GG^{(n)}_0$ consists of the two \lmgs bounding $\GG^{(n)}_0$, and the bottom boundary $\partial^b \GG^{(n)}_0$ is defined in an obvious way.

Let us argue on the event where $\GG^{(n)}_0$ is well-defined.
Using \propref{Prop_P15} and \propref{Prop_P5}, and taking $\eta$ even smaller if necessary, we can ensure that the following holds except on a set of probability at most $\vd/(8N)$: any point $u^{(n)}_{k}$ with $|k| \leq an^2/16$ is at $\dgr^{(n)}$-distance at least $(4\vk+1)\eta n$ from $u^{(n)}_{-\floor{an^2/4}}$ and $u^{(n)}_{\floor{an^2/4}}$. By the triangle inequality, we thus obtain that on this event, the $\dgr^{(n)}$-distance between any point $u^{(n)}_{k}$ with $|k| \leq an^2/16$ and $\partial^l \GG^{(n)}_0$ is at least $4\vk \eta n$. Any path in the annulus 
$\CC(n-\floor{\eta n},n)$ with one endpoint in $\{u^{(n)}_{k} , |k| \leq an^2/16\}$ and the other one in $\partial_{n-\floor{\eta n}} Q_\infty$ that crosses $\partial ^l \GG^{(n)}_0$ will have length at least $4\eta \vk n$, and thus first-passage percolation weight at least $4\eta \vk n$. On the other hand, the \lmg started at any $u^{(n)}_{k} , |k| \leq an^2/16$, gives a path of length at most $\eta n$ between $u^{(n)}_{k}$ and $\partial_{n-\floor{\eta n}} Q_\infty$, that is thus of first-passage-percolation weight at most $\vk \eta n$. It follows that no $\dfpp^{(n)}$-shortest path between a vertex of the form $u^{(n)}_{k}$, $|k| \leq an^2/16$, and $\partial_{n-\floor{\eta n}} Q_\infty$ reaches $\partial ^l \GG^{(n)}_0$, except on an event of probability at most $\vd/(8N)$.

The previous considerations apply as well if we replace $u^{(n)}_0$ by $u^{(n)}_j$ for any $j$ (possibly depending on $n$). Let 
$\GG^{(n)}_j$ stand for the analog of $\GG^{(n)}_0$ when $u^{(n)}_0$ is replaced by $u^{(n)}_j$. We obtain that, except possibly on an event 
of probability at most $\vd/4$, for every
$j$ of the form $j=i\floor{ an^2/8 }$, $0\leq i\leq N-1$, the set $\GG^{(n)}_j$  is well-defined and for every
integer $k$ with $j-\floor{an^2/16} \leq k \leq j+\floor{an^2/16}$, any $\dfpp^{(n)}$-shortest path from $u^{(n)}_{k}$ to $\partial_{n-\floor{\eta n}} Q_\infty$ reaches the bottom boundary of $\GG^{(n)}_j$ before its lateral boundary. We write $\mathcal{D}^{(n)}_\eta$ for the event
of probability at least $1-\vd/4$ where the preceding properties hold.  On the intersection $\EE^{(n)}_\eta\cap \mathcal{D}^{(n)}_\eta$, 
for any choice of $j$ and $k$ as previously, the $\dfpp^{(n)}$-distance from $u^{(n)}_{k}$ to $\partial_{n-\floor{\eta n}} Q_\infty$ 
can be computed from 
the information given by $\GG^{(n)}_j$ and the weights on edges of $\GG^{(n)}_j$. From our choice of $N$,
we also see that the vertices $u^{(n)}_{k}$ with $j-\floor{an^2/16} \leq k \leq j+\floor{an^2/16}$ and $j$ of the form $j=i\floor{an^2/8}$, $0\leq i\leq N-1$, cover 
the whole boundary $\partial_nQ_\infty$ (provided  $\EE^{(n)}_\eta$ holds). 

At this stage, we use the absolute continuity relations in \propref{Prop_P5} in connection with \propref{Prop_P18 - ergodic}. Let $j$ and $k$ be as previously (possibly depending on $n$). 
On the event $\EE^{(n)}_\eta\cap \mathcal{D}^{(n)}_\eta$, the $\dfpp^{(n)}$-distance from $u^{(n)}_k$ to $\partial_{n-\floor{\eta n}} Q_\infty$ is determined as a function of the skeleton of $\GG^{(n)}_j$ (meaning the forest consisting of the trees of the skeleton of $\CC(n-\floor{\eta n}, n)$ rooted at edges between $u^{(n)}_{j-\floor{an^2/4}}$ and $u^{(n)}_{j+\floor{an^2/4}}$ in clockwise order) and the quadrangulations that fill in the slots --- and of course of the weights on edges. But the same function determines the first passage percolation distance in the \LHPQ (which is estimated by \propref{Prop_P18 - ergodic}) and one just has to compare the distributions of skeletons, for which one may use \propref{Prop_P5}. It follows that, on
the event $\EE^{(n)}_\eta\cap \mathcal{D}^{(n)}_\eta$, we have
\begin{equation}
\label{Eq_borne distances a travers anneau}
\dfpp^{(n)}(u^{(n)}_k, \partial_{n-\floor{\eta n}} Q_\infty) \ \in \ [(1-\ve/2)\bc_p \eta n, (1+\ve/2)\bc_p \eta n] .
\end{equation}
except possibly on an event of probability tending to $0$ as $n\to\infty$.

We now want to argue that \eqref{Eq_borne distances a travers anneau} holds simultaneously for all $k$ outside a set of small probability. To this end, we rely on the coalescence of geodesics (\propref{Prop_P17 coalescence primal}). 
 Let $A$ be chosen as in \propref{Prop_P17 coalescence primal}, replacing $\vg$ by $\bc_p \eta \ve/(4\vk)$ and $\vd$ by $\vd/4$. As in \propref{Prop_P17 coalescence primal}, consider  indices $k$ of the form $\floor{i n^2/A}$, $0 \leq i \leq \floor{A/a}$. Then, for $n$ large enough, \eqref{Eq_borne distances a travers anneau} holds simultaneously for 
 all these values of $k$, on the event $\EE^{(n)}_\eta\cap \mathcal{D}^{(n)}_\eta$, except possibly on event of probability less than $\vd/4$. 
Furthermore, thanks to \propref{Prop_P17 coalescence primal}, we know on the event $\EE^{(n)}_\eta\cap \mathcal{D}^{(n)}_\eta$ that, outside an event
of probability at most $\vd/4$, every vertex $v \in \partial_n Q_\infty$ is at $\dgr^{(n)}$-distance at most $\ve \bc_p \eta n / (2\vk)$ (thus at $\dfpp^{(n)}$-distance at most $\ve \bc_p \eta n/2$) from one of these vertices $u^{(n)}_k$. We conclude that we have $\dfpp^{(n)}(v, \partial_{n-\floor{\eta n}} Q_\infty) \in [(1-\ve)\bc_p \eta n, (1+\ve)\bc_p \eta n]$
for every vertex $v$ of $\partial_n Q_\infty$, outside an event of probability at most $\vd$. This is the desired result, except that we need to replace
$\dfpp^{(n)}$ by $\dfpp^{Q_\infty}$. This is however easy since on one hand $\dfpp^{Q_\infty}\leq \dfpp^{(n)}$ and on the other hand the minimal 
values of $\dfpp^{(n)}(v, \partial_{n-\floor{\eta n}} Q_\infty) $ and $\dfpp^{Q_\infty}(v, \partial_{n-\floor{\eta n}} Q_\infty)$ on $\partial_nQ_\infty$
are the same. This completes the proof.
\end{proof}

\subsection{Distance from the boundary of a hull to its center}
\label{Sec_Distance from the boundary of a hull to its center}

The next step is to show that the distance from the root vertex of the \UIPQ ${Q_\infty}$ to an 
arbitrary vertex of the boundary of a hull is close to a constant times the radius. Recall that $\rho$ is the root vertex of $Q_\infty$.

\begin{proposition}
\label{Prop_P20 distances root - boundary of the hull, UIPQ}
For every $\ve \in (0,1)$,
$$ \P\l( (\bc_p - \ve) n \leq \dfpp^{Q_\infty}(\rho, v) \leq (\bc_p + \ve) n,  \ \text{  for every } v \text{ in } \partial_n Q_\infty \r) \ulim{n}{\infty} 1  . $$
\end{proposition}

\begin{proof}
Fix $\ve \in (0,1)$, and take $\vd = \ve^2/(5\vk | \ln(\ve/(5\vk))| )$. For every $0<m<n$, we say that the annulus $\CC(m,n)$ is good if, for every $v \in \partial_n Q_\infty$,
\begin{equation}
(1-\ve/2) \bc_p \eta n \leq  \dfpp^{Q_\infty}(v , \partial_m Q_\infty) \leq (1+\ve/2) \bc_p \eta n ,
\end{equation}
and it is bad otherwise. 
\propref{Prop_P19 thin annuli} ensures that we can fix $\eta \in (0,1)$ small enough such that for all $n$ large enough, the annulus $\CC(n-\floor{\eta n},n)$ is good with probability at least $1-\vd$. 

Let $n_0 = n$, and define by induction $n_{k+1} = n_k - \floor{\eta n_k}$ for every $k \geq 0$. Let $q = \floor{\frac{\ln(\ve/(5\vk))}{\ln(1-\eta)}}$. Note that $n_q \geq \frac{\ve}{5\vk} n$ and $n_q \leq \frac{\ve}{4\vk}n$ for $n$ large. 
 Using Markov's inequality we get that for $n$ large enough,
\begin{equation*}
\P\l( \# \{ k \in \{0,1,…, q-1\} \ : \ \CC(n_{k+1},n_k) \text{ is bad} \} \ > \ \frac{\ve}{5\vk |\ln(1-\eta)|} \r) \leq \frac{5\vk |\ln(1-\eta)|}{\ve} q \vd \leq \ve 
\end{equation*}
by our choice of $\vd$ and $q$. Let $\mathcal{D}^\ve_n$ denote the event whose probability appears in the previous display. We will show that the property $(\bc_p-\ve)n \leq \dfpp^{Q_\infty}(\rho, v) \leq (\bc_p + \ve)n$ for every $v\in \partial_n Q_\infty$ holds on the complement of $\mathcal{D}^\ve_n$. Since $\P(\mathcal{D}^\ve_n)\leq \ve$
for $n$ large, this will complete the proof.

Suppose that $\mathcal{D}^\ve_n$ does not hold. Then the \fpp-distance between any vertex $v \in \partial_n Q_\infty$ and $\rho$ is larger than the cost one must pay to cross the good annuli, that is 
\begin{align*}
\dfpp^{Q_\infty}(v,\rho) &\geq \sum_{k=0}^{q-1} (1-\ve/2)\bc_p (n_k-n_{k+1}) - \sum_{k \ : \ \CC(n_{k+1},n_k) \text{ bad}} (1-\ve/2)\bc_p (n_k-n_{k+1}) \\
&\geq (1-\ve/2)\bc_p (n_0-n_{q}) - \frac{\ve}{5\vk|\ln(1-\eta)|}\bc_p \max_{0\leq k < q} (n_k-n_{k+1}) \\
&\geq \bc_p n \l[ 1 - \frac \ve 2 -\frac{\ve}{4} - \frac{\ve}{4|\ln(1-\eta)|} \eta \r] \\
&\geq \bc_p n \l( 1 - \ve \r) ,
\end{align*}
using the properties $\vk \geq 1$ and $\eta/|\ln(1-\eta)| < 1$ for $\eta \in (0,1)$. 
Conversely, we can build a path from $v$ to $\rho$ that crosses the good annuli in the $\dfpp$-shortest way, and the bad annuli or the hull $\Htr_{Q_\infty}(n_q)$ along \lmgs. Using the properties $\dgr \leq \dfpp \leq \vk \dgr$ and $1 \leq \bc_p\leq \vk$, we can bound its \fpp-weight by
\begin{align*}
&\sum_{k=0}^{q-1} (1+\ve/2)\bc_p (n_k-n_{k+1}) + \vk n_q + \sum_{k \ : \ \CC(n_{k+1},n_k) \text{ bad}} \vk (n_k-n_{k+1})   \\
&\leq (1+\ve/2)\bc_p (n_0-n_{q}) + \vk \frac{\ve}{4\vk}n + \frac{\ve}{5\vk|\ln(1-\eta)|} \vk\max_{0\leq k < q} (n_k-n_{k+1}) \\
&\leq \bc_p n \l[ 1 + \frac \ve 2 + \frac{\ve}{4} + \frac{\ve}{4|\ln(1-\eta)|} \eta \r] \\
&\leq \bc_p n \l( 1 + \ve \r) ,
\end{align*}
giving $\dfpp^{Q_\infty}(v,\rho) \leq \bc_p n \l( 1 + \ve \r)$. This completes the proof.
\end{proof}

\subsection{Distance between two uniform points in finite quadrangulations}

\label{Sec_Distance between two uniformly sampled points in finite quadrangulations}

In this section, we consider a uniformly distributed rooted and pointed quadrangulation with $n$ faces, which we denote by $Q^\bullet_n$. The associated (unpointed)
rooted quadrangulation is simply denoted by $Q_n$. We will write $\rho_n$ for the root vertex of $Q_n$ (or of $Q^\bullet_n$) and $\partial_n$ for the distinguished 
vertex of $Q^\bullet_n$. This notation will be in force throughout the remaining part of this work. We note that, conditionally on the unpointed map $Q_n$, $\partial_n$ is uniformly distributed over $V(Q_n)$.

Our next goal is to control the \fpp-distance between $\rho_n$ and $\partial_n$. 

\begin{proposition}
\label{Prop_P21 distances root - uniform point, Qn}
For every $\ve \in (0,1)$,
\begin{equation*}
\P\l( | \dfpp^{Q_n}(\rho_n, \partial_n) - \bc_p \dgr^{Q_n}(\rho_n, \partial_n) | > \ve n^{1/4} \r) \ulim n \infty 0 .
\end{equation*}
\end{proposition}

We postpone the proof of this result to \secref{Sec_Proof_P21}, and first give some technical tools that are needed in this proof. 

The idea is to transfer the results that we established in the \UIPQ to the setting of finite quadrangulations. 
The core tool that we establish in this section compares the law of a neighborhood of the root in the \UIPQ and in a finite quadrangulation. 
In this direction, 
\propref{Prop_TBP 9} gives a result valid for neighborhoods of diameter smaller than a constant times the typical diameter of the quadrangulation.
This is closely related to \cite[Lemma 8 and Proposition 9]{curienLeGall2014brownian}, but we need sharper results.

Let us briefly introduce the objects we need. Our proofs in this section and the next one make use of the (now classical)
Cori-Vauquelin-Schaeffer bijection \cite[Section 5.4]{gall2012buziosf} between rooted and pointed quadrangulations with $n$ faces, and labeled rooted plane trees with $n$ edges. 
For more details, we refer to \cite[Section 4]{curienLeGall2014brownian}.

\paragraph{The Cori-Vauquelin-Schaeffer correspondence.}

 Consider a rooted plane tree $\tau$, with root vertex $\vs$, together with a labeling $Z=(Z_x)_{x \in V(\tau)}$ of its vertices by integers such that $Z_\vs=0$, and $|Z_x-Z_y|\leq 1$ if $x$ and $y$ are adjacent. We fix $\epsilon \in \{0,1\}$ and explain how to get a pointed and rooted quadrangulation $Q$ from $(\tau,Z,\epsilon)$. To this end,
 we suppose that $\tau$ is embedded in the plane as shown on \figref{Fig_Schaeffer}. Then, firstly, we add a vertex $\partial$, and extend the labeling to $\partial$ so that $Z_\partial = -1+\min_{x \in V(\tau)} Z_x$. The vertex set of the quadrangulation $Q$ is $V(\tau) \cup \{ \partial \}$, and $\partial$ is its distinguished vertex.
We also extend the labeling to corners of $\tau$, by declaring that the label of a corner is the label of the incident vertex of $\tau$. Secondly, we order the corners of $\tau$ in clockwise order around $\tau$, starting from the bottom corner of $\vs$. For every $n>Z_\partial+1$, we draw an edge of $Q$ from each corner labeled $n$ to the first next corner labeled $n-1$; and for each corner of index $Z_\partial+1$, we draw an edge from this corner to $\partial$. This defines the edges of $Q$. 
Finally, we root $Q$ at the edge drawn from the bottom corner of $\vs$ and use $\epsilon$ to determine its orientation: the root vertex is $\vs$ iff $\epsilon=1$. The construction should be clear from \figref{Fig_Schaeffer}. We mention the following important property relating labels on $\tau$ to distances from $\partial$ in $Q$: For every
vertex $v\in V(Q)$, $\dgr^Q(\partial,v)=Z_v-Z_\partial$.

\begin{figure}[h!]
\begin{center}
\includegraphics[width=0.7\textwidth]{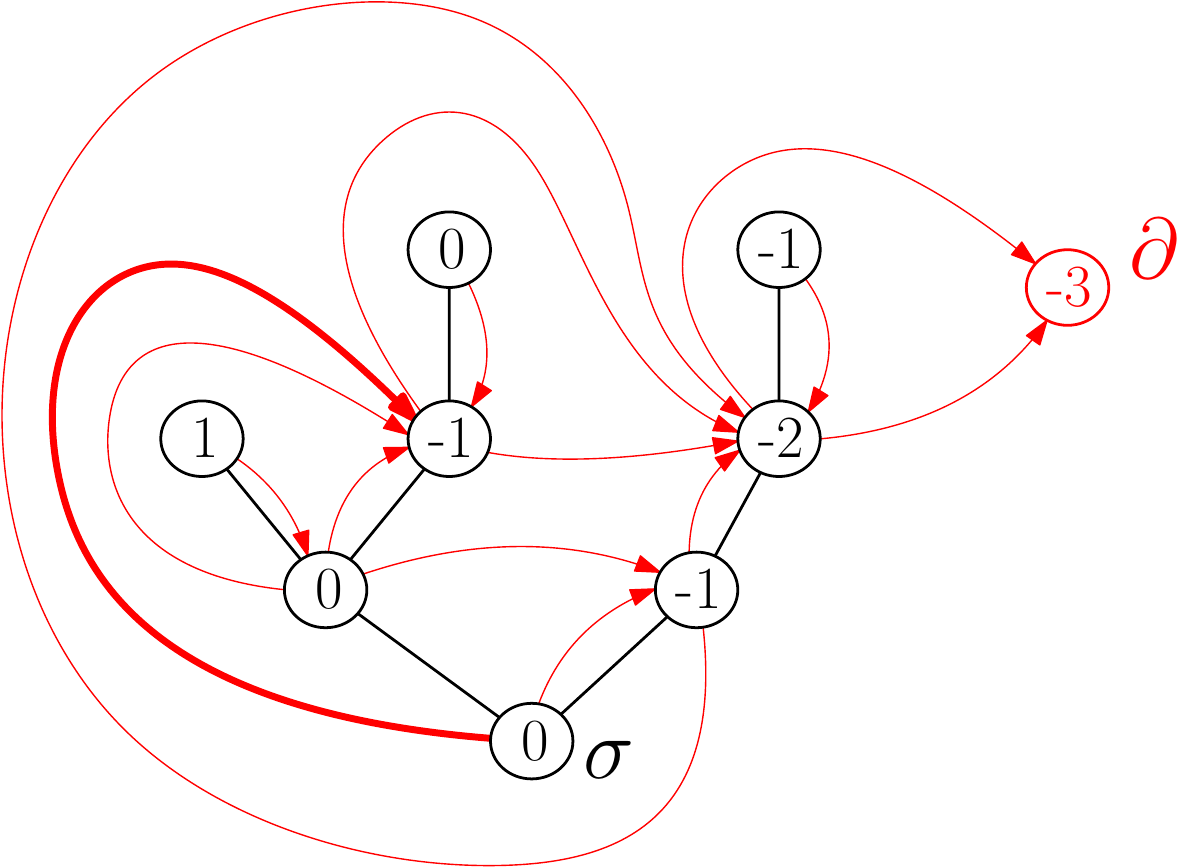}
\end{center}
\caption{The \CVS applied to a finite labeled tree (in black). The thick red edge is the root edge of the quadrangulation, and its orientation is determined by $\epsilon$ (here $\epsilon = 1$).}
\label{Fig_Schaeffer}
\end{figure}

The \CVS allows us to code uniform rooted and pointed quadrangulations by uniform rooted labeled plane trees. More precisely, let $T_n$ be a uniform rooted plane tree with $n$ edges. Given $T_n$, assign \iid weights on its edges, with uniform law over $\{ -1, 0, +1 \}$. For every $x \in V(T_n)$, define the label $Z^n_x$ as the sum of the weights of edges along the geodesic from the root to $x$ in $T_n$. Pick $\epsilon\in\{0,1\}$ uniformly at random. The \CVS applied to $(T_n, Z^n, \epsilon)$ then gives a uniform rooted and pointed quadrangulation with $n$ faces.

\paragraph{Pruned trees.}

 Let $\tau$ be a (finite) rooted plane tree. For every vertex $v$ of $\tau$ and for every $h>0$ such that $\dgr^{\tau}(\vs,v)\geq \floor{h}$, we denote the ancestor of $v$ at height $\floor h$ in $\tau$ by $[v]_h$.  We construct the \emph{pruned tree} $\PP((\tau,v),h)$ by removing all strict descendants of $[v]_h$ in $\tau$ (see \cite[Fig. 5]{curienLeGall2014brownian}) and we see $\PP((\tau,v),h)$ as a rooted plane tree pointed at $[v]_h$. 

Note that the vertex set of $\PP((\tau,v),h)$ is a subset of all vertices of $\tau$. It follows that, if $\tau$ is labeled by $Z$, we can construct a labeling of $\PP((\tau,v),h)$ by restricting $Z$ to $V(\PP((\tau,v),h))$. We will do so implicitely, keeping the same notation for the labeling on $\tau$ and on the pruned tree.

The case where $\tau$ is infinite is similar. In that case, we always assume that the tree has 
one end: there is only one infinite injective path started at its root, called the spine. We use the notation $[\infty]_h$ for the unique vertex of the spine at distance $h$ from the root, and remove its strict descendants in $\tau$ to get the pruned tree $\PP(\tau,h)$. Informally, $\infty$ plays the role of the distinguished vertex.

We now state the first result of this section. 
Let $(\tau, Z)$ be a finite rooted labeled plane tree with root vertex $\vs$. 
 Let $\epsilon \in \{0, 1\}$, and let $Q$ be the rooted and pointed quadrangulation constructed from $(\tau, Z, \epsilon)$ via the \CVS (we denote the root vertex of $Q$ by $\rho$). Let $\xi \in V(\tau)$, and let $h$ be an integer such that $0<h<\dgr^\tau(\vs,\xi)$. Set
\begin{equation*}
r = - \min_{0\leq i \leq h} Z_{[\xi]_i} \geq 0.
\end{equation*}

\begin{proposition}
\label{Prop_TBP 8}
Assume that $r \geq 4$. Then
$\Hull_Q(r-3)$ is a function of the pruned tree $\PP((\tau,\xi),h)$, its labeling, and $\epsilon$.
\end{proposition}

\begin{proof}

The idea is to prove that, if $Q'$ is the quadrangulation obtained by applying the \CVS to the pruned tree, then $\Hull_Q(r-3) = \Hull_{Q'}(r-3)$ (considered as an equality between rooted quadrangulations with a boundary).

Let us state some useful observations. 
Without loss of generality, by taking $h$ smaller we can assume that $Z_{[\xi]_h} = \min_{0\leq i \leq h} Z_{[\xi]_i} = -r$. 
We note that $\dgr^Q(\vs, \rho) = 0$ or $1$ depending on $\epsilon$. If $v \in V(\tau)$,
\begin{equation*}
\dgr^Q(\vs,v) \geq |Z_\vs - Z_v| = |Z_v| ,
\end{equation*}
and by the triangle inequality $\dgr^Q(\rho,v) \geq |Z_v|-1$.

Our first step is to prove that vertices of $\Hull_Q(r-3)$ ``belong'' to the pruned tree, and that their labels are at least $-r+1$.

Let $v$ be a vertex of $\tau$ such that  $Z_v \leq -r$. 
 Starting from any corner of $v$, the construction of edges in the \CVS yields a path starting from $v$ that visits vertices with strictly decreasing labels. This path ultimately connects $v$ to $\partial$ by visiting only vertices with labels less than $-r$, thus at distance at least $r-1$ from $\rho$ in $Q$. 
 By construction, any vertex of $\Hull_Q(r-3)$ is such that any path from this vertex to $\partial$ visits a vertex $w$ with $\dgr^Q(\rho,w) \leq r-2$. It follows that $v$ does not belong to $\Hull_Q(r-3)$.

Let us now check that vertices in $V(\tau) \setminus V(\PP((\tau,\xi),h))$ do not belong to $\Hull_Q(r-3)$ either. Let $v$ be such a vertex
with $Z_v>-r$ (the case $Z_v\leq -r$ was already considered above). Then $[\xi]_h$ is an ancestor of $v$, and the cactus bound \cite[Proposition 5.9 (ii)]{gall2012buziosf} shows that
\begin{equation*}
\dgr^{Q}(\vs,v) \geq Z_\vs - Z_{[\xi]_h} = r ,
\end{equation*}
and thus $\dgr^{Q}(\rho,v)\geq r-1$. 

Let $c$ be any corner of $v$. Order the corners of $\tau$ in clockwise order starting at $c$, and let $c''$ be the first corner of $[\xi]_h$ that appears in this enumeration: every corner between $c$ and $c''$ is incident to a vertex of $V(\tau) \setminus V(\PP((\tau,\xi),h))$. Since labels change by at most $1$ in this enumeration, there will be a corner $c'$ with label $Z_v-1$ between $c$ and $c''$ (possibly $c' = c''$). This ensures that the edge drawn from $c$ in the \CVS ends at a vertex in $V(\tau) \setminus V(\PP((\tau,\xi),h))$ with label $Z_v-1$, or possibly at $[\xi]_h$. Therefore, we can inductively construct a path from $v$ that stays in $V(\tau) \setminus V(\PP((\tau,\xi),h))$ until it reaches a vertex $w$ of label $-r$, and we extend this path to a path from $v$ to $\partial$ as we did previously. Every vertex of this path is at distance at least $r-1$ from $\rho$, so it follows again that $v$ does not belong to $\Hull_Q(r-3)$.

Note that by the same reasoning, if $c$ is a corner of label at least $-r+1$ that belongs to the pruned tree, then the edge drawn from $c$ will reach a vertex of the pruned tree (possibly $[\xi]_h$), and thus will be present in $Q'$ as well as in $Q$.

As our second and last step, we now verify that $\Hull_Q(r-3) = \Hull_{Q'}(r-3)$. By the first step, any edge belonging to $\Hull_Q(r-3)$ is drawn from a corner $c$ of $\tau$ (not incident to $[\xi]_h$) with label at least $-r+1$ that ``belongs'' to the pruned tree, so it appears both in $Q$ and $Q'$. 
It follows that we have both
\begin{align*}
&\dgr^{Q'}(\rho,v) \leq \dgr^Q(\rho,v) \qquad \text{ if } v \in \Hull_Q(r-3), \\
&\dgr^Q(\rho,v) \leq \dgr^{Q'}(\rho,v) \qquad \text{ if } v \in \Hull_{Q'}(r-3) .
\end{align*}
Next, any edge of the boundary of $\Hull_Q(r-3)$ is incident to a face of $Q$ containing a vertex at $\dgr^Q$-distance (hence $\dgr^{Q'}$-distance) less than or equal to $r-4$ from $\rho$. This face must then also be contained in $\Hull_{Q'}(r-3)$. It easily follows that $\Hull_Q(r-3) \subset \Hull_{Q'}(r-3)$, and the converse is also true by a symmetric argument.
\end{proof}

The \CVS can be extended to infinite labeled trees. Precisely, we consider the set $\SS$ of all 
infinite rooted labeled trees with one end, such that the infimum of the labels on the spine is $-\infty$. With every 
$(\tau,Z)\in \SS$ and $\ve\in\{0,1\}$, one can associate an infinite planar quadrangulation $Q$, which is defined via a direct extension of the rules
of the \CVS (see  \cite[Proposition 2.5]{curien2012view}). Furthermore, the preceding proposition is immediately extended to
that case, with the same proof: for every integer $h>0$, if $r:=-\min\{Z_{[\infty]_j}:0\leq j\leq h\}\geq 4$, the hull
$\Hull_Q(r-3)$ only depends on the pruned tree $\PP(\tau,h)$, its labeling, and $\epsilon$.

The following proposition is closely related to \cite[Proposition 9]{curienLeGall2014brownian}, but deals with hulls instead of balls. 
Recall that $Q^\bullet_n$ is uniformly distributed over the set of rooted and pointed quadrangulations with $n$ faces. 

\begin{proposition}
\label{Prop_TBP 9}
For every $\ve>0$, there exists $\chi>0$ such that for every sufficiently large $n$, we can construct $Q^\bullet_n$ and ${Q_\infty}$ on the same probability space in such a way that the equality
\begin{equation*}
\Hull_{Q^\bullet_n}(\chi n^{1/4}) = \Hull_{{Q_\infty}}(\chi n^{1/4})
\end{equation*}
holds with probability at least $1-\ve$.
\end{proposition}

\begin{proof}
The proof is is very similar to that of \cite[Proposition 9]{curienLeGall2014brownian}, using our Proposition \ref{Prop_TBP 8} instead of \cite[Proposition 8]{curienLeGall2014brownian}.
Let us only outline the argument. We may assume that $Q^\bullet_n$ is obtained via the \CVS from a uniform labeled tree with $n$ edges $(T_n,Z^n)$, and
we consider a uniformly distributed vertex $\xi_n$ of $T_n$. From \cite[Theorem 2.8]{curien2012view}, we also get that $Q_\infty$ can be constructed as the image under the
extended \CVS of the so-called uniform infinite labeled tree  $(T_\infty,Z^\infty)$ (see \cite[Definition 2.6]{curien2012view} for the definition of the latter object).
Using Proposition \ref{Prop_TBP 8} (and its analog in the infinite case), the desired result follows once we know that we can couple the labeled trees $(T_n, Z^n)$ and $(T_\infty,Z^\infty)$ 
so that the (labeled) pruned trees $\PP((T_n,\xi_n),\chi n^{1/4}+3)$ and $\PP(T_\infty,\chi n^{1/4}+3)$ are equal with probability at least $1-\ve$. 
We refer to the proof of \cite[Proposition 9]{curienLeGall2014brownian} for additional details. \end{proof}

\subsection{Proof of \propref{Prop_P21 distances root - uniform point, Qn}}
\label{Sec_Proof_P21}

Recall the notation of \secref{Sec_Distance between two uniformly sampled points in finite quadrangulations}. 
In particular, $Q^\bullet_n$ is a uniformly distributed rooted and pointed quadrangulation with $n$ faces. We denote its root vertex by $\rho_n$, and its distinguished vertex by $\partial_n$. The triplet associated with $Q^\bullet_n$ via the \CVS is denoted by $(T_n,Z^n,\epsilon_n)$. We will also use the uniform infinite labeled tree  $(T_\infty,Z^\infty)$
(cf. \cite[Definition 2.6]{curien2012view}) from which one constructs the \UIPQ $Q_\infty$ via the extended \CVS.

\subsubsection{First step: Pruning finite trees and infinite trees}

For any rooted plane tree $\tau$, any vertex $v \in V(\tau)$, and $h \leq \dgr^\tau(\vs,v)$, we set
\begin{equation*}
\vT(\tau,v,h)= \# V(\tau) - \# V(\PP((\tau,v),h)) ,
\end{equation*}
which is the number of vertices that are removed from the tree when pruning it. 
Denote the first (in lexicographical order) vertex with minimal label in $T_n$ by $\eta_n$, and
consider the pointed tree $(T_n,\eta_n)$. 
Let $\vb >0$ and $b\in(0,1)$. We claim that we can find a constant $C$ that depends only on $b$ such that, for all large enough $n$, for every nonnegative function $F$ on the space of all rooted and pointed labeled plane trees,
\begin{equation*}
\E\l[ F\l(\PP((T_n,\eta_n),\floor{ \vb\sqrt{n} })\r)  \ \ind{\dgr^{T_n}(\rho_n,\eta_n)>\vb\sqrt{n} ,  \ \vT(T_n,\eta_n,\floor{ \vb\sqrt{n} }\geq b (n+1)}\r]
\leq C \  \E\l[ F\l(\PP(T_\infty,\floor{ \vb\sqrt{n} })\r)\r].
\end{equation*}
In the previous display, we slightly abuse notation by viewing both $\PP((T_n,\eta_n),\floor{ \vb\sqrt{n} }) $ and $\PP(T_\infty,\floor{ \vb\sqrt{n} })$
as {\em labeled} trees --- we obviously keep the labels of the original trees.

\begin{proof}[Proof of the claim]
 To simplify notation, we write $A_n$ for the event
\begin{equation*}
A_n := \Big\{\dgr^{T_n}(\rho_n,\eta_n)>\vb\sqrt{n} \ , \ 
\vT(T_n,\eta_n,\floor{\vb\sqrt{n}}) \geq b(n+1) \Big\}.
\end{equation*}
Let $\xi_n$ be uniformly distributed over $V(T_n)$. We note that on the event $A_n$, the conditional probability that $\xi_n$ is not in $V(\PP((T_n,\eta_n),\floor{\vb\sqrt{n}}))$ knowing $T_n$ is at least $b$, because on $A_n$,
\begin{equation*}
\# V(T_n) - \# V(\PP((T_n,\eta_n),\floor{\vb\sqrt{n}}))\geq b(n+1).
\end{equation*}
It follows that
\begin{equation}
\label{Eq_density for pruned trees 1}
\E\l[ F\l(\PP((T_n,\eta_n),\floor{\vb\sqrt{n}}\r)\mathds{1}_{A_n}\r] 
\leq \frac{1}{b} \ \E\l[F\l(\PP((T_n,\eta_n),\floor{\vb\sqrt{n}}\r)\mathds{1}_{A_n}\ind{\xi_n\notin V(\PP((T_n,\eta_n),\floor{\vb\sqrt{n}}))}\r]
\end{equation}
If $\xi_n\notin V(\PP((T_n,\eta_n),\floor{\vb\sqrt{n}}))$, we have $[\xi_n]_{\vb\sqrt{n}}=[\eta_n]_{\vb\sqrt{n}}$ and
\begin{equation*}
\PP\l((T_n,\xi_n),\floor{\vb\sqrt{n}}\r) =\PP\l((T_n,\eta_n),\floor{\vb\sqrt{n}}\r).
\end{equation*}
It follows that the expectation in the right-hand side of \eqref{Eq_density for pruned trees 1} is bounded above by
\begin{equation}
\label{Eq_pruned tree, upper bound of expectation}
\E\l[F\l(\PP((T_n,\xi_n),\floor{\vb\sqrt{n}}\r)\mathds{1}_{A'_n}\r]
\end{equation}
where
$
A'_n := \{\dgr^{T_n}(\rho_n,\xi_n)>\vb\sqrt{n}, \ \vT(T_n,\xi_n,\floor{\vb\sqrt{n}}\geq b(n+1)\}.
$

Next let $\tau$ be a rooted and pointed labeled plane tree such that $|\tau|<n$ (here $|\tau|$ denotes the number of edges of $\tau$) and the distinguished vertex is at generation $\floor{\vb\sqrt{n}}$ and has no strict descendants. Formulas (19) and (21) in \cite{curienLeGall2014brownian} show that, for $n$ large enough (independently of the choice of $\tau$),
\begin{equation*}
\frac{\P(\PP((T_n,\xi_n),\floor{\vb\sqrt{n}}))=\tau)}
{\P(\PP(T_\infty,\floor{\vb\sqrt{n}}=\tau)} \leq 2 \ \l(1-\frac{|\tau|}{n}\r)^{-1/2}.
\end{equation*}
For $n$ large enough, the condition $\vT(T_n,\xi_n,\floor{\vb\sqrt{n}}\geq b(n+1)$ ensures that $\PP((T_n,\xi_n),\floor{\vb\sqrt{n}}))$ has less than $(1-\frac{b}{2})n$ edges. It follows that the quantity in \eqref{Eq_pruned tree, upper bound of expectation} is bounded above by
\begin{equation*}
2 \ \pfrac{b}{2}^{-1/2} \ \E\l[ F\l(\PP(T_\infty,\floor{\vb\sqrt{n}}\r)\r].
\end{equation*}
This completes the proof of the claim.
 \end{proof}

\subsubsection{Second step: Hulls in finite quadrangulations and in the UIPQ}

Recall that, for every integer $r\geq 1$, the hull $\Hull_{Q^\bullet_n}(r)$ is well defined under the condition $\dgr^{Q^\bullet_n}(\rho_n,\partial_n)>r+1$.

Let $\va>0$, and set $\van=\floor{ \va \ n^{1/4}}-1$ to simplify notation. Recall the notation $A_n$ introduced in the first step above, and set
\begin{equation*}
E_n=A_n\cap \{Z^n_{[\eta_n]_{\vb\sqrt{n}}}<-\va \ n^{1/4}\} ,
\end{equation*}
We note that (on $A_n$) the condition $Z^n_{[\eta_n]_{\vb\sqrt{n}}}<-\floor{ \va \ n^{1/4}}$ implies a fortiori $Z^n_{\eta_n} \leq Z^n_{[\eta_n]_{\vb\sqrt{n}}}-1 \leq -\van-3$ and $\dgr^{Q_n}(\rho_n,\eta_n) \geq \van+2$ so that $\dgr^{Q_n}(\rho_n,\partial_n) \geq \van+3$. in particular, the hull $\Hull_{Q^\bullet_n}(\van)$ is well defined on the event $E_n$.

By \propref{Prop_TBP 8}, on the event $E_n$, the hull $\Hull_{Q^\bullet_n}(\van)$ is equal to a deterministic function of the pruned tree $\PP((T_n,\eta_n),\floor{\vb\sqrt{n}})$ (and labels of this tree and $\epsilon_n$). Furthermore on the event $\{Z^\infty_{[\infty]_{\vb\sqrt{n}} }< -\van\}$, the hull $\Hull_{Q_\infty}(\van)$ is equal to the same deterministic function of $\PP(T_\infty,\floor{\vb\sqrt{n}})$ (and its labels and $\epsilon_\infty$).  
As a consequence of this fact and the first step, we have also, for every nonnegative function $G$ on the space of rooted planar maps,
\begin{equation}
\label{Eq_density for hulls in finite and infinite case}
\E\l[G(\Hull_{Q^\bullet_n}(\van)) \ \mathds{1}_{E_n} \r]
\leq C \ \E \Big[G(\Hull_{Q_\infty}(\van)) \ \mathds{1}_{\{Z^\infty_{[\infty]_{\vb\sqrt{n}} }< -\va\,n^{1/4}\}} \Big].
\end{equation}

\subsubsection{Final step}

Let $\vd>0$ to be fixed later, and for every integers $j,l\geq 1$ set
\begin{align*}
&\va_j=j\vd^2,\quad \va'_j=(j+1)\vd^2,\quad \va''_j=(j+2)\vd^2,\\
&\vb_l=l \vd^5,\quad \vb'_l=(l+1)\vd^{5},\quad\vb''_l=(l+2)\vd^5.
\end{align*}

\begin{lemma}
\label{Lemma_technical slicing}
For every integers $j,l\geq 1$, set
\begin{align*}
H_{n,\vd}^{j,l}=& \ \{Z^n_{\eta_n}\in[-\va''_jn^{1/4},-\va'_j n^{1/4}), \ Z^n_{[\eta_n]_{\vb_l\sqrt{n}}}<-\va_j \ n^{1/4}\}\\
&\cap \{\vb'_l\sqrt{n}\leq \dgr^{T_n}(\rho_n,\eta_n)\leq \vb''_l\sqrt{n}\}\cap 
\{\vT(T_n,\eta_n,\floor{ \vb_l\sqrt{n}}>\vd^{11}n\}.
\end{align*}
Let $\ve >0$. For all $\vd\in(0,1)$ small enough, for all large enough $n$, the event
\begin{equation*}
H_{n,\vd}:= \bigcup_{j=\floor{\vd^{-1}}}^{\floor{\vd^{-3}}}\bigcup_{l=\floor{\vd^{-4}}}^{\floor{\vd^{-6}}} H_{n,\vd}^{j,l}
\end{equation*}
has probability at least $1-\ve$. 
\end{lemma}

Let us postpone the proof of this lemma and complete that of \propref{Prop_P21 distances root - uniform point, Qn}.

Each set $H_{n,\vd}^{j,l}$ is contained in a set of the type $E_n$ (with $\va=\va_j$, $\vb=\vb_l$ and $b=\vd^{11}$). Using \eqref{Eq_density for hulls in finite and infinite case} and \propref{Prop_P20 distances root - boundary of the hull, UIPQ}, we get that, on the event $H_{n,\vd}^{j,l}$, except on a set  of probability tending to $0$ as $n\to\infty$, the \fpp-distance between any point of the boundary of $\Hull_{Q^\bullet_n}(\floor{\va_{j}n^{1/4}}-1)$ and the root vertex $\rho_n$ is close to $\mathbf{c_0} \va_{j}n^{1/4}$, up to an error bounded by $\ve n^{1/4}$. Note that \propref{Prop_P20 distances root - boundary of the hull, UIPQ} considers vertices of the boundary of the truncated hull, whereas here we
want to deal with the boundary of the standard hull of the same radius: This makes no difference since it is easily checked that 
any vertex of the boundary of the standard hull either belongs to the  boundary of the truncated hull or is adjacent to a vertex of the latter boundary. Moreover, when applying \eqref{Eq_density for hulls in finite and infinite case}, we should consider the ``intrinsic \fpp-distance'' on $\Hull_{Q^\bullet_n}(\floor{\va_{j}n^{1/4}}-1)$ in order to compare it with the similar intrinsic distance on the corresponding hull of $Q_\infty$. However, similarly as in the proof
of Proposition \ref{Prop_P19 thin annuli}, we can use the fact that the \fpp-distance is bounded above by the intrinsic \fpp-distance, and the minimal \fpp-distance from a point of the hull boundary is equal to the minimal intrinsic \fpp-distance. 

We also know that, still on the event $H_{n,\vd}^{j,l}$, the graph distance (in $Q^\bullet_n$) between $\partial_n$ and the boundary of the hull $\Hull_{Q^\bullet_n}(\floor{\va_{j}n^{1/4}})$ is bounded above by $2\vd^2 n^{1/4}$ (to see this, recall that labels $Z^n_a$ correspond to distances from $\partial_n$, up to a shift by $-Z^n_{\eta_n}+1$, and consider a
geodesic from $\rho_n$ to $\partial_n$). 

Recalling that the \fpp-weights are bounded above by $\vk$, we then obtain, on the event $H_{n,\vd}^{j,l}$ except on a set  of probability tending to $0$ as $n\to\infty$, that the \fpp-distance (in $Q_n$) between $\partial_n$  and $\rho_n$ is close to $\mathbf{c_0} \ \dgr^{Q_n}(\rho_n,\partial_n)$, up to an error bounded by $(2\vd^2+\ve)\vk n^{1/4}$.

We can apply the previous property to each set $H_{n,\vd}^{j,l}$, and we obtain that on the event $H_{n,\vd}$, except on a set of probability tending to $0$ when $n\to\infty$, we have
\begin{equation*}
|\dfpp^{Q_n}(\rho_n,\partial_n)-\mathbf{c_0} \ \dgr^{Q_n}(\rho_n,\partial_n)| \leq (2\vd^2+\ve)\vk n^{1/4}.
\end{equation*}
This completes the proof of \propref{Prop_P21 distances root - uniform point, Qn}.

\begin{proof}[Proof of \lemref{Lemma_technical slicing}]

 We need to introduce some notation. We write $u^n_0,u^n_1,\ldots,u^n_{2n}$ for the contour sequence of $T_n$: $u^n_0$ is the root of $T_n$ and, for $1 \leq j \leq 2n$, $u^n_j$ is either the first child of $u^n_{j-1}$ that does not appear in $u^n_0, ... u^n_{j-1}$, or the parent of $u^n_{j-1}$ if there is no such child.  
 We write $(C^n_k)_{0\leq k\leq 2n}$ for the contour function (so that $C^n_k=|u^n_k|$). The discrete snake associated with $T_n$ is denoted by $(W^n_k)_{0\leq k\leq 2n}$: $W^n_k=(W^n_k(j))_{0\leq j\leq C^n_k}$, where $W^n_k(j)$ is the label of the ancestor of $u^n_k$ at generation $j$. For simplicity, we write $Y^n_k=W^n_k(C^n_k)=Z^n_{u^n_k}$. By the results of Janson and Marckert \cite{janson2005convergence}, we have
\begin{equation}
\label{Eq_convergence of the brownian snake}
\l(\frac{1}{\sqrt{2n}} \  C^n_{\floor{ 2nt}}, \l(\frac{9}{8}\r)^{1/4} n^{-1/4}W^n_{\floor{ 2nt}}\l(\floor{\sqrt{2n} \ \cdot} \r)\r)_{0\leq t\leq 1}
\ulimd{n}{\infty} (\be_t,W_t)_{0\leq t\leq 1},
\end{equation}
where $\be$ is a normalized Brownian excursion, and $W$ is the Brownian snake driven by $\be$. The convergence in \eqref{Eq_convergence of the brownian snake} holds in the topology of uniform convergence. 
By the Skorokhod representation theorem, we may and will assume that the latter convergence holds \as We write $(\TT_\be,d_\be)$ for the tree coded by $\be$ and we also set $Y_t=\wh W_t$. The process $Y$ (or $W$) can be viewed as indexed by $\TT_\be$. Fix $\va\in(0,1/100)$. We observe that
\begin{equation*}
\sup_{a,b\in \TT_\be, a\not =b}  \ \frac{|Y_a-Y_b|}{d_\be(a,b)^{\frac{1}{2}-\va}} =: C_\omega <\infty,\qquad \as
\end{equation*}
Since conditionally given $\be$, $Y$ can be interpreted as Brownian motion indexed by $\TT_\be$, this follows from standard chaining arguments (using metric entropy bounds) and we omit the details.

It follows that
\begin{equation*}
\sup_{0\leq s\leq 1}\sup_{0\leq r\leq \vd^5\wedge\be_s} |\wh W_s - W_s(\be_s-r)| \leq C_\omega (\vd^5)^{\frac{1}{2}-\va}
\end{equation*}
and the right-hand side is smaller than $\vd^2/10$ except on a set of small probability when $\vd$ is small. On the other hand, it follows from the (\as) convergence \eqref{Eq_convergence of the brownian snake} that
\begin{equation*}
\sup_{0\leq k\leq 2n} \sup_{0\leq j\leq 2\vd^5\sqrt{n}\wedge C^n_k} \l(\frac{9}{8}\r)^{1/4} n^{-1/4} |W^n_k(C^n_k)-W^n_k(C^n_k-j)|
\ulimas n \infty \sup_{0\leq s\leq 1}\sup_{0\leq r\leq \sqrt{2}\vd^5\wedge\be_s} |\wh W_s - W_s(\be_s-r)|.
\end{equation*}
Hence, by the preceding observations, we can find $\vd_0>0$ such that, for every $\vd\in(0,\vd_0]$, for all $n$ large enough, 
\begin{equation}
\label{Eq_bound on sup of brownian snake}
\sup_{0\leq k\leq 2n} \sup_{0\leq j\leq 2\vd^5\sqrt{n}\wedge C^n_k} \l(\frac{9}{8}\r)^{1/4} n^{-1/4} |W^n_k(C^n_k)-W^n_k(C^n_k-j)|
<\frac{\vd^2}{10}
\end{equation}
except possibly on a set of probability bounded above by $\ve/4$.

Write $k_n$ for the first index such that $u^n_{k_n}=\eta_n$ (so with our notation $Y^n_{k_n}=Z^n_{\eta_n}$). We note that, by \eqref{Eq_convergence of the brownian snake}, we have $k_n/\sqrt{2n}\la t_*$ as $n\to\infty$, \as, where $t_*$ is the unique value such that $Y_{t_*}=\min\{Y_t:0\leq t\leq 1\}$.

Outside a set of small probability when $\vd$ is small (uniformly in $n$) we can find $j\in \{\floor{\vd^{-1}},\ldots,\floor{\vd^{-3}}\}$ and $l \in \{\floor{\vd^{-4}},\ldots,\floor{\vd^{-6}}\}$ such that
\begin{equation*}
-\va''_j n^{1/4} \leq Z^n_{\eta_n} < -\va'_j n^{1/4} \ ,\qquad \vb'_l\sqrt{n}\leq C^n_{k_n}=\dgr^{T_n}(\rho_n,\eta_n)< \vb''_l \sqrt{n}.
\end{equation*}
We must now justify the fact that we have also $Z^n_{[\eta_n]_{\vb_l\sqrt{n}}}<-\va_j \ n^{1/4}$ and $\vT(T_n,\eta_n,\floor{ \vb_l\sqrt{n}})>\vd^{11}n$. The first property follows from \eqref{Eq_bound on sup of brownian snake} since 
\begin{equation*}
|Z^n_{\eta_n}-Z^n_{[\eta_n]_{\vb_l\sqrt{n}}}|=|W^n_{k_n}(C^n_{k_n})- W^n_{k_n}(\floor{\vb_l\sqrt{n}})|,
\end{equation*}
 and $C^n_{k_n}-\floor{\vb_l\sqrt{n}} \leq 2\vd^5\sqrt{n}$. For the second property, we note that $\vT(T_n,\eta_n,\floor{ \vb_l\sqrt{n} })\geq \frac{1}{2}(k'_n-k_n)$ where $k'_n \eqdef \min\{j\geq k_n: C^n_j\leq \vb_l\sqrt{n}\}$, and using \eqref{Eq_convergence of the brownian snake}, we have
 \begin{equation*}
 \liminf_{n\to\infty} (2n)^{-1} ( k'_n-k_n) \geq \inf\{s>t_*:\be_s=(\be_{t_*} - \vd^5/\sqrt{2})^+\} - t_* .
 \end{equation*}
 For $\vd>0$ small enough, for $n$ large, the H\"older continuity properties of the Brownian excursion show that the right-hand side of the last display is greater than $\vd^{11}$ except on a set of probability bounded above by $\ve/4$. This completes the proof.
 \end{proof}

\subsection{Distances between two arbitrary points in a finite quadrangulation}
\label{Sec_Distances between any pair of points of finite quadrangulations (Theorem 1)}

The next statement gives the part of Theorem \ref{Th1} concerning quadrangulations.

\begin{theorem}
\label{Th_T1 controle distances fpp et gr dans quad finies}
For every $\ve>0$, we have
\begin{equation*}
\P\l( \sup_{x,y \in V(Q_n)} \l| \dfpp^{Q_n}(x,y) - \bc_p \dgr^{Q_n}(x,y) \r| > \ve n^{1/4} \r) \ulim{n}{\infty} 0\,.
\end{equation*}
\end{theorem}

\begin{proof}
The proof follows the same pattern as that of \cite[Theorem 1]{fpp}, and we refer to \cite{fpp} for more details. 
We first claim that \propref{Prop_P21 distances root - uniform point, Qn} remains valid if  $\rho_n$ is replaced by a uniformly distributed vertex of $Q_n$. 
In other words, if $\partial'_n$ is uniformly distributed on $V(Q_n)$ conditionally on $Q^\bullet_n$, 
$$\P\l( | \dfpp^{Q_n}(\partial_n, \partial'_n) - \bc_p \dgr^{Q_n}(\partial_n, \partial'_n) | > \ve n^{1/4} \r) \ulim n \infty 0.$$

Let us explain this. The law of $Q_n$ is invariant under uniform re-rooting, so that the statement of \propref{Prop_P21 distances root - uniform point, Qn} still holds if we replace $\rho_n$ by the tail of an oriented edge chosen uniformly on $Q_n$. 
Let $\overrightarrow{E}(Q_n)$ be the set of all oriented edges of $Q_n$ (with cardinality $4n$), and, for every $e\in \overrightarrow{E}(Q_n)$, write $t(e)$ for the tail of $e$. Then 
\propref{Prop_P21 distances root - uniform point, Qn} and the invariance under uniform re-rooting give
\begin{equation*}
\frac{1}{4n} \E\l[ \sum_{e \in \overrightarrow{E}(Q_n)} \ind{| \dfpp^{Q_n}(\partial_n, t(e)) - \bc_p \dgr^{Q_n}(\partial_n, t(e)) | > \ve n^{1/4}} \r] \ulim n \infty 0 .
\end{equation*}
Every vertex of $Q_n$ appears at least once as the tail $t(e)$ of an oriented edge $e$, and thus it also follows that 
\begin{equation*}
 \E\l[ \frac{1}{n+2}\,\sum_{v \in V(Q_n)} \ind{| \dfpp^{Q_n}(\partial_n, v) - \bc_p \dgr^{Q_n}(\partial_n, v) | > \ve n^{1/4}} \r] \ulim n \infty 0 .
\end{equation*}
This proves our claim.

Let $\alpha>0$. We know that $V(Q_n)$ equipped with the graph distance rescaled by $\pfrac{9}{8n}^{1/4}$ and with the uniform probability measure converges in distribution in the Gromov-Hausdorff-Prokhorov topology towards the Brownian map equipped with its volume measure (see \cite[Theorem 7]{gall2017browniandiskssnake}). Since the Brownian map is a compact metric space, it follows that for every $\ve>0$, we can fix $N$ large enough (not depending on $n$) such that, if $(\partial_n^i)_{1 \leq i \leq N}$ are \iid uniformly distributed random vertices of $Q_n$, then the $\dgr^{Q_n}$-balls of radius $\ve n^{1/4}$ centered at the vertices $\partial_n^i$, $1\leq i\leq N$, cover $Q_n$ with probability at least $1-\alpha$ (see the end of Appendix A1 in \cite{gall2017browniandiskssnake} for a detailed
justification).

The preceding assertion ensures that, on an event of probability at least $1-\alpha$,  the $\dgr^{Q_n}$-distance (respectively the $\dfpp^{Q_n}$-distance) between any pair of points is well approximated by the distance between a certain pair of vertices in $(\partial^i_n)_{1\leq i \leq N}$, up to a difference bounded by $2\ve n^{1/4}$ (\resp by $2\kappa\ve n^{1/4}$).
On the other hand, the first part of the proof shows that, for $n$ large enough, we have $| \dfpp^{Q_n}(\partial^i_n, \partial^j_n) - \bc_p \dgr^{Q_n}(\partial^i_n, \partial^j_n) | \leq \ve n^{1/4}$ for all $1 \leq i,j \leq N$ with probability at least $1-\alpha$. We conclude that, except on a set of probability at most $2\alpha$, $\bc_p \dgr^{Q_n}$ and $\dfpp^{Q_n}$ differ by at most $(1+4\vk) \ve n^{1/4}$. This completes the proof.
\end{proof}

The next result is very similar to \cite[Theorem 2]{fpp}. Stating the result for hulls instead of balls is a minor improvement that could  also be achieved in the framework of \cite{fpp}. 
Balls and hulls with respect to the first-passage percolation distance are defined in the same way as for the graph distance: For every $r\in (0,\infty)$, the \fpp-ball $B_{Q_\infty}^\fpp(r)$ is the union of all faces of $Q_\infty$ that are incident to a vertex at \fpp-distance strictly less than $r$ from the root vertex of $Q_\infty$, and the \fpp-hull $B^{\bullet,\fpp}_{Q_\infty}(r)$ is the union of $B_{Q_\infty}^\fpp(r)$ and of the finite connected components of its complement.

\begin{theorem}
\label{Th_T2 hulls in the UIPQ are close for dfpp and dgr}
Let $\ve \in (0,1)$. We have
\begin{equation}
\label{Eq_distances are close in hulls of Qinfty}
\lim_{r \to \infty} \P\l( \sup_{x,y \in V(\Hull_{Q_\infty}(r))} \l| \dfpp^{Q_\infty}(x,y) - \bc_p \dgr^{Q_\infty}(x,y) \r| > \ve r \r) = 0 .
\end{equation}
Consequently,
\begin{align*}
\P\l( B_{Q_\infty}((1-\ve)r / \bc_p) \subset B^{\mathrm{fpp}}_{Q_\infty}(r) \subset B_{Q_\infty}((1+\ve)r / \bc_p) \r) \ulim r \infty 1 , \\
\P\l( \Hull_{Q_\infty}((1-\ve)r / \bc_p) \subset B^{\bullet,\mathrm{fpp}}_{Q_\infty}(r) \subset \Hull_{Q_\infty}((1+\ve)r / \bc_p) \r) \ulim r \infty 1 .
\end{align*}
\end{theorem}

We will need the following lemma, where we use the same notation $Q^\bullet_n$ as in the preceding sections.
We make the convention that, if $r\geq \dgr^{Q_n^\bullet}(\rho_n,\partial_n)-1$, 
then $\Hull_{Q^\bullet_n}(r) = Q^\bullet_n$.

\begin{lemma}
\label{Lemma_hulls are contained in larger balls whp}
For every $\ve>0$, we can choose $K'>1$ s.t. for every $\vb>0$, for all $n$ large enough, we have with probability greater than $1-\ve$,
\begin{equation*}
\Hull_{Q^\bullet_n}(\vb n^{1/4}) \subset B_{Q_n}(K' \vb n^{1/4}),
\end{equation*}
and for all $r$ large enough, with probability greater than $1-\ve$,
\begin{equation*}
\Hull_{Q_\infty}(r) \subset B_{Q_\infty}(K' r) .
\end{equation*}
\end{lemma}

\begin{proof}[Proof of \lemref{Lemma_hulls are contained in larger balls whp}]

Let us begin with an observation that will be useful later in the proof.  Fix $\ve>0$. 
Let $(\PP, D_\infty)$ stand for the Brownian plane of \cite{curienLeGall2014brownian}. Recall that
the Brownian plane comes with a distinguished point, which we denote by $x_0$. 
We write $\mathcal{B}_{\PP}(\vb)$ for the closed ball of radius $\beta$ centered at $x_0$ in $\PP$. For every $\vb>0$, the hull of radius $\beta$ in $(\PP, D_\infty)$, denoted by
$\mathcal{B}^\bullet_\PP(\beta)$, is the complement of the unbounded connected
component of the complement of $\mathcal{B}_{\PP}(\vb)$. Then,
\begin{equation*}
\sup \{ D_\infty(x_0, x) : x \in \mathcal{B}^\bullet_\PP(1) \}<\infty\hbox{ a.s.}
\end{equation*}
and thus we can find $K>1$ such that the latter supremum is smaller than $K$ with probability at least $1-\ve/4$.
By the scaling invariance of the Brownian plane, it follows that for every $\vb>0$, 
\begin{equation}
\label{Hull/Ball}
\P(\mathcal{B}^\bullet_\PP(\vb) \subset \mathcal{B}_\PP(K\vb))\geq1-\ve/4.
\end{equation}

Consider then the Brownian map $(\bmap, D^*)$, which also comes with a distinguished point $x_*$
(in the construction of the Brownian motion from the CRT indexed by Brownian labels, $x_*$ is the point with
minimal label). We write $\mathcal{B}_\bmap(\vb)$ for the closed ball of radius $\beta$ centered at $x_*$. We let $\partial$ be another distinguished point uniformly distributed 
over $\bmap$, and, if $0<\beta<D^*(x_*,\partial)$, we define $\mathcal{B}^\bullet_{\bmap}(\vb)$ as the complement of the connected
component of the complement of the ball $\mathcal{B}_\bmap(\vb)$ that contains $\partial$. If $\beta\geq D^*(x_*,\partial)$, we take $\mathcal{B}^\bullet_{\bmap}(\vb)=\bmap$.
Using the coupling between the Brownian map $(\bmap, D^*)$ and the Brownian plane found in \cite[Theorem 1]{curienLeGall2014brownian}, one gets
from \eqref{Hull/Ball} that
there exists $\vd>0$ such that 
\begin{equation}
\label{Eq_hull dans boules version bmap}
\P\l(\mathcal{B}^\bullet_{\bmap}(\vb) \subset \mathcal{B}_\bmap(K\vb)\r) \geq 1-\ve/2,
\end{equation}
for every $0<\vb<\vd$. Let us briefly justify this. We note that  \cite[Theorem 1]{curienLeGall2014brownian} allows us to couple $\PP$ and $\bmap$ so that there exists
$\alpha_0>0$ such that, with high probability, we have $\mathcal{B}_\bmap(\alpha)=\mathcal{B}_\PP(\alpha)$ for every $\alpha\in(0,\alpha_0]$. Then, if $K\beta<\alpha\leq \alpha_0$,
the property  $\mathcal{B}^\bullet_\PP(\vb) \subset \mathcal{B}_\PP(K\vb)=\mathcal{B}_\bmap(K\vb)$ also implies that 
$\mathcal{B}^\bullet_{\bmap}(\vb)=\mathcal{B}^\bullet_\PP(\vb) $, provided
that $\partial$ does not belong to $\mathcal{B}_\bmap(\alpha)$, which holds with high probability if $\alpha$ has been taken small enough.

In fact, taking the constant $K$ larger if necessary, we may even assume that the bound in \eqref{Eq_hull dans boules version bmap} holds
for {\em every} $\vb>0$. Indeed, we just have to take $K$ so large that $\P(\mathcal{B}_\bmap(K\vd)=\bmap)\geq 1-\ve/2$. 

In order to deduce the first assertion of the lemma from the preceding considerations, 
we use the convergence  of $(V(Q_n^\bullet),(8/9)^{1/4}\dgr^{Q_n^\bullet})$ towards the Brownian map
in the bipointed Gromov-Hausdorff topology (see \cite[Theorem 7]{gall2017browniandiskssnake}). Note $V(Q_n^\bullet)$ is viewed as a bipointed space with distinguished points $\rho_n$ and $\partial_n$
(in this order) and similarly $\bmap$ is a bipointed space with distinguished points $x_*$ and $\partial$ --- at this point we note that \cite[Theorem 7]{gall2017browniandiskssnake}
considers a seemingly different choice of distinguished points in the Brownian map, but the re-rooting invariance properties of \cite[Section 8]{LeGall2010Geodesics}
show that this makes no difference. It follows from the preceding convergence of bipointed spaces that, for any choice of $0<\beta<\beta'<\gamma'<\gamma$, 
\begin{equation}
\label{GH/Hull}
\liminf_{n\to\infty} \P\Big( \Hull_{Q^\bullet_n}((8/9)^{1/4} \vb n^{1/4})\subset B_{Q_n}((8/9)^{1/4}\gamma n^{1/4})\Big)
\geq \P\l(\mathcal{B}^\bullet_{\bmap}(\vb')\subset \mathcal{B}_\bmap(\gamma')\r).
\end{equation}
The derivation of \eqref{GH/Hull} is a simple exercise on the Gromov-Hausdorff convergence and we omit the details.

The first assertion of the lemma now follows from \eqref{Eq_hull dans boules version bmap} and \eqref{GH/Hull}: just take $K'>K$ to
obtain the desired statement for $n$ large enough. The second assertion of the lemma can be derived by
similar arguments using now the fact that the Brownian plane is the scaling limit of the \UIPQ in the
local Gromov-Hausdorff sense \cite[Theorem 2]{curienLeGall2014brownian}.
\end{proof}

\begin{proof}[Proof of \thref{Th_T2 hulls in the UIPQ are close for dfpp and dgr}]

Let us focus on the first statement (the second one follows easily). Let $\vd>0$. 

By \lemref{Lemma_hulls are contained in larger balls whp} applied to $Q_\infty$, we can find $K'>1$ such that, for $r$ large enough, $\Hull_{Q_\infty}(r) \subset B_{Q_\infty}({K' r})$ with probability at least $1-\vd/4$. An elementary argument allows one to find a large enough constant $C>1$ such that, for every $r\geq 1$, the $\dgr^{Q_\infty}$ and $\dfpp^{Q_\infty}$-distances between vertices of $B_{Q_\infty}({K' r})$ are determined by $B_{Q_\infty}({C K'r})$ and the weights on the edges of $B_{Q_\infty}({C K' r})$. In particular, outside of an event of probability smaller than $\vd/4$, the event whose probability is considered in \eqref{Eq_distances are close in hulls of Qinfty} can be expressed in terms of the ball $B_{Q_\infty}({CK'r})$ (and weights in this ball). 
Similarly it follows from \lemref{Lemma_hulls are contained in larger balls whp} that for any $\vb>0$, for $n$ large enough, the $\dgr^{Q_n}$-distance and the $\dfpp^{Q_n}$-distance between two vertices of $\Hull_{Q^\bullet_n}(\vb n^{1/4})$ are determined by $B_{Q_n}( CK'\vb n^{1/4})$, except on a set of probability smaller than $\vd/4$.

On the other hand, by \propref{Prop_TBP 9}, we can find $\chi>0$ such that for all $n$ large, we can couple $Q^\bullet_n$ and $Q_\infty$ in such a way that $B_{Q_n}(\chi n^{1/4}) = B_{Q_\infty}(\chi n^{1/4})$ except on a set of probability at most $\vd/4$. 

Let $\ve>0$. For $n$ large we have
\begin{align*}
&\P\l( \sup_{x,y \in V( \Hull_{Q_\infty}(\frac{\chi n^{1/4}}{CK'} ) )}  | \dfpp^{Q_\infty}(x,y)-\bc_p \dgr^{Q_\infty}(x,y)| > \ve n^{1/4} \r) \\
&\leq \frac{3\vd}{4} + \P\l( \sup_{x,y \in V(Q^\bullet_n)}  | \dfpp^{Q^\bullet_n}(x,y)-\bc_p \dgr^{Q^\bullet_n}(x,y)| > \ve n^{1/4} \r) \\
&\leq \vd .
\end{align*}
The second inequality follows from \thref{Th_T1 controle distances fpp et gr dans quad finies}. For the first one, we observe that both the $\dfpp^{Q_\infty}$-distance and the $\dgr^{Q_\infty}$-distance on $V( \Hull_{Q_\infty}(\frac{\chi n^{1/4}}{CK'}) )$ only depend on the ball $B_{Q_\infty}(\chi n^{1/4})$ (except on a set of probability at most $\vd/4$ in each case) and we know that we can couple $Q_\infty$ and $Q_n$ so that the balls $B_{Q_\infty}(\chi n^{1/4})$ and $B_{Q_n}(\chi n^{1/4})$ are equal except on a set of probability less than $\vd/4$. This completes the proof.
\end{proof}

\section{Technical lemmas for distances in the general map}
\label{technical-general}

We now proceed to prove that \Tuttebij is asymptotically an isometry. In order to do so, we first prove a handful of lemmas that control the distance in the map obtained
from a quadrangulation $Q$ via \Tuttebij  in terms of the graph distance in $Q$. The key object is an analog of \lmgs, which we call downward paths, and which we define in \secref{Section_Downward_paths}.

Let us briefly recall the definition of \Tuttebij (see \figref{Fig_Tutte_definition}). Let $Q$ be a quadrangulation with $n$ faces, and color its vertices in black and white so that adjacent vertices have a different color and the root vertex is white. In every face of $Q$, draw an edge between the two white corners of this face. Then erase the edges of $Q$ and all black vertices. We denote by $\Ti Q$ the map with $n$ edges obtained in this way.

The preceding construction of a (general) planar map from a quadrangulation can also be applied to the \UIPQ, and yields an infinite
(random) planar map called the \UIPM for Uniform Infinite Planar Map. Indeed, it was observed in \cite{menard2014percolation} that the \UIPM is
the local limit of uniformly distributed (general) planar maps with $n$ edges.

We can extend Tutte's correspondence to truncated hulls of even radius: The white vertices are those whose graph distance from the root vertex is even, then we draw diagonals
between the two white corners of any quadrangular inner face, and we also keep the edges of the external boundary (indeed this external boundary was made of diagonals
in the construction of the truncated hull). By definition, the (new) root edge is the diagonal drawn in the face to the left of the original root edge,
and the root vertex remains the same.  See \figref{Fig_Tutte_definition_hulls} for an example. Similarly, we can extend Tutte's correspondence to quadrangulations of a cylinder of even height, in such a way that we keep edges 
of both the top and the bottom boundary. The root edge then remains the same.

Finally, Tutte's correspondence is also extended in an obvious manner to the \LHPQ, in such a way that all edges of the boundary of the \LHPQ remain present
in the resulting infinite map (the latter also contains all edges of the form $((i,j),(i+1,j))$ for even values of $j\leq 0$).

\begin{figure}
\begin{center}
\includegraphics[width=\textwidth]{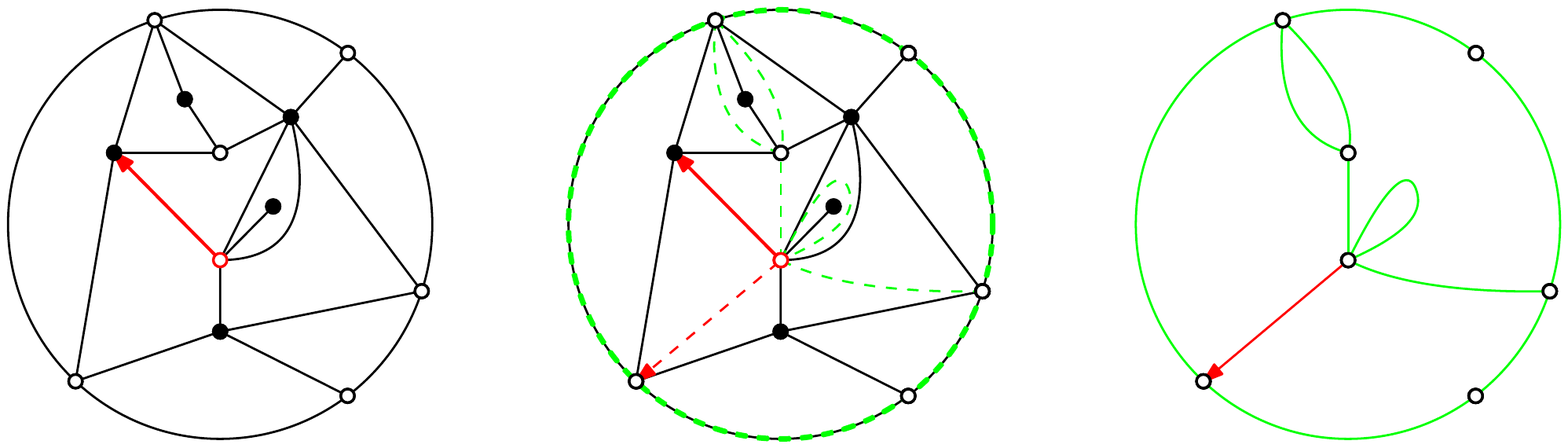}
\end{center}
\caption{\Tuttebij applied to a truncated hull of even radius, here of radius 2. In every face of $Q$ of degree 4 but the external one, draw a ``diagonal'' between the two white corners. 
Then erase the edges of the original map and all black vertices, keeping however the edges of the boundary. 
The map we obtain is rooted at the edge drawn in the face of $Q$ on the left of the root edge of $Q$, oriented in such a way that its tail vertex is the root vertex of $Q$.
}
\label{Fig_Tutte_definition_hulls}
\end{figure}

\subsection{Downward paths}
\label{Section_Downward_paths}

In this section, we define certain special paths called downward paths, in the image of a quadrangulation of the cylinder under
Tutte's correspondence. These special paths will be used to derive upper bounds for the distances in the \UIPM.

Let $R>0$ be an integer, let $Q$ be a quadrangulation of the cylinder of height $2R$ with top boundary length $p$, and let $u_0$ be a vertex of its top boundary. 
 We write $(\ur_i)_{0 \leq i < p}$ for the vertices of the top boundary in clockwise order, and extend this numbering to $i \in \Z$ by periodicity (recall that the top face is drawn as the infinite face). 
 Recall that for every $i \in \Z$, the edge $\{ \ur_i, \ur_{i+1} \}$ is also an edge of $\Ti{Q}$.
 
 Recall the skeleton decomposition from \secref{Sec_Preliminaries}: $Q$ is encoded by a forest $(\tau_i)_{0\leq i <p}$, whose vertices are identified with the edges of $\partial_r q$ for $0 \leq r \leq 2R$, and a collection of truncated quadrangulations. We extend the numbering of the forest to $\Z$ by periodicity, and shift the indices so that for all $i\in\Z$, the left-most vertex of the root of $\tau_i$ is $\ur_i$.

We say that the vertex $u_i$ is good if the slot associated with the edge $(u_i,u_{i+1})$ is filled in by the truncated quadrangulation represented in the right side of \figref{Fig_Downward paths exponentielle} (in particular the edge $(u_i,u_{i+1})$ must have exactly one child in the skeleton), and bad otherwise. Assume that at least one of the $\ur_j$'s is labeled good, and let $\ur_i$, with $-p \leq i \leq -1$, be the first good vertex visited when exploring the top boundary in counterclockwise order starting from $\ur_{-1}$. We define the downward path $\DP(\ur_0, 2R-2)$ from $\ur_0$ to $\partial_{2R-2} Q$ as follows. We first move along $\partial_{2R} Q$
 in counterclockwise order from $\ur_0$ to $\ur_i$. Then, we follow the unique edge of $\Ti{Q}$ from $\ur_i$ to $\partial_{2R-2} q$ inside the slot associated with  $(u_i,u_{i+1})$. See \figref{Fig_Downward paths exponentielle} for an illustration.

\begin{figure}[h]
\includegraphics[width=\textwidth]{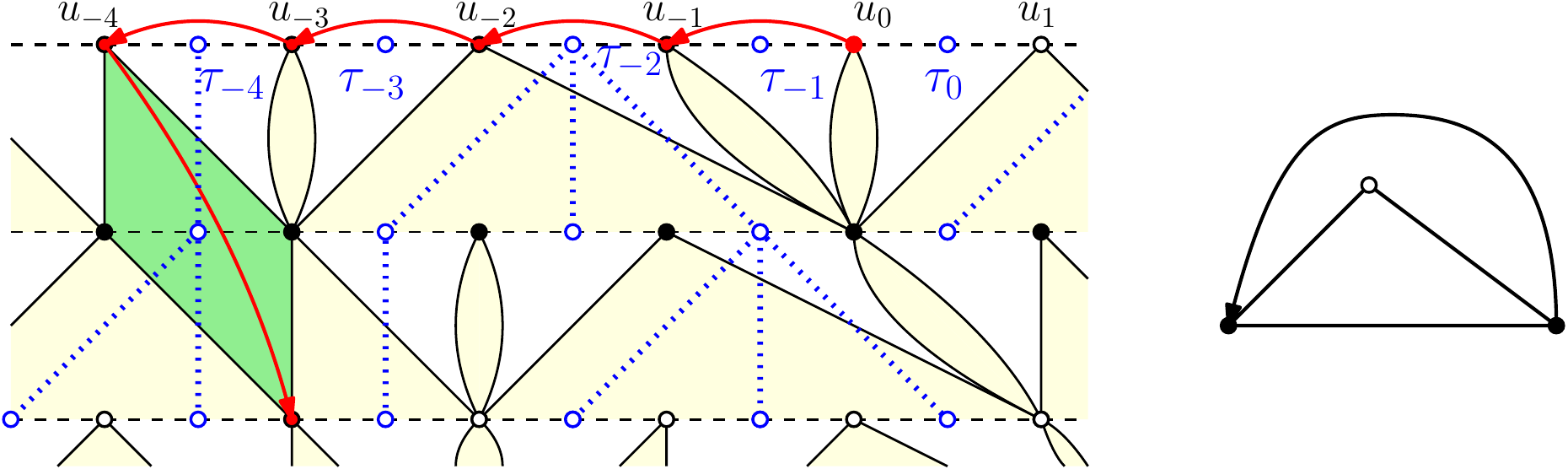}
\caption{Left, a part of the annulus $\CC(2R-2, 2R)$ with slots in pale yellow, the skeleton in dotted blue lines, and the special slot in green. We have not drawn the edges and vertices inside the other slots. The downward path (in red) visits $u_0, u_{-1},u_{-2},\ldots$ until it meets a ``good'' vertex (here $u_{-4})$, whose corresponding slot is filled by the truncated quadrangulation on the right side.}
\label{Fig_Downward paths exponentielle}
\end{figure}

This path can be extended by induction to a path in $\Ti{Q}$ from $\ur_0$ to $\partial_{2s} Q$ for every $0\leq s < 2R$, provided we can find good vertices at height $2k$ for all $s<k\leq R$. If not, the downward path is not defined, but we still define its length to be $+\infty$.

We can extend this definition to downward paths in $\Ti{\LL}$ where $\LL$ is the \LHPQ with truncated boundary. There will \as be good vertices at all heights, thus downward paths 
are always well defined.

The following lemma provides an upper bound on the length of downward paths in annuli of the \UIPQ
(recall that these annuli are quadrangulations of the cylinder). Roughly speaking, this upper bound shows that the graph distance 
(in $\Ti{Q_\infty}$) between a vertex of $\partial_{2s} Q_\infty$ and the cycle $\partial_{2r} Q_\infty$ is not larger than a constant times $s-r$, with high 
probability uniformly in $r<s<R$. 

\begin{lemma}
\label{Lemma_control_length_DP}
We can find $\va>0$ such that the following holds. 
For every $\ve, \vd>0$ and every integer $R>0$, let $A_R(\vd, \ve)$ be the event where for all $\vd \frac{R}{\ln R} \leq r < s \leq R$ such that $s-r \geq \ve r$, for every $v \in \partial_{2s} Q_\infty$, the downward path from $v$ to $\partial_{2r} Q_\infty$ is well defined and
 has length smaller than $\va(s-r)$.  Then
\begin{equation*}
\P(A_R(\vd, \ve)) \ulim R \infty 1 .
\end{equation*}
\end{lemma}

\begin{proof}
We fix an integer $R>0$. 
Let us first consider a forest $(\wt \tau_1,\wt \tau_2, ... , \wt \tau_p)$ of $p$ independent \BGW trees with offspring distribution $\vt$, where $R<p\leq R^3$. 
We truncate this forest at generation $2R$ (we remove all vertices whose height is greater than $2R$).
We can view this forest as the skeleton of a quadrangulation $\wt Q$ of the cylinder whose height 
is the maximal generation in the truncated forest (in the skeleton decomposition of $\wt Q$, slots are filled in independently according to the same distribution as in the \UIPQ).  We again say that a vertex $u$ of the skeleton is good if it has a unique child and the slot corresponding to $u$ is filled as shown in the right side of \figref{Fig_Downward paths exponentielle}. 
We let $a>0$ be the product of $\vt(1)$ with the probability that a slot with lower boundary of size $1$ is filled in as just explained. Informally, $a$ represents the probability that a vertex is good.

Let $T$ stand for the smallest $i \geq 0$ such that generation $2i$ has no good vertex. For every $i\geq 0$, let $Y_i$ be the number of vertices at generation $2i$, and 
\begin{equation*}
\zeta = \min \{i : Y_i \leq R \text{ or } Y_i > R^3 \} .
\end{equation*}
Note that, on the event $\{\zeta>R\}$, the height of the truncated forest is $2R$ and $\wt Q$ is a quadrangulation of the cylinder of height $2R$.
 For every $i \geq 0$, let $\FF_i$ be the $\vs$-field generated by the trees truncated at generation $2i$ and the labels good or bad up to generation $2i-2$ (the labels at generation $2i$ are not $\FF_i$-measurable).

Fix a vertex $u$ at generation $0$ in the forest. For every $1\leq i\leq T$, we can construct the downward path from $u$, or rather from the vertex $v$ of $\wt Q$ which is the left end of the edge associated with $u$, to generation $2i$ of the forest (more precisely to the cycle whose edges form 
generation $2i$ of the forest), and we define $L_u(i)$ as the length of this downward path. By convention, $L_u(0)=0$.

Let us observe the following key fact: we can construct a sequence $G_0, G_1, ...$ of \iid geometric random variables with parameter $a$ such that, for every $0\leq i\leq R-1$,
\begin{equation}
\label{Eq_geom_bound_step_DP}
\ind{T \geq i} \P\l( L_u(i+1)-L_u(i) \neq G_i+2 \ | \ \FF_i \r) \leq \ind{T \geq i} a^{Y_i} 
\end{equation}
(note that $\{ T\geq i \} \in \FF_i$). 
This bound holds because $\FF_i$ gives no information on whether vertices at generation $2i$ are good or not: if these vertices are enumerated in clockwise order starting from a random vertex measurable \wrt $\FF_i$ the index of the first good one will follow a geometric distribution ``truncated'' at $Y_i$.

Let us now bound, for $1 \leq k\leq R$,
\begin{align*}
&\P\l( T \geq k , \zeta \geq k \ , \ L_u(k) \neq \sum_{i=0}^{k-1} (G_i+2) \r) \\
\leq & \ \P\l( T \geq k , \zeta \geq k \ , \ L_u(k-1) \neq \sum_{i=0}^{k-2} (G_i+2) \r) + \P\l( T \geq k , \zeta \geq k \ , \ L_u(k)-L_u(k-1) \neq G_{k-1}+2 \r) \\
\leq & \ \P\l( T \geq k-1 , \zeta \geq k-1 \ , \ L_u(k-1) \neq \sum_{i=0}^{k-2} (G_i+2) \r) + \E\l[ \ind{ \zeta \geq k } a^{Y_{k-1}} \r] ,
\end{align*}
using \eqref{Eq_geom_bound_step_DP} and the property that $\{ \zeta \geq k \}$ is $\FF_{k-1}$-measurable. By induction, we get
\begin{equation}
\label{Eq_borne DP 1}
\P\l( T \geq k , \zeta \geq k \ , \ L_u(k) \neq \sum_{i=0}^{k-1} (G_i+2) \r) \leq \sum_{i=0}^{k-1} \E\l[ \ind{ \zeta > i } a^{Y_{i}} \r] \leq ka^R .
\end{equation}
We can similarly bound
\begin{align*}
\P\l( \zeta \geq k, T<k \r) &= \sum_{i=0}^{k-1} \P(T=i, \zeta\geq k) \\
 &\leq \sum_{i=0}^{k-1} \P(T=i, \zeta> i)\\
 &= \sum_{i=0}^{k-1} \E\l[ \ind{ T>i-1 } \ind{ \zeta > i }a^{Y_i} \r]
\end{align*}
where the last equality is obtained by conditioning with respect to $\FF_i$ (on the event $\{T>i-1\}$, we have $\P(T=i \ | \ \FF_i) = a^{Y_i}$). It follows that
\begin{equation}
\label{Eq_borne DP 2}
\P\l( \zeta \geq k, T<k \r) \leq ka^R .
\end{equation}
On the other hand, by elementary large deviations estimates, there exist $\va, A>0$ such that for every $k\geq 0$,
\begin{equation}
\label{Eq_borne DP 3}
\P \l( \sum_{i=0}^{k-1} (G_i+2) > \va k\r) \leq e^{-Ak} .
\end{equation}
By combining \eqref{Eq_borne DP 1} and \eqref{Eq_borne DP 3}, we arrive at
$$\P\l( T \geq k , \zeta \geq k \ ,\ L_u(k)>\alpha k\r) \leq ka^R + e^{-Ak}.$$
We apply this to $k=R-r$ for all values of $r$ such that $r\geq \vd \frac{R}{\ln R}$
and $R-r>\ve r$, to get
\begin{align*}
&\P\l( T \geq R -\ceil{\vd\frac{R}{\ln R}}, \zeta \geq R-\ceil{\vd\frac{R}{\ln R}}\ ,\ L_u(R-r)>\alpha (R-r)\hbox{ for some }
r\hbox{ s.t. } \vd \frac{R}{\ln R}\leq r<\frac{R}{1+\ve}\r)\\
&\quad \leq R^2a^R + R\,e^{-A\vd \frac{R}{\ln R}}.
\end{align*}
We can then consider the union over all vertices $u$ at generation $0$ of the events appearing in the last display.
Since generation $0$ has at most $R^3$ vertices, the probability of this union is trivially bounded by
$R^3(R^2a^R + R\,e^{-A\vd R/\ln R})$.

Fix an integer $s$ such that $ \frac{R}{\ln R} \leq r < s \leq R$.
For every vertex $u$ at generation $2(R-s)$ in the forest and every $R-s\leq k\leq R$, we use the same notation $L_u(k)$ for the length of the downward path in $\wt Q$ from $u$
(or rather from the vertex $v$ of $\wt Q$ which is the left end of the edge associated with $u$) to $\partial_{2R-2k} \wt Q$ (corresponding to generation $2k$ in the forest coding $\wt Q$), provided this downward path exists. The same argument we used in the case $s=R$ shows that the probability of the event where $T\wedge \zeta \geq R -\ceil{\vd\frac{R}{\ln R}}$
 and there exists a vertex $u$ at generation $2(R-s)$ such that $L_u(s-r)>\alpha (s-r)$
for some $r$ such that $\vd \frac{R}{\ln R} \leq r < s $ and $s-r \geq \ve r$ is bounded above by 
$$R^3(R^2a^R + R\,e^{-A\vd\ve R/\ln R}).$$

We then sum this estimate over possible values of $s$. To simplify notation, set 
$$\RR(R,\vd,\ve) = \l\{ (r,s)\in\N\times\N \ : \ \vd \frac{R}{\ln R} \leq r < s \leq R , \  s-r \geq \ve r \r\},$$
and also write $\mathcal{D}$ for the event where $ L_u(R-r) < \va(s-r)$ for every $(r,s) \in \RR(R,\vd,\ve)$ and every vertex $u$ at generation 
$2(R-s)$ in the forest. Then using also \eqref{Eq_borne DP 2} we get
\begin{equation}
\label{Eq_proba_de_BB_cas_iid}
\P\l( \l\{ \zeta \geq R-\vd \frac{R}{\ln R} \r\}  \cap \mathcal{D}^c \r) \leq Ra^R + R^6 a^R +  R^5 e^{-A \ve \vd R/\ln R}.
\end{equation}

Consider now the \UIPQ $Q_\infty$ and fix $\eta>0$. Using the tail estimates in \propref{Prop_law_perimeter_UIPQ}, we can easily
verify  that for $R$ large enough, the event
\begin{equation*}
\EE \eqdef \bigcap_{\ceil{\vd \frac{R}{\ln R}} \leq r \leq R}  \l\{ R < H_{2r} \leq R^3 \r\} .
\end{equation*}
has probability at least $1-\eta/2$. Let $A_R(\vd,\ve)$ be the event considered in the statement of the lemma: $A_R(\vd,\ve)$ is
the event where, for every $(r,s) \in \RR(R,\vd,\ve)$, for every vertex $v$ of $\partial_{2s} Q_\infty$ the downward path
from $v$ to $\partial_{2r}Q_\infty$ (exists and) has length smaller than $\va(s-r)$. 
We observe that the event
\begin{equation*}
\EE \cap A_R(\vd,\ve)^c
\end{equation*}
is a function of the skeleton of the annulus $\CC(2\ceil{\vd \frac{R}{\ln R}},2R)$ (and the quadrangulations filling in the slots). The point is now that 
the event in the left-hand side of \eqref{Eq_proba_de_BB_cas_iid} is the same function of the forest $(\wt \tau_1,\wt \tau_2, ... , \wt \tau_p)$ of independent trees
truncated at height $2R-2\ceil{\vd\frac{R}{\ln R}}$ (and 
of the quadrangulations used to construct $\wt Q$ from its skeleton). This means that we can use the relations between the law of the skeleton of the annulus and
that of a forest of independent trees to compare $\P(\EE \cap A_R(\vd,\ve)^c)$ with the probability in \eqref{Eq_proba_de_BB_cas_iid}. More precisely,
\propref{Prop_law_annulus_UIPQ_cond_top} gives for every $R < p' \leq R^3$,
\begin{align}
\label{Eq_Proba_DP_bien_defs_et_controles_sans_perim}
&\P\Big( \EE \cap A_R(\vd,\ve)^c \cap \l\{H_{2\ceil{\vd \frac{R}{\ln R}}} = p' \r\} \Big| \ H_{2R} = p \Big)\nonumber \\
&\quad= \frac{\vp_{2\ceil{\vd\frac{R}{\ln R}}}(p')}{\vp_{2R}(p)} \P\l( \l\{ \zeta > R-\vd \frac{R}{\ln R} \r\} \cap \mathcal{D}^c \cap \l\{Y_{R-\ceil{\vd \frac{R}{\ln R}}} = p' \r\} \r) ,
\end{align}
where we recall that $Y_i$ is the number of vertices at generation $2i$ in the forest $(\wt \tau_1,\wt \tau_2, ... , \wt \tau_p)$.
Using the explicit formula \eqref{Eq_phirp}, we find a constant $C>0$ such that for every sufficiently large $R$ and $p'\leq R^3$,
\begin{align*}
\vp_{2\ceil{\vd\frac{R}{\ln R}}}(p') &\leq C R^3 \l( \vd \frac{R}{\ln R} \r)^{-3} \leq C\vd^{-3} (\ln R)^3 .
\end{align*}
On the other hand, \eqref{Eq_phirp} and \propref{Prop_law_perimeter_UIPQ} give for $p' > 0$
$$
\frac{\P(H_{2R} = p)}{\vp_{2R}(p)} = \frac{32}{3\vk_1} \frac{3+2R}{((3+2R)^2-1)^2} \vk_{p} (2\pi_{2R})^{p} \l( \frac{64}{3}{p} \frac{3+2R}{((3+2R)^2-1)^2} \pi_{2R}^{{p}-1} \r)^{-1}
= \frac{2^{p}\vk_{p}}{{p}(2\vk_1)} \pi_{2R} 
\leq C'
$$
for a suitable constant $C'$ independent of ${p}>0$. Multiplying  \eqref{Eq_Proba_DP_bien_defs_et_controles_sans_perim} by $\P(H_{2R} = {p})$ and summing over all 
choices of $R<p,{p'} \leq R^3$ (using \eqref{Eq_proba_de_BB_cas_iid}), we get
\begin{equation*}
\P\Big( \EE \cap A_R(\vd,\ve)^c \Big) \leq \l(C\vd^{-3} (\ln R)^3\r) \l(C' R^3\r) \l( Ra^R + R^6 a^R +  R^5 e^{-A \ve \vd R/\ln R} \r) ,
\end{equation*}
which is smaller than $\eta/2$ for $R$ large. We already saw that the probability of $\EE$ is larger than $1-\eta/2$ for $R$ large enough, so we get that for all $\eta>0$, for $R$ large enough $\P( A_R(\vd,\ve)) \geq 1-\eta$. This completes the proof.
\end{proof}

\subsection{Coalescence in the UIPM}

Downward paths do not coalesce as nicely as \lmgs, but we can still get an ersatz of \propref{Prop_P17 coalescence primal}. 
Let $R>0$ an integer. Choose a vertex $u_0^{(R)}$ uniformly distributed over $\partial_{2R} Q_\infty$, write $(u_j^{(R)})_{0\leq j\leq H_{2R}-1}$ for the vertices of $\partial_{2R} Q_\infty$ 
enumerated in clockwise order, and extend the definition of $u_j^{(R)}$ to all $j\in\Z$ by periodicity.

\begin{corollary}[Coalescence of downward paths]
\label{Cor_P17 coalescence Tutte}
For every $\ve>0$ and $\eta\in(0,1)$, we can find a constant $A>1$ such that, for every large enough integer $R$, the following holds with probability at least $1-\ve$: any vertex of $\partial_{2R} Q_\infty$ is connected to one of the vertices $u_{\floor{k R^2 / A}}^{(R)}$, $0 \leq k \leq \floor{AH_{2R}/R^2}$, by a path in $\Ti{\CC(2R-2 \ceil{\eta R}, 2R)}$ of length at most $\eta R$.
\end{corollary}

\begin{proof} Consider two integers $0<s<R$. In the same way as we defined the downward path from a vertex $v$ of $\partial_{2R} Q_\infty$ to $\partial_{2s} Q_\infty$,
we can define the dual notion of the right downward path from $v$ to $\partial_{2s} Q_\infty$: If $v=u^{(R)}_j$, the first step goes from $v$ to $u^{(R)}_{j+1}$, then we
move clockwise along $\partial_{2R} Q_\infty$ until we meet a good vertex which allows us to go in one step from $\partial_{2R} Q_\infty$ to $\partial_{2(R-2)} Q_\infty$,
and we continue by induction. As  in the case of downward paths, the existence of the right downward path requires the existence of at least one good vertex
on every cycle $\partial_{2j} Q_\infty$, $s<j\leq R$.

Let $v,w$ be two distinct vertices of $\partial_{2R} Q_\infty$ and let $0<s<R$. Assume that the downward paths (and right downward paths) from
$v$ and $w$ to $\partial_{2s} Q_\infty$ are well defined, and furthermore that the \lmgs from $v$ and from $w$ coalesce (strictly) before reaching $\partial_{2s} Q_\infty$. Write 
$\mathcal{L}$ for the union of these two \lmgs up to their coalescence time, and $\CC_1$, \resp $\CC_2$, for the path from $v$ to $w$ along $\partial_{2R} Q_\infty$ in clockwise order,
\resp in counterclockwise order. Let $\RR_1$, \resp $\RR_2$, be the (closed) bounded region delimited by the closed path which is the union of $\mathcal{L}$ and $\CC_1$, \resp 
the union of $\mathcal{L}$ and $\CC_2$. Then either $\RR_1$ or $\RR_2$ does not intersect $\partial_{2s} Q_\infty$. Consider the first case for definiteness, so that
$\RR_1$ is contained in the annulus between $\partial_{2R} Q_\infty$ and $\partial_{2s} Q_\infty$. Then our
construction shows that the downward path starting from $w$ can only exit $\RR_1$ after hitting the \lmg started from $v$ (informally a downward path cannot cross a
\lmg ``from left to right'' --- for this property to hold, it is important that we started our downward paths with a ``horizontal'' step). Similarly, the right downward  path started from $v$ can only exit
$\RR_1$ after hitting the \lmg started from $w$. A planarity argument now shows that the downward path started from $w$ and the right downward  path started from $v$
intersect before exiting $\RR_1$, and therefore before hitting $\partial_{2s} Q_\infty$. Consequently, $v$ and $w$ are connected by a path in $\Ti{\CC(2R-2 \ceil{\eta R}, 2R)}$ 
whose length is bounded by the sum of the lengths of the downward path from $w$ to $\partial_{2s} Q_\infty$ and the right downward path from $v$ to $\partial_{2s} Q_\infty$

Let $\va$ be as given in \lemref{Lemma_control_length_DP}. Without loss of generality we can assume $\va \geq 1$. We apply the preceding
considerations with $s=s(R)=R-\ceil{\eta R/(3\va)}$. Using also  \lemref{Lemma_control_length_DP}, we get that, if $R$ is large enough, the following
holds with probability at least $1-\ve/2$: Whenever $v$ and $w$ are two vertices of $\partial_{2R} Q_\infty$ such that 
the \lmgs from $v$ and from $w$ coalesce before reaching $\partial_{2s} Q_\infty$, $v$ and $w$ are connected by a path of $\Ti{\CC(2s, 2R)}$
of length at most $2\alpha (R-s)\leq \eta R$.

\begin{figure}[h]
\begin{center}
\includegraphics[width=0.8\textwidth]{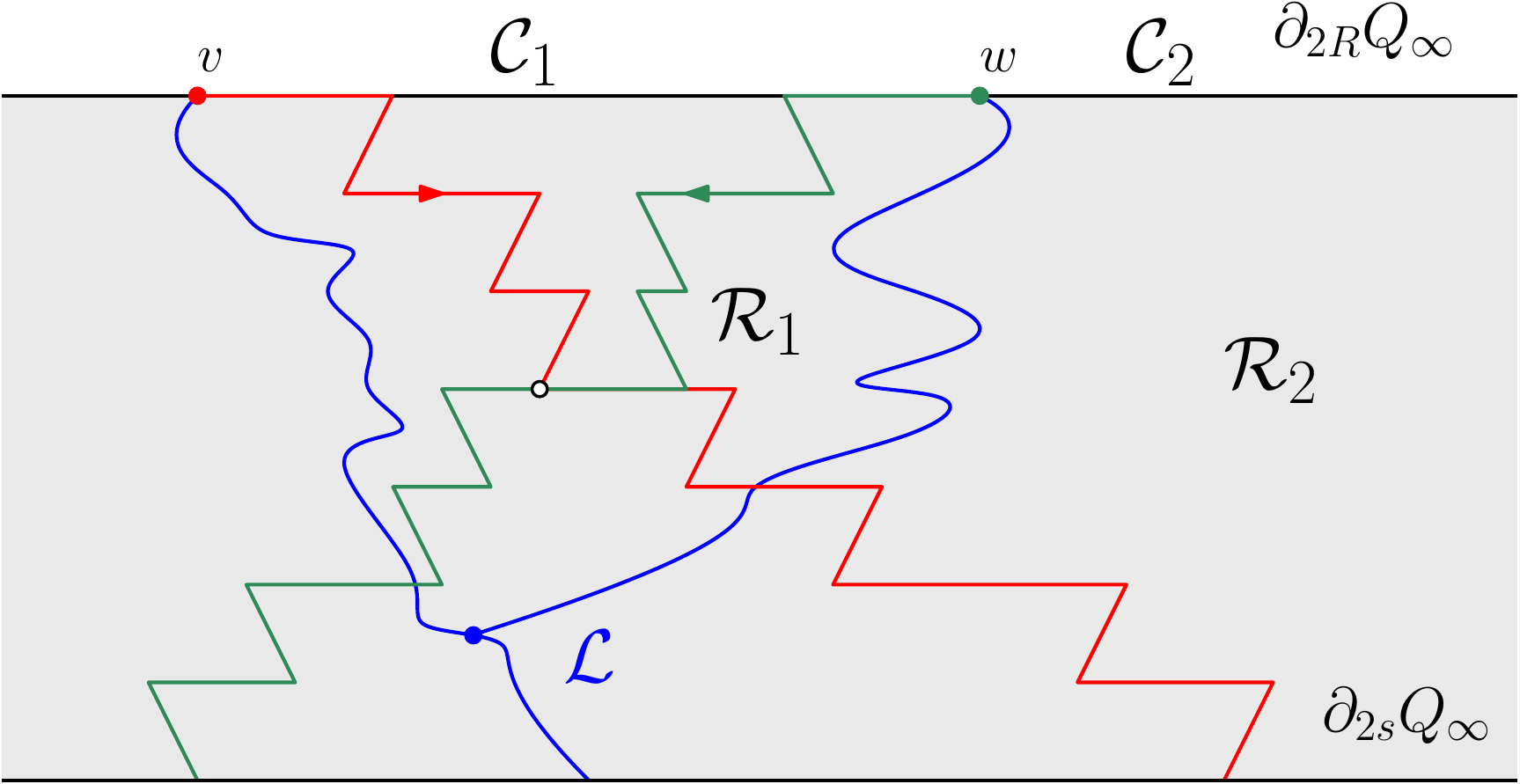}
\end{center}
\caption{In blue, the \lmgs (in $Q_\infty$) started from $v, w$ two vertices of $\partial_{2R} Q_\infty$, that coalesce before reaching $\partial_{2s} Q_\infty$. The downward path in $\Ti{Q_\infty}$ started at $w$ (in green) cannot cross the \lmg in $Q_\infty$ from $w$, and the right downward path in $\Ti{Q_\infty}$ started at $v$ (in red) cannot cross the \lmg in $Q_\infty$ from $v$: by planarity, they must cross (and thus intersect) before reaching $\partial_{2s} Q_\infty$. We get a path in $\Ti{Q_\infty}$ between $v$ and $w$ by first following the right downward path from $v$ up to the intersection (white point), then the (reverse) downward path from $w$.}
\label{Fig_right_DP}
\end{figure}

On the other hand, \propref{Prop_P17 coalescence primal} yields $A>0$ such that with probability at least $1-\ve/2$, 
any \lmgs starting from a vertex of $\partial_{2R} Q_\infty$ coalesces before reaching $\partial_{2s(R)} Q$ with one of the \lmgs started from $\ur_{\floor{kR^2/A}}$, $0\leq k \leq \floor{A H_{2R}/(2R)^2}$. 
By combining this with the preceding paragraph, we get that, with probability at least $1-\ve$, any vertex of $\partial_{2R} Q_\infty$ is connected to  one of these vertices $\ur_{\floor{kR^2/A}}$ by a path of $\Ti{\CC(2s, 2R)}$
of length at most $\eta R$.
This completes the proof.
\end{proof}

\subsection{Two technical lemmas}
\label{Sec_locality_distances}

We prove an estimate on the maximum degree of an inner face in the image of a large truncated hull of $Q_\infty$ by Tutte's correspondence. 
We then obtain a bound on the first-passage percolation distance in $\Ti{Q_\infty}$ between a vertex of $\partial_{2n} Q_\infty$ and the root vertex, for $n$ large enough.

\begin{lemma}
\label{Lemma_bound degree hulls of UIPQ}
For every integer $r\geq 1$, let $\vD^\circ\l(\Ti{ \Htr_{Q_\infty}(2r) } \r) $
denote the maximal degree of internal faces of $\Ti{ \Htr_{Q_\infty}(2r) }$.
\begin{equation*}
\P\l(\vD^\circ\l(\Ti{ \Htr_{Q_\infty}(2r) } \r) > 5\ln r \r) \ulim r \infty 0 .
\end{equation*}
\end{lemma}

\begin{proof}
For any map $M$, let $\vD(M)$ be the maximal degree of a face of $M$. 
Let $M_n$ be a uniform rooted map with $n$ edges. It follows from \cite[Theorem 3]{gao2000distribution} that
\begin{equation}
\label{Eq_bound on degree finite map}
\P(\vD(M_n) > \ln n) \ulim n \infty 0 .
\end{equation}
This result is actually stated for the maximal degree of a vertex of $M_n$ in \cite{gao2000distribution}, but  \eqref{Eq_bound on degree finite map} then follows by self-duality of $M_n$, see \cite[Lemma 3.2]{bettinelli2014scaling}.

By \propref{Prop_TBP 9} we can fix an integer $A>0$ large enough such that, for every $r$ large enough, we can couple $Q^\bullet_{\floor{Ar^4}}$ and $Q_\infty$ so that $\Hull_{Q^\bullet_{\floor{Ar^4}}}(2r) = \Hull_{Q_{\infty}}(2r)$ except on a set of probability at most $\ve/2$. Note that this equality of hulls also implies $\Htr_{Q^\bullet_{\floor{Ar^4}}}(2r) = \Htr_{Q_{\infty}}(2r)$ (the truncated hull is determined by the ``standard'' hull). Thus on the latter event,
\begin{equation*}
\vD^\circ\l( \Ti{\Htr_{Q_{\infty}}(2r)} \r) = \vD^\circ\l( \Ti{\Htr_{Q^\bullet_{\floor{Ar^4}}}(2r)} \r) \leq \vD\l( \Ti{Q_{\floor{Ar^4}}} \r) .
\end{equation*}
To get the last inequality, we observe that the degree of an internal face of $\TT(\Htr_{Q^\bullet_{\floor{Ar^4}}}(2r))$ is exactly the degree of the black vertex of $\Htr_{Q^\bullet_{\floor{Ar^4}}}(2r)$ that is contained in this face, and this vertex has the same degree in $\Htr_{Q^\bullet_{\floor{Ar^4}}}(2r)$ and in $Q_{\floor{Ar^4}}$. See \figref{Fig_Tutte_definition_hulls}.

We know that $\Ti{Q_{\floor{Ar^4}}}$ is distributed as $M_{\floor{Ar^4}}$. 
To complete the proof, we note that \eqref{Eq_bound on degree finite map} ensures that for $r$ large enough, $\vD\l( M_{\floor{Ar^4}} \r) \leq 5 \ln r$ with probability at least $1-\ve/2$.
\end{proof}

\begin{lemma}
\label{Lemma_for PT20 upper bound distance root small hull}
For every $\ve>0$, there exists $A>0$ such that for all $n$ large enough,
\begin{equation*}
\P\l( \max_{v \in \partial_{2n} Q_\infty} \dfpp^{\Ti{Q_\infty}}(\rho,v) \leq 2An \r) \geq 1- \ve.
\end{equation*}
\end{lemma}

\begin{proof}
Since weights are bounded, it is enough to prove this statement with $\dfpp^{\Ti{Q_\infty}}$ replaced by $\dgr^{\Ti{Q_\infty}}$.
Let $\eta \in (0,1)$. 
To simplify notation, in the remaining part of the proof, we write $H_n = \Htr_{Q_\infty} \l(2 \floor{\eta \frac{n}{\ln n}} \r)$ 
so that $\partial H_n = \partial_{2 \floor{\eta \frac{n}{\ln n}}} Q_\infty$.

By \lemref{Lemma_control_length_DP}, if $n$ is large enough, we have with probability at least $1-\ve/2$
\begin{equation*}
\dgr^{\Ti{Q_\infty}} \l(v, \partial H_n \r) \leq \va n,
\end{equation*}
for every $v \in \partial_{2n} Q_\infty$.
On the other hand, by \lemref{Lemma_bound degree hulls of UIPQ} the bound $\vD^\circ\l( \Ti{H_n} \r) \leq 5 \ln n$ holds with probability at least $1-\ve/2$ when is large. 
On this event, the simple bound 
$$\dgr^{\Ti{H_n}}(x,y) \leq \vD^\circ\l( \Ti{H_n} \r) \dgr^{H_n}(x,y),$$
valid for all $x,y \in V\l(\Ti{H_n} \r)$, ensures that for all $v \in \partial H_n$, $\dgr^{\Ti{Q_\infty}}(\rho, v) \leq 10\eta n$.
The statement of the lemma follows by combining these observations. \end{proof}

\subsection{Continuity properties of the Tutte correspondence}

Let us use the notation $M_n=\Ti{Q_n}$ for the image of $Q_n$ under Tutte's bijection.
Note that $M_n$ is uniformly distributed over (rooted) planar maps with $n$
edges. For every $r>0$, we write $\mathcal{B}_{M_n}(r)$ for the (closed) metric ball 
of radius $r$ centered at the root vertex $\rho_n$ in $V(M_n)$. We may view this ball (\resp its complement) 
as a graph by keeping only those edges incident to a face of the ball (\resp of its complement).

\begin{proposition}
\label{Cor_L28 locality}
For every $\ve, \eta>0$, there exists $\vd>0$ s.t. for all $n$ large enough,
\begin{equation*}
\P\l( \sup_{x \in V(M_n) , \ \dgr^{Q_n}(\rho_n,x) \leq \vd n^{1/4}} \dgr^{M_n}(\rho_n,x) \leq \ve n^{1/4} \r) \geq 1-\eta .
\end{equation*}
\end{proposition}

\begin{proof} 
We again use the convergence in distribution to the Brownian map $(\bmap,D^*)$. It follows from \eqref{Eq_hull dans boules version bmap}
that we can choose a constant $K>2$ large enough so that, for every $\ve>0$, the probability of the event where
at least two connected components of the complement of the ball $\mathcal{B}_\bmap(\ve)$ intersect the complement of $\mathcal{B}_\bmap(K\ve/2)$
is bounded above by $\eta/4$. 
From \cite[Corollary 1.2]{bettinelli2014scaling}, we know that the random compact
metric spaces $(V(M_n),(\frac{9}{8n})^{1/4}\dgr^{M_n})$ converge in distribution in the Gromov-Hausdorff sense
to the Brownian map. Although this is not stated in \cite{bettinelli2014scaling}, it follows from the proof that
this convergence also holds in the pointed Gromov-Hausdorff sense, if $M_n$
is pointed at the root vertex $\rho_n$ (and $\bmap$ is pointed at $x_*$). From this pointed convergence,
and the properties of the Brownian map stated above, we can now deduce that, for every $\ve>0$,
for all sufficiently large $n$, the probability that at least two components of the complement of 
the ball $\mathcal{B}_{M_n}(\ve n^{1/4})$ intersect the complement of the ball $\mathcal{B}_{M_n}(K\ve n^{1/4})$
is bounded above by $\eta/2$. 

Let us fix $\ve>0$ and set $\beta=\ve/K$. We can assume that $\ve$ is so small that $\P(\dgr^{Q_n}(\rho_n,\partial_n)>4K\beta n^{1/4})>1-\eta/4$. 
Using the coupling between $Q_n^\bullet$ and the \UIPQ $Q_\infty$ in Proposition \ref{Prop_TBP 9}, we get
from Lemma \ref{Lemma_for PT20 upper bound distance root small hull} that there exists $\mu>0$ such that
$$\max_{v\in \partial_{2\floor{\mu n^{1/4}}}Q_n^\bullet} \dgr^{M_n}(\rho_n,v)<\frac{\beta}{2} n^{1/4}$$
with probability at least $1-\eta/4$. Recall that the edges of the cycle  $\partial_{2\floor{\mu n^{1/4}}}Q^\bullet_n$
are also edges of $M_n$. Argue on the event where both the bound in the last display holds and $\dgr^{Q_n}(\rho_n,\partial_n)>4K\beta n^{1/4}$. Then, except 
on an event of probability at most $\eta/2$, at most one of the two components bounded by the cycle
$\partial_{2\floor{\mu n^{1/4}}}Q^\bullet_n$ can intersect the complement of the ball  $\mathcal{B}_{M_n}(K\beta n^{1/4})$, and this must be the component that
contains the distinguished vertex $\partial_n$. We conclude that, except on a set of probability at most $\eta$, the
truncated hull $\Htr_{Q^\bullet_n}(2\floor{\mu n^{1/4}})$ does not intersect the complement of the ball
$\mathcal{B}_{M_n}(K\beta n^{1/4})$, and in particular the bound 
$$\sup_{x \in V(M_n) , \ \dgr^{Q_n}(\rho_n,x) \leq 2\floor{\mu n^{1/4}}}\dgr^{M_n}(\rho_n,x) \leq K\beta n^{1/4}=\ve n^{1/4}$$
holds with probability at least $1-\eta$. 
This completes the proof. \end{proof}

\section{Main results for general maps}
\label{Sec_Main results for generic maps}

\subsection{Subadditivity in the \textsc{LHPQ}}
\label{SecT_Subadditivity in the LHPQ}

Recall from the beginning of Section \ref{technical-general} that we can apply Tutte's bijection to the \LHPQ $\LL$, and that
$\Lhm$ denotes the resulting infinite map. 
We observe that every edge of the form $((i,-2r), (i+1, -2r))$ for $r \geq 0$ and $i \in \Z$ appear in $\Lhm$ because vertices of the type $(i,-2r)$ are white, and every edge of the preceding form is a diagonal of some quadrangle. It follows that we can define slices of $\Lhm$ for even coordinates in a way similar to the case of the \LHPQ: for even $j \leq j' \leq 0$, $\Lhm_j^{j'}$ is the part of $\Lhm$ contained in $\R\times[j,j']$. Furthermore, disjoint slices are independent.

We write $\dfpp^\Lhm(x,y)$, for $x,y \in V(\Lhm)$,  for the first-passage percolation distance on $V(\Lhm)$ (recall that
weights belong to $[1,\kappa]$).

\begin{proposition}
\label{Prop_PT18 - ergodic}
There exists a constant $\bc_T \in [\frac{1}{2}, \infty)$ such that
\begin{equation}
\label{Eq_P18_ergo}
(2r)^{-1} \dfpp^\Lhm(\rho, \partial_{-2r} \LL) \overset{\as}{\ulim{r}{\infty}} \bc_T .
\end{equation}
\end{proposition}

\begin{proof}
The proof uses the same subadditivity argument as that of Proposition \ref{Prop_P18 - ergodic}, with the minor difference
that we restrict our attention to even heights. 

We first note that, for $x,y \in V(\Lhm)$, $\dfpp^\Lhm(x,y) \geq \dgr^\Lhm(x,y) \geq \frac 1 2 \dgr^\LL(x,y)$, so that the limit $\bc_T$ in \eqref{Eq_P18_ergo}, if it exists, has to be greater than 
or equal to $1/2$. 
 The only new part is that we have to check that $\E [ \dgr^\Lhm(\rho, \partial_{-2} \LL)] < \infty$.

As we already noticed in \secref{Section_Downward_paths}, the downward path started from the root
is well defined in the \LHPQ, and provides an upper bound on $\dgr^\Lhm(\rho, \partial_{-2} \LL)$. 
The number of steps of this downward path before it reaches a vertex of $\partial_{-2} \LL$ is distributed as $G+2$, where $G$ is a geometric random variable, thus has a finite expectation, giving the desired result.
\end{proof}

\subsection{Distance through a thin annulus}
\label{SecT_Distance through a thin annulus}

We use the notation $M_\infty=\Ti{Q_\infty}$, so that $M_\infty$ may be called the uniform infinite planar map or \UIPM.

\begin{proposition}
\label{Prop_PT19 thin annuli}
Let $\ve \in (0,1)$ and $\vd >0$. For every $\eta>0$ small enough, for all sufficiently large $n$, the property 
\begin{equation}
\label{Eq_Tcontrol of distances through a thin annulus}
2(1-\ve) \bc_T \eta n \leq  \dfpp^{M_\infty}(v , \partial_{2(n-\floor{\eta n})} Q_\infty)  \leq 2(1+\ve) \bc_T \eta n,
\end{equation}
holds for every $v \in \partial_{2n} Q_\infty$, with probability at least $1-\vd$.
\end{proposition}

\begin{proof}
The proof is patterned after that of \propref{Prop_P19 thin annuli} using Proposition \ref{Prop_PT18 - ergodic} instead of Proposition \ref{Prop_P18 - ergodic}. The following minor adaptations are required. 

Let $u^{(n)}_0$ be a uniformly distributed vertex of $\partial_{2n} Q_\infty$. 
Then \lemref{Lemma_control_length_DP} ensures that we have with high probability,
\begin{equation*}
\dfpp^{M_\infty} \l(u^{(n)}_0, \partial_{2(n-\floor{\eta n})} Q_\infty\r) \leq 2\va\vk\floor{\eta n} .
\end{equation*}
Let $\GG^{(n)}_0$ be defined as in the proof of \propref{Prop_P19 thin annuli} (with $n$ replaced by $2n$). 
As in the latter proof, we know with high probability that  the length of the minimal path (in $Q_\infty$)  between $u^{(n)}_0$ and the lateral boundary of $\GG^{(n)}_0$ that stays in $\wt{\mathcal{H}}^{\mathrm{tr}}_{Q_\infty}(2n)$ 
 is 
bounded below by $cn$ with some constant $c>0$. Trivially, two vertices of ${\mathcal{H}}^{\mathrm{tr}}_{Q_\infty}(2n)$  that are linked by an edge of 
$\Ti{Q_\infty}$ are also connected by a $Q_\infty$-path of length two in $\wt{\mathcal{H}}^{\mathrm{tr}}_{Q_\infty}(2n)$. Taking $\eta$ smaller if necessary, it follows that the $\dfpp^{M_\infty}$-shortest path between $u^{(n)}_0$ and $\partial_{2(n-\floor{\eta n})} Q_\infty$ that stays in ${\mathcal{H}}^{\mathrm{tr}}_{Q_\infty}(2n)$ does not leave $\GG^{(n)}_0$ on an event of high probability. We can then use the same density arguments as in the proof of \propref{Prop_P19 thin annuli}.

In the last step of the proof, we need to verify that it suffices to obtain
\eqref{Eq_Tcontrol of distances through a thin annulus} for a bounded number of vertices $v$ of $\partial_{2n} Q_\infty$. The argument is the same as
in the proof of \propref{Prop_P19 thin annuli}, but we use
\corref{Cor_P17 coalescence Tutte} in place of \propref{Prop_P17 coalescence primal}.
\end{proof}

\subsection{Distance from the boundary of a hull to its center}
\label{SecT_Distance from the boundary of a hull to its center}

The next proposition is analogous to \propref{Prop_P20 distances root - boundary of the hull, UIPQ}.

\begin{proposition}
\label{Prop_PT20 distances root - boundary of the hull, UIPQ}
For every $\ve \in (0,1)$,
$$ \P\l( 2(\bc_T - \ve) n \leq \dfpp^{M_\infty}(\rho, v) \leq 2(\bc_T + \ve) n  \  \text{  for every } v \text{ in } \partial_{2n} Q_\infty \r) \ulim{n}{\infty} 1  . $$
\end{proposition}

\begin{proof}
The proof is very similar to that of \propref{Prop_P20 distances root - boundary of the hull, UIPQ}. Consider the annuli $\CC(2n_{k+1}, 2n_k)$ for every $0 \leq k < q$, where $n_0 = n$ and $n_{k+1} = n_k - \floor{\eta n_k}$, and $q$ is as defined in \propref{Prop_P20 distances root - boundary of the hull, UIPQ}. 
By \propref{Prop_PT19 thin annuli}, we can find $\eta>0$ such that ``most'' of these annuli will satisfy the analog of \eqref{Eq_Tcontrol of distances through a thin annulus}, except 
possibly on a set of probability at most $\ve$. 
We then use \lemref{Lemma_control_length_DP} to bound the $\dfpp^{M_\infty}$-distance through the annuli where \eqref{Eq_Tcontrol of distances through a thin annulus} does not hold, and \lemref{Lemma_for PT20 upper bound distance root small hull} to control the $\dfpp^{M_\infty}$-distance between $\rho$ and $\partial_{2n_q} Q_\infty$.
\end{proof}

\subsection{Distance between two uniformly sampled points in finite maps}
\label{SecT_Distance between two uniformly sampled points in finite maps}

Recall that $M_n=\Ti{Q_n}$ in such a way that $V(M_n)$ can be viewed as the subset of $V(Q_n)$ consisting
of the ``white'' vertices. Also recall that $\#V(Q_n)=n+2$.
We observe that
$$\frac{\#V(M_n)}{\#V(Q_n)} \ulim n \infty \frac{1}{2}$$
in probability (see e.g.
the proof of Proposition 3.1 in \cite{bettinelli2014scaling}).
Then, since conditionally on $Q_n$ the distinguished vertex $\partial_n$ 
is uniformly distributed over $V(Q_n)$, we have also
$$\P(\partial_n\in V(M_n))\ulim n \infty \frac{1}{2}.$$

We also observe that the result of Proposition \ref{Cor_L28 locality} remains valid if we
replace the root vertex $\rho_n$ by $\partial_n$. More precisely, for every $\ve,\eta>0$, we can find
$\beta>0$ such that, for all $n$ large enough, we have
\begin{equation}
\label{conti-Tutte}
\P\l( \sup_{x \in V(M_n) , \ \dgr^{Q_n}(\partial_n,x) \leq \beta n^{1/4}} \dgr^{M_n}(\partial_n,x) \leq \ve n^{1/4}
\Bigg|\, \partial_n\in V(M_n)\r) \geq 1-\eta .
\end{equation}
This follows from the invariance of $Q_n$ under uniform re-rooting, by an argument very similar to the one used in
the proof of Theorem \ref{Th_T1 controle distances fpp et gr dans quad finies}.

\begin{proposition}
\label{Prop_PT21 distances root - uniform point, Qn}
For every $\vg \in (0,1)$,
\begin{equation*}
\P\l( | \dfpp^{M_n}(\rho_n, \partial_n) - \bc_T \dgr^{Q_n}(\rho_n, \partial_n) | > \vg n^{1/4} \,\Big|\, \partial_n\in V(M_n)\r) \ulim n \infty 0 .
\end{equation*}
\end{proposition}

\begin{proof}

The proof is based on the same ingredients as that of \propref{Prop_P21 distances root - uniform point, Qn}, but
we use Proposition \ref{Prop_PT20 distances root - boundary of the hull, UIPQ}, or rather \eqref{conti-Tutte},  instead of Proposition \ref{Prop_P20 distances root - boundary of the hull, UIPQ}. 
We argue on the event where $ \partial_n\in V(M_n)$. Let $\vg>0$, $\eta>0$ and choose $\beta>0$ so
that \eqref{conti-Tutte} holds with $\ve=\vg$.

Fix $\vd >0$ small enough so that $3\vd^2<\beta$ and the event $H_{n,\vd}$
of \lemref{Lemma_technical slicing} has probability at least $1-\eta$ when $n$ is large. For integers $j,l\geq 1$,  let $H_{n,\vd}^{j,l}$ be defined as in \lemref{Lemma_technical slicing}. 
If $n$ is large, the $\dgr^{Q_n}$-distance between $\partial_n$ and $\partial_{2\floor{\va_j n^{1/4}/2}} Q^\bullet_n$ is smaller than $\beta n^{1/4}$ on $H_{n,\vd}^{j,l}$.
Thus, on the intersection of  $H_{n,\vd}^{j,l}$ with the event 
considered in \eqref{conti-Tutte} (with $\ve=\vg$),
the $\dgr^{M_n}$-distance between $\partial_n$ and $\partial_{2\floor{\va_j n^{1/4}/2}} Q^\bullet_n$ is smaller that $\vg n^{1/4}$. 

On the other hand, from \propref{Prop_PT20 distances root - boundary of the hull, UIPQ}
and using \eqref{Eq_density for hulls in finite and infinite case} as in the proof of Proposition \ref{Prop_P21 distances root - uniform point, Qn}, 
we get that, outside of a set of probability going to $0$ as $n \to \infty$, the $\dfpp^{M_n}$-distance between any vertex of $\partial_{2\floor{\va_j n^{1/4}/2}} Q^\bullet_n$ and $\rho_n$ is close to $\bc_T \va_j n^{1/4}$, up to an error term bounded by $\gamma n^{1/4}$.

Finally, on the intersection of $H_{n,\vd}$ with $\{\partial_n\in V(M_n)\}$  and with the event considered
in \eqref{conti-Tutte}, we have
\begin{equation*}
 | d(\rho_n, \partial_n) - \bc_T \dgr^{Q^\bullet_n}(\rho_n, \partial_n) | \leq (1+\kappa)\vg n^{1/4},
\end{equation*}
except on a set of probability tending to $0$ as $n \to \infty$.
Using \lemref{Lemma_technical slicing} we obtain that the latter intersection has probability larger than $\P(\partial_n\in V(M_n))-2\eta$ for all $n$ large enough. This completes the proof. \end{proof}

\subsection{Distances between any pair of points of finite maps}
\label{SecT_Distances between any pair of points of finite maps (Theorem 1)}

The next statement gives both Theorem \ref{Th2} and the part of Theorem \ref{Th1} concerning general planar maps.

\begin{theorem}
\label{dist-fpp-gr}
For every $\ve>0$, we have
\begin{equation*}
\P\l( \sup_{x,y \in V(M_n)} \l| \dfpp^{M_n}(x,y) - \bc_T \dgr^{Q_n}(x,y) \r| > \ve n^{1/4} \r) \ulim{n}{\infty} 0
\end{equation*}
If all weights are equal to $1$ (that is, $\dfpp^{M_n}=\dgr^{M_n}$), we have $\bc_T=1$.
\end{theorem}

Before we prove Theorem \ref{dist-fpp-gr}, we state and prove 
a lemma.

\begin{lemma}
\label{covering-lemma}
Let $\eta\in(0,1)$, and, for every $n\geq1$, conditionally on $M_n$, let $\partial_n^1,\partial_n^2,\ldots$ be independent 
random vertices uniformly distributed over $V(M_n)$. Then, for every $\ve>0$, we can find an integer $N\geq 1$
such that, for every sufficiently large $n$, we have
$$\P\l(\max_{v\in V(M_n)}\,\l(\min_{1\leq \ell\leq N} \dgr^{M_n}(\partial^\ell_n,v)\r) <\ve\,n^{1/4}\r) >1-\eta.$$
\end{lemma}

\begin{proof} We first note that the statement would follow if we knew the convergence in the
Gromov-Hausdorff-Prokhorov sense of $(V(M_n),(9/8n)^{1/4}\dgr^{M_n})$ equipped with the
uniform probability measure to the Brownian map --- cf. the analogous statement for $Q_n$
used in the proof of Theorem \ref{Th_T1 controle distances fpp et gr dans quad finies}. 
Unfortunately, \cite{bettinelli2014scaling} does not give the Gromov-Hausdorff-Prokhorov convergence,
and so we will provide a direct proof, which still relies much on the arguments of \cite{bettinelli2014scaling}.
We start by observing that \cite[Proposition 3.1]{bettinelli2014scaling} allows us to replace $M_n$ by a random
{\em pointed} planar map $M_n^\bullet$ which is uniformly distributed over pointed 
planar maps with $n$ edges (this replacement needs to be justified because, in contrast with the
case of quadrangulations, forgetting the distinguished vertex of $M_n^\bullet$ does not give a
map distributed as $M_n$). Then, as in \cite[Section 4]{bettinelli2014scaling}, we can construct a finite sequence
$\wt v^n_0,\wt v^n_1,\ldots,\wt v^n_{2n}$ such that every vertex $v$ of $M_n^\bullet$ appears at least once in this sequence,
and, if we set $\wt D_n(i,j)=\dgr^{M_n^\bullet}(\wt v^n_i,\wt v^n_i)$ for $i,j\in\{0,1,\ldots,2n\}$ and interpolate linearly
to get a function $\wt D(s,t)$ defined on $[0,2n]^2$, we have
$$\l( \l(\frac{9}{8n}\r)^{1/4} \wt D_n(2ns,2nt)\r)_{0\leq s,t\leq 1} \ulim{n}{\infty} (D^*(s,t))_{0\leq s,t\leq 1},$$
in distribution in the space of continuous functions on $[0,1]^2$. Here $D^*(s,t)$ is the random
pseudo-metric on $[0,1]^2$ that defines the Brownian map. Since $D^*$ vanishes on the diagonal, we can fix an integer
$A\geq1$ such that, writing $\vd=1/A$, the property
$$D^*(s,s')< \frac{\ve}{4}\;,\qquad \forall s,s'\in [(k-1)\vd,k\vd],\;\forall k\in\{1,\ldots,A\}$$
holds with probability greater than $1-\eta/2$. Using the preceding convergence, it follows that, for $n$ large enough,
the property
$$\dgr^{M_n^\bullet}(\wt v^n_i,\wt v^n_j)< \frac{\ve}{2}\,n^{1/4}\;,\qquad \forall i,j\in [2n(k-1)\vd,2nk\vd]\cap\Z,\;\forall k\in\{1,\ldots,A\}$$
also holds with probability greater than $1-\eta/2$. We claim that we can find an integer $N$ large enough so that,
for every $n$ large enough,
with probability greater than $1-\eta/2$, there exists for each $k\in\{1,\ldots,A\}$ an
index $\ell\in\{1,\ldots,N\}$ and an integer $i\in [2n(k-1)\vd,2nk\vd]$ such that
$$\dgr^{M_n^\bullet}(\partial_n^\ell,\wt v^n_i)\leq \frac{\ve}{2}\,n^{1/4}.$$
If we combine the claim with the preceding considerations, we get that, with probability at least $1-\eta$,
any vertex of $M_n^\bullet$ is at distance smaller than $\ve n^{1/4}$ from one of the vertices $\partial_n^\ell$,
$\ell\in\{1,\ldots,N\}$, which was the desired result.

It remains to prove our claim. To this end, we need more information about the sequence $\wt v^n_i$
(we refer to \cite{bettinelli2014scaling} for more details). Via the Ambj\o rn-Budd bijection, the pointed planar map $M^\bullet_n$
is associated with a (uniformly distributed) pointed quadrangulation $Q^\bullet_n$ with $n$ faces,
in such a way that $V(M^\bullet_n)$ is identified to a subset of $V(Q^\bullet_n)$, and in
particular $\#V(M^\bullet_n) \leq \#V(Q^\bullet_n)=n+2$. In the \CVS, $V(Q^\bullet_n)$
corresponds to a labeled tree $T_n$, and the contour sequence $v^n_0,v^n_1,\ldots,v^n_{2n}$ of the tree $T_n$
(defined as in the proof of \lemref{Lemma_technical slicing})
can also be viewed as a sequence of vertices of $Q^\bullet_n$. Then, for every 
$i\in\{1,\ldots,2n\}$, $v^n_i$ and $\wt v^n_i$ are linked by an edge of $Q^\bullet_n$
(see \cite{bettinelli2014scaling}). Moreover,
in the case when $v^n_i\in V(M^\bullet_n) $, one has
\begin{equation}
\label{bound-dgr-v}
\dgr^{M^\bullet_n}(v^n_i,\wt v^n_i) \leq \Delta(M^\bullet_n).
\end{equation}
This bound follows directly from the construction of the Ambj\o rn-Budd bijection
(the point is that any edge of $Q^\bullet_n$ is contained in a face of $M^\bullet_n$).
Recalling
\eqref{Eq_bound on degree finite map}, and using again \cite[Proposition 3.1]{bettinelli2014scaling}, we know that
we have $\Delta(M^\bullet_n)<\frac{\ve}{2} n^{1/4}$ with probability greater than $1-\eta/8$. For every 
integer $p\in\{0,1,\ldots,n\}$, let $\mathcal{N}^{(n)}_p$ be the number of distinct vertices $v^n_i$ with $i\in\{0,1,\ldots,p\}$
that belong to 
$V(M^\bullet_n)$. Then, from the 
end of \cite[Section 5]{bettinelli2014scaling}, we have for every $t\in[0,1]$,
$$\frac{1}{n}\,\mathcal{N}^{(n)}_{\floor{2nt}} \ulim{n}{\infty}  \frac{t}{2}$$
in probability. It follows that, for $n$ large enough, we have
\begin{equation}
\label{event-good}
\#\{v^n_i: i\in[2n(k-1)\vd,2nk\vd] \hbox{ and }v^n_i\in V(M^\bullet_n)\}\geq \mathcal{N}^{(n)}_{\floor{2nk\vd}}  - \mathcal{N}^{(n)}_{\lceil 2n(k-1)\vd\rceil}\geq \frac{\vd}{4}\,n\;,\  \forall k\in\{0,1,\ldots,A\},
\end{equation}
with probability at least $1-\eta/8$. We can choose $N$ large enough so that, on the event \eqref{event-good}, the conditional 
probability given $M_n^\bullet$ that each set $\{v^n_i : i\in[2n(k-1)\vd,2nk\vd] \hbox{ and }v^n_i\in M^\bullet_n\}$, for $1\leq k\leq A$, contains at least one of the vertices $\partial_n^\ell$,
$1\leq \ell\leq N$, is greater than $1-\eta/4$. Summarizing the preceding considerations and using \eqref{bound-dgr-v}, we get that, with probability at least
$1-\eta/2$, for every $k\in\{1,\ldots,A\}$, there exist an index $\ell\in\{1,\ldots,N\}$ and an integer $i\in [2n(k-1)\vd,2nk\vd] $ such that
$\partial_n^\ell=v^n_i$ and $\dgr^{M^\bullet_n}(\partial_n^\ell,\wt v^n_i)= \dgr^{M^\bullet_n}(v^n_i,\wt v^n_i)<\frac{\ve}{2} n^{1/4}$. This completes
the proof of the claim and of the lemma.
\end{proof}

\begin{proof}[Proof of Theorem \ref{dist-fpp-gr}]
By the same re-rooting invariance argument as in the proof
of Theorem \ref{Th_T1 controle distances fpp et gr dans quad finies}, the statement of
Proposition \ref{Prop_PT21 distances root - uniform point, Qn} remains valid if
the pair $(\rho_n,\partial_n)$ is replaced by $(\partial'_n,\partial''_n)$,
where, conditionally on $Q_n$, $\partial'_n$
and $\partial''_n$ are independent and uniformly distributed over $V(Q_n)$: more precisely, we have, for every $\ve>0$,
\begin{equation*}
\P\l( | \dfpp^{M_n}(\partial'_n, \partial''_n) - \bc_T \dgr^{Q_n}(\partial'_n, \partial''_n) | > \ve n^{1/4} \,\Big|\, \partial'_n\in V(M_n),
\partial''_n\in V(M_n)\r) \ulim n \infty 0 .
\end{equation*}
Let us fix $\ve>0$ and $\eta>0$. Thanks to Lemma \ref{covering-lemma}, we can fix an integer $N$ large enough 
so that, if $\partial^1_n,\partial^2_n,\ldots,\partial^N_n$ are independent and uniformly distributed over $V(M_n)$, then, with probability
at least $1-\eta$,
the metric balls of radius $\ve n^{1/4}$ in $(V(M_n),\dgr^{M_n})$ centered at $\partial^1_n,\ldots,\partial^N_n$ cover $V(M_n)$.
Let us call $\mathcal{H}_n$ the event where this covering property holds.

On the other hand, consider the event
$$\mathcal{K}_n:=\{ | \dfpp^{M_n}(\partial^i_n, \partial^j_n) - \bc_T \dgr^{Q_n}(\partial^i_n, \partial^j_n) | \leq \ve n^{1/4},\;
\forall i,j\in\{1,\ldots,N\}\}.$$
By the first observation of the proof, we have also $\P(\mathcal{K}_n)\geq 1-\eta$ for $n$ large enough. 

For $n$ large, the event $\mathcal{H}_n\cap\mathcal{K}_n$
has probability at least $1-2\eta$. Let us argue on this event in the remaining part of the proof. Let $x,y\in V(M_n)$, we can find $i,j\in\{1,\ldots,N\}$
such that $\dgr^{M_n}(\partial^i_n,x)\leq \ve n^{1/4}$ and $\dgr^{M_n}(\partial^j_n,y)\leq \ve n^{1/4}$. Note that this
implies $\dgr^{Q_n}(\partial^i_n,x)\leq 2\ve n^{1/4}$ and $\dgr^{Q_n}(\partial^j_n,y)\leq 2\ve n^{1/4}$. It follows that
we have
$$|\dfpp^{M_n}(x,y) -\dfpp^{M_n}(\partial^i_n,\partial^j_n)|\leq \dfpp^{M_n}(\partial^i_n,x)+ \dfpp^{M_n}(\partial^j_n,y)\leq 2\kappa \ve n^{1/4}$$
and 
$$|\dgr^{Q_n}(x,y) - \dgr^{Q_n}(\partial^i_n,\partial^j_n)| \leq 4\ve n^{1/4}.$$
Hence, from the definition of $\mathcal{K}_n$,
$$| \dfpp^{M_n}(x,y) - \bc_T \dgr^{Q_n}(x,y)|\leq (1+4\bc_T+2\kappa)\ve \,n^{1/4}.$$
This completes the proof of the first assertion.

As for the second one, we observe that the first assertion, together with the known convergence 
of rescaled quadrangulations to the Brownian map, implies that
$$\l( V(M_n), \pfrac{9}{8n}^{1/4} \dfpp^{M_n} \r) \ulimd n \infty (\bmap, \bc_T D^*) $$
in distribution in the Gromov-Hausdorff sense. In the case where all weights are equal to $1$,
comparing this convergence with \cite[Corollary 1.2]{bettinelli2014scaling} gives $c_T=1$. 
\end{proof}

\subsection{Distances in the UIPM}
\label{SecT_Hulls in the UIPM are close}

Recall that $M_\infty = \Ti{Q_\infty}$ is the \UIPM. 

\begin{theorem}
\label{Th_T4 hulls in the UIPQ are close for dfpp and dgr}
Let $\ve \in (0,1)$. We have
\begin{equation}
\label{Eq_distances are close in hulls of Qinfty2}
\lim_{r \to \infty} \P\l( \sup_{x,y \in V(M_\infty),\,\dgr^{M_\infty}(\rho,x)\vee\dgr^{M_\infty}(\rho,y)\leq r} \l| \dfpp^{M_\infty}(x,y) - \bc_T \dgr^{M_\infty}(x,y) \r| > \ve r \r) = 0,
\end{equation}
and 
\begin{equation}
\label{Tutte-UIPM}
\lim_{r \to \infty} \P\l( \sup_{x,y \in V(M_\infty),\,\dgr^{M_\infty}(\rho,x)\vee\dgr^{M_\infty}(\rho,y)\leq r} \l| \dgr^{M_\infty}(x,y) - \dgr^{Q_\infty}(x,y) \r| > \ve r \r) = 0.
\end{equation}

\end{theorem}

We only sketch the proof, as it is very similar to that of Theorem \ref{Th_T2 hulls in the UIPQ are close for dfpp and dgr}.
Since $\dgr^{M_\infty}(x,y)\geq \frac{1}{2} \dgr^{Q_\infty}(x,y)$ for every $x,y\in V(M_\infty)$,
the condition $\dgr^{M_\infty}(\rho,x)\leq r$ implies $\dgr^{Q_\infty}(\rho,x)\leq 2r$. By the same argument, 
we can find a constant $K$ large enough so that, for every $r\geq 1$ and for every $x,y\in V(M_\infty)$ such that 
$\dgr^{M_\infty}(\rho,x)\leq r$ and $\dgr^{M_\infty}(\rho,y)\leq r$, the quantities $\dfpp^{M_\infty}(x,y)$, $\dgr^{M_\infty}(x,y)$
and $ \dgr^{Q_\infty}(x,y) $ are determined by the hull $B^\bullet_{Q_\infty}(Kr)$ (and of course weights on edges in the case
of $\dfpp^{M_\infty}$). We then use Proposition \ref{Prop_TBP 9}
that allows us to find a large constant $C$ such that the hulls $\Hull_{Q^\bullet_{\floor{C(Kr)^4}}}(Kr)$ and $\Hull_{{Q_\infty}}(Kr)$
are equal with probability close to $1$. We conclude by using Theorem \ref{dist-fpp-gr}.

Theorem \ref{Th3} stated in the introduction follows from Theorem \ref{Th_T2 hulls in the UIPQ are close for dfpp and dgr} and
Theorem \ref{Th_T4 hulls in the UIPQ are close for dfpp and dgr}.

\bibliographystyle{plain}
\bibliography{bibliographie_fpp.bib}

\end{document}